\documentclass[11pt]{amsart}

\usepackage{amsmath,amssymb}
\usepackage{mathrsfs} 
\usepackage{enumerate}
\usepackage{enumitem}
\usepackage{appendix}
\usepackage{wrapfig}
\usepackage{comment}
\usepackage{float}
\makeatletter
\renewcommand*{\fps@figure}{htb}
\makeatother

\usepackage[all]{xy}
\usepackage{tikz-cd}
\tikzcdset{arrow style=tikz,diagrams={>=stealth}}

\tikzstyle{stuff_fill}=[rectangle,fill=white,minimum size=1em]
\tikzstyle{stuff_fillc}=[circle,fill=white,minimum size=1em]

\usepackage{tablefootnote}

\usepackage{pgf,tikz}
\usetikzlibrary{arrows}
\usepackage[all]{xy}
\usepackage{appendix}
\allowdisplaybreaks 

\makeatletter
\renewcommand\subsection{\@startsection{subsection}{2}%
  \z@{-.5\linespacing\@plus-.7\linespacing}{.5\linespacing}%
  {\normalfont\bfseries}}
\renewcommand\subsubsection{\@startsection{subsubsection}{3}%
  \z@{.5\linespacing\@plus.7\linespacing}{-.5em}%
  {\normalfont\bfseries}}
\makeatother

\makeatletter
\@namedef{subjclassname@2020}{\textup{2020} Mathematics Subject Classification}
\makeatother

\usepackage{array,booktabs,arydshln,xcolor}

\usetikzlibrary {positioning}
\usetikzlibrary{arrows}
\usetikzlibrary{positioning}
\usepackage{multirow}
\usepackage{tikz-cd}

\usepackage{multicol}
\usepackage{crossreftools}

\usepackage{geometry} 
\usepackage{tikz-cd}
\tikzcdset{arrow style=tikz,diagrams={>=stealth}}

\tikzstyle{stuff_fill}=[rectangle,fill=white,minimum size=1em]

\RequirePackage{hyperref}
\hypersetup{pdfpagemode={UseOutlines},
bookmarksopen=true,
bookmarksopenlevel=0,
hypertexnames=false,
colorlinks=true, 
citecolor=teal, 
linkcolor=blue, 
urlcolor=magenta, 
pdfstartview={FitV},
unicode,
breaklinks=true,
}

\usepackage[capitalize,noabbrev]{cleveref}
\hypersetup{pdfpagemode={UseOutlines},
	bookmarksopen=true,
	bookmarksopenlevel=0,
	hypertexnames=true,
	colorlinks=true, 
	citecolor=teal, 
	linkcolor=blue, 
	urlcolor=magenta, 
	pdfstartview={FitV},
	unicode,
	breaklinks=true,
}

\newcommand \R {\mathbf R}
\newcommand \Z {\mathbf Z}

\newcommand \N {\mathbf N}
\newcommand \calp {\mathcal{P}}
\newcommand \lesap{\lesssim_{\alpha,p}}

\newtheorem{theorem}{Theorem}[section]
\newtheorem{theoremintro}{Theorem}
\newtheorem{corollaryintro}[theoremintro]{Corollary}
\newtheorem{lemma}[theorem]{Lemma}
\newtheorem{proposition}[theorem]{Proposition}
\newtheorem{corollary}[theorem]{Corollary}
\newtheorem{conjecture}[theorem]{Conjecture}

\newtheorem{addendum}[theorem]{Addendum}

\newtheorem*{theorem*}{Theorem}
\newtheorem*{ques*}{Question}
\newtheorem*{prop*}{Proposition}

\theoremstyle{definition}
\newtheorem{definition}[theorem]{Definition}
\newtheorem{example}[theorem]{Example}

\newtheorem{ques}[theorem]{Question}

\newtheorem{notation}[theorem]{Notation}

\newtheorem{remark}[theorem]{Remark}

\newtheorem*{definition*}{Definition}

\numberwithin{equation}{section}

\title[Dehn functions of central products]{On the Dehn functions of central products of nilpotent groups}

\author{Jer\'onimo Garc\'ia-Mej\'ia}
\address{Mathematics Institute, University of Warwick, Coventry CV4 7AL}
\email{jeronimo.garcia-mejia@kwarwick.ac.uk}

\author{Claudio Llosa Isenrich}
\address{Department of Mathematics, University of Luxembourg, 6 Avenue de la Fonte, 4364
Esch-sur-Alzette, Luxembourg}
\email{claudio.llosaisenrich@uni.lu}

\author{Gabriel Pallier}
\address{Univ. Lille, CNRS, UMR 8524 - Laboratoire Paul Painlevé, F-59000 Lille, France}
\email{gabriel.pallier@univ-lille.fr}

\thanks{The authors gratefully acknowledge funding by the DFG 281869850 (RTG 2229).}



\dedicatory{}

\renewcommand{\thesubsection}{\arabic{section}.\Alph{subsection}}

\newcommand{\gabriel}[1]{\begin{color}{purple}Gabriel: #1\end{color}}

\begin{document}

\begin{abstract}
    We determine the Dehn functions of central products of two families of filiform nilpotent Lie groups of arbitrary dimension with all simply connected nilpotent Lie groups with cyclic centre and strictly lower nilpotency class. We also determine the Dehn functions of all central products of nilpotent Lie groups of dimension at most $5$ with one-dimensional centre. This confirms a conjecture of Llosa Isenrich, Pallier and Tessera for these cases, providing further evidence that the Dehn functions of central products are often strictly lower than those of the factors. Our work generalises previous results of Llosa Isenrich, Pallier and Tessera and produces an uncountable family of nilpotent Lie groups without lattices whose Dehn functions are strictly lower than the ones of the associated Carnot graded groups.
    A consequence of our main result is the existence of an infinite family of groups such that Cornulier's bounds on the $e$ for which there is an $O(r^e)$-bilipschitz equivalence 
    between them and their Carnot graded groups
    are asymptotically optimal, as the nilpotency class goes to infinity.
\end{abstract}
\maketitle

\tableofcontents

\section{Introduction}

The Dehn function is a natural and powerful quasiisometry invariant of compactly presented groups, which has attracted much interest in Geometric Group Theory over the last decades.\footnote{Strictly speaking Dehn functions are only quasiisometry invariants up to asymptotic equivalence. We make this precise later.} For a given compactly presented group $G$, its Dehn function $\delta_G$ provides a quantitative measure for the complexity of its word problem. Dehn functions can also be interpreted geometrically as optimal isoperimetric functions for loops in simply connected spaces on which the group acts geometrically. 

An important class of compactly presented groups in Geometric Group Theory are finitely generated nilpotent groups. This can at least be partially motivated by Gromov's Polynomial Growth Theorem, which says that (up to finite index)  finitely generated nilpotent groups are precisely the finitely generated groups of polynomial growth, making them a natural class of groups to study from a geometric point of view. While a lot of work has been done in this area and many interesting results have been proved, maybe somewhat surprisingly basic questions about their geometry remain open. In particular, this is the case for their conjectural quasiisometry classification, which asserts that two simply connected nilpotent Lie groups are quasiisometric if and only if they are isomorphic \cite[Conjecture 19.114]{CornulierQIHLC}. 
In view of Pansu's theorems it reduces to classifying nilpotent groups with isomorphic associated Carnot graded group (or, equivalently, bilipschitz homeomorphic asymptotic cones) up to quasiisometry \cite{PanCBN,PansuCCqi}; see \Cref{sec:Carnot-graded-and-Pansus-Theorem} for a precise statement and definitions.  One reason that further progress has been very limited is that we still only know very few quasiisometry invariants of nilpotent groups that distinguish between groups with the same associated Carnot graded group. For a long time the only such invariant was the real cohomology algebra \cite{ShalomHarmonic,SauerHom,GotfredsenKyed}. In \cite{lipt} the second and third author with Tessera proved that Dehn functions provide such an invariant by producing a family of $k$-nilpotent pairs $(G_k,G_k')$ for every $k\geqslant 3$ such that $\delta_{G_k}(n)\asymp n^k\prec n^{k+1}\asymp \delta_{G_k'}(n)$.

The result of the second and third author with Tessera showed that Dehn functions are of interest in the context of the quasiisometry classification of nilpotent groups. It also highlighted that despite extensive work and many fundamental results, we are still far from understanding them completely. By work of Gersten--Holt--Riley \cite{GerstenRileyHolt} we know that the Dehn function of a finitely generated $k$-nilpotent group is at most $n^{k+1}$. It is also known that Dehn functions can take any polynomial value with integral exponent \cite{BaumslagMillerShort,BridsonPittet} and that they can even be non-polynomial \cite{WengerNonPoly}. This brief overview is far from exhaustive, rather than trying to give a full survey, we  now proceed to highlight some open problems that this work addresses.

The focus of this work is on Dehn functions of central products. Let $K$ and $L$ be groups, and let $\theta:Z(K)\to Z(L)$ be an isomorphism between their centres. The \emph{central product} of $K$ and $L$ is $K\times_\theta L = (K \times L)/R$, where 
\[ R = \{ (g,h) \in Z(K) \times Z(L): h = \theta(g) \}. \]
When there is no ambiguity on $\theta$ we simply write $G = K\times_ZL$. Central products can be taken in the category of topological groups, Lie groups and real Lie algebras.

For central products of nilpotent groups, there is at present growing evidence that the Dehn function is often smaller than the Gersten--Holt--Riley upper bound \cite{Allcock, OlsSapCombDehn,YoungFillingNil,lipt}. 
The first result in this direction was due to Allcock \cite{Allcock} and Olshanskii--Sapir \cite{OlsSapCombDehn} who proved that the $5$-Heisenberg group, which is the central product of two $3$-Heisenberg groups (which have cubic Dehn function) has quadratic Dehn function. In Olshanskii--Sapir's proof the decomposition as a central product played a key role and indeed they showed more generally.

\begin{theorem}[{Allcock \cite{Allcock}, Olshanskii--Sapir \cite{OlsSapCombDehn}, Young \cite{YoungFillingNil}\footnote{The statement about Heisenberg groups is due to Allcock and Olshanskii--Sapir. The second part was stated without proof by Olshanskii--Sapir in \cite{OlsSapCombDehn} and a proof was given by Young \cite{YoungFillingNil}.}}]\label{thm:Allcock-OS}
    The Dehn function of the $(2m+1)$-Heisenberg group is $n^2$ if $m \geqslant 2$. If, more generally, $G=K\times_Z K$ is a central product of a non-trivial finitely generated torsion-free 2-nilpotent group $K$ with itself, then $\delta_G(n)\preccurlyeq n^2 \log(n)$.
\end{theorem}

The main result proved by the second and third author with Tessera in \cite{lipt} is a version of Theorem \ref{thm:Allcock-OS} for  a family of groups of higher nilpotency class. As a consequence they proposed the following version of Theorem \ref{thm:Allcock-OS} for groups of higher nilpotency class.

\begin{conjecture}[{\cite[Conjecture 11.3]{lipt}}]
\label{conj:main}
Consider a central product 
\[G =K \times_Z L \]
where $K$ and $L$ are simply connected nilpotent Lie groups with one-dimensional centres, and class $k$ and $\ell$ respectively, where $2 \leqslant \ell \leqslant k$.
Then $n^k \preccurlyeq \delta_G(n) \prec n^{k+1}$.
\end{conjecture}

In this work we provide evidence for this conjecture, by proving it for large classes of groups. In particular, we prove it for certain $L$ and all $K$ provided that the nilpotency class $k$ of $K$ is lower than the nilpotency class $\ell$ of $L$, and for all central products with both factors of dimension at most five.

In order to give precise statements of our main results, we recall that a simply connected nilpotent Lie group is called \emph{filiform of class $k$} if it is of minimal possible dimension, namely $k+1$, among all nilpotent groups of class $k$. 
One can easily prove that the Lie algebra of such a group is spanned by a particular basis, called an adapted basis, $X_1, \ldots, X_{k+1}$, such that $[X_1,X_i] = X_{i+1}$ for $2 \leqslant i \leqslant k$, with possibly other brackets being nonzero. If the Lie algebra has an adapted basis where the brackets above are the only nonzero brackets, we say that the group is \emph{model filiform}, and denote it $L_{k+1}$. When $k \geqslant 4$, if the only nonzero brackets are the ones above and $[X_2,X_3] = X_{k+1}$, we denote the group $L^\lrcorner_{k+1}$. The group $L^{\lrcorner}_{k+1}$ is not a Carnot graded group; its associated Carnot graded group is the model filiform group $L_{k+1}$. Note that the group $L_{3}$ is the $3$-Heisenberg group, making our results below natural generalisations of Theorem \ref{thm:Allcock-OS} and of \cite[Theorem A]{lipt}.

\begin{theoremintro}[One filiform factor of highest nilpotency class]\label{th:general-factor}
    Let $k> \ell\geqslant 2$ be integers.
    Let $K$ be either the group $L_{k+1}$ or $L_{k+1}^\lrcorner$.\footnote{The group $L_{k+1}^{\lrcorner}$ is only defined for $k\geqslant 4$. So for $k=3$ (resp. $\ell\leqslant 3$) the only allowed choice for $K$ (resp. $K$ or $L$) in Theorem \ref{th:general-factor} (resp. Theorem \ref{th:model-filiform}) is $L_{4}$ (resp. $L_3$ or $L_4$).} Let $L$ be a simply connected nilpotent Lie group with one-dimensional centre of nilpotency class $\ell$.
    Let $G = K\times_ZL$.
    Then $\delta_G(n) \asymp n^k$.
\end{theoremintro}

Theorem~\ref{th:general-factor} shares a large part of its proof with the following Theorem~\ref{th:model-filiform}, which is more general in that it allows equality of the nilpotency classes, but less general in that it makes some assumptions on both factors of the central product.

\begingroup

\begin{theoremintro}[Two filiform factors]\label{th:model-filiform}
    Let $k\geqslant \ell \geqslant 2$ be integers.
    Let $K$ be either the group $L_{k+1}$ or $L_{k+1}^\lrcorner$ . Let $L$ be either the group $L_{\ell+1}$ or $L_{\ell+1}^\lrcorner$. Let $G = K\times_ZL$.
    Then $\delta_G(n) \asymp n^k$.
\end{theoremintro}
\endgroup

For a comparison, \cite[Theorem A]{lipt} is the special case of Theorem~\ref{th:model-filiform} where $k-1 \leqslant \ell \leqslant k$ and both groups $K$ and $L$ are assumed to be model filiform; the case $k=\ell=2$ follows from Allcock and Olshanskii-Sapir's theorems. The fact that we obtain $n^k$ as the Dehn function also answers the question raised in \cite[Remark 8.13]{lipt} if the Dehn functions of the central products considered in Theorem \ref{th:model-filiform} could have non-integer exponents for certain combinations of $k$ and $\ell$. We emphasize that this only removes these groups as candidates for such examples, leaving open the general question if there exists a nilpotent group whose Dehn function is polynomial (or non-polynomial) with non-integer exponent (in the polynomial part).

\begin{figure}[H]
    \begin{tikzpicture}[line cap=round,line join=round,>=triangle 45,x=1.0cm,y=0.9cm]
\clip(-1.7,2.5) rectangle (4.3,6);
\draw (-0.8,6.06) node[anchor=north] {$x_1 $};
\draw (1,6.06) node[anchor=north] {$ x_2 $};
\draw (1,5.07) node[anchor=north ] {$ x_3 $};
\draw (1,4.06) node[anchor=north ] {$ x_4 $};
\draw (1.05,3.05) node[anchor=north] {$ x_5 $};
\draw (2.05,6.06) node[anchor=north ] {$ y_1 $};
\draw (3.05,6.06) node[anchor=north ] {$ y_2 $};
\draw (-0.6,5.6)-- (1,5);
\draw (1,5.6)-- (1,5);
\draw (-0.8,5.6)-- (1,4);
\draw (1,4.6)-- (1,4);
\draw (-1,5.6)-- (1,3);
\draw (1,3.4)-- (1,3);
\draw (2,5.5)-- (1.2,3);
\draw (3,5.5)-- (1.2,3);
\end{tikzpicture}
\begin{tikzpicture}[line cap=round,line join=round,>=triangle 45,x=1.0cm,y=0.9cm]
\clip(-1.7,2.5) rectangle (4.3,6.4);
\draw (-0.8,6.06) node[anchor=north] {$x_1 $};
\draw (1,6.06) node[anchor=north] {$ x_2 $};
\draw (1,5.07) node[anchor=north ] {$ x_3 $};
\draw (1,4.06) node[anchor=north ] {$ x_4 $};
\draw (1.05,3.05) node[anchor=north] {$ x_5 $};
\draw (3.05,6.06) node[anchor=north ] {$ y_1 $};
\draw (4.05,6.06) node[anchor=north ] {$ y_2 $};
\draw (-0.6,5.6)-- (1,5);
\draw (1,5.6)-- (1,5);
\draw (-0.8,5.6)-- (1,4);
\draw (1,4.6)-- (1,4);
\draw (-1,5.6)-- (1,3);
\draw (1,3.4)-- (1,3);
\draw (3,5.5)-- (1.4,3);
\draw (4,5.5)-- (1.4,3);
\draw [shift={(0.1,4)}] plot[domain=-0.69:0.49,variable=\t]({1*1.41*cos(\t r)},{1*1.41*sin(\t r)});
\draw [shift={(-0.4,4.5)}] plot[domain=-0.79:0.53,variable=\t]({1*2.12*cos(\t r)},{1*2.12*sin(\t r)});
\draw [shift={(1.1,0)}] (0.62,4.51)-- (0.5,4.75);
\draw [shift={(1.1,0)}]  (0.62,4.51)-- (0.74,4.74);
\draw [shift={(1.1,0)}]  (0.33,4.47)-- (0.28,4.2);
\draw [shift={(1.1,0)}] (0.33,4.47)-- (0.5,4.22);
\end{tikzpicture}
    \caption{The groups in Example~\ref{exm:L5L3} have lattices whose presentation we may draw using diagrams where the letters represent the generators and there is an oriented curve with one corner that goes from $x$ to $z$ and then to $y$
    whenever $[x,y] = z$. When the broken curve can be read from left to right we omit the orientation.}
    \label{fig:exm-intro}
\end{figure}

\begin{example}[See Figure \ref{fig:exm-intro}] \label{exm:L5L3}
    The $7$-dimensional central products $L_5 \times_Z L_3$ and $L^\lrcorner_5 \times_Z L_3$, whose Lie algebras are the only real forms of the Lie algebras denoted $\mathcal G_{7,3.17}$ and $\mathcal G_{7,2.30}$ respectively in \cite{Magnin}, have a Dehn function $n^4$ as a consequence of Theorem~\ref{th:model-filiform}. This does not follow from \cite{lipt}, since $k - \ell >1$ in these cases, and $L^\lrcorner_5$ is not model filiform.
\end{example}

Although Theorems~\ref{th:general-factor} and \ref{th:model-filiform} support Conjecture~\ref{conj:main}, they do not provide a complete picture in low dimensions. We are able to treat low-dimensional groups separately, with tailored methods, which turn out to be often more elementary than those employed for these two Theorems. Namely, we prove the following.

\begin{theoremintro}[Low dimensional case]
\label{th:lowdim}
    Let $k\geqslant \ell\geqslant 2$ be integers and let $K$ and $L$ be simply connected nilpotent Lie groups with one-dimensional centres, and class $k$ and $\ell$ respectively. Assume that
    $
    \max\{\dim K, \dim L\} \leqslant 5 
    $.
    Let $G = K \times_ZL$.
    Then $\delta_G(n) \asymp n^k$.
\end{theoremintro}

As a consequence of Theorems \ref{th:general-factor}, \ref{th:model-filiform} and \ref{th:lowdim}, we significantly extend the class of pairs of groups with the same associated Carnot graded group and different Dehn functions. Since the Dehn function is a quasiisometry invariant we distinguish such groups from each other up to quasiisometry. 

\begin{corollaryintro}\label{cor:QI}
Let $K$ and $L$ be simply connected nilpotent Lie groups of classes $k$ and $\ell$ with $k>\ell$ satisfying the assumption in Theorem~\ref{th:general-factor}, \ref{th:model-filiform} or \ref{th:lowdim}.
Then the group $G = K \times_ZL$ is not quasiisometric to its associated Carnot graded group.
\end{corollaryintro}

 In Appendix~\ref{sec:QI} we explain how one can also deduce Corollary~\ref{cor:QI} in special cases from Shalom's Theorem \cite{ShalomHarmonic}. We are not aware of a way to derive the general case of Corollary \ref{cor:QI} from \cite{ShalomHarmonic,SauerHom,GotfredsenKyed} and it would be interesting to know if such a proof exists. We note that for some groups covered by Corollary~\ref{cor:QI}, the first, second and (in one case) the third Betti numbers are the same as the ones of their associated Carnot graded groups, and thus if one uses \cite{ShalomHarmonic} or \cite{GotfredsenKyed} to reach the same conclusion, one needs to consider third or fourth Betti numbers (See Figure \ref{fig:fourth-betti}, and Table \ref{tab:betti} in Appendix~\ref{sec:QI}).

 \begin{figure}
     \centering
       \begin{tikzpicture}[line cap=round,line join=round,>=triangle 45,x=1cm,y=0.6cm]
\clip(-5,-2.5) rectangle (7,1.9);
\draw (-3.65,2.04) node[anchor=north west] {$x_1 $};
\draw (-1.64,2.04) node[anchor=north west] {$ x_2 $};
\draw (0.08,2.04) node[anchor=north west] {$ x_5 $};
\draw (2.07,2.04) node[anchor=north west] {$ y_1 $};
\draw (4.3,2.04) node[anchor=north west] {$y_3$};
\draw (-2,0.21) node[anchor=north west] {$ x_3 $};
\draw (0.08,-1.85) node[anchor=north west] {$ x_4 $};
\draw (-3,1)-- (-1.72,0.11);
\draw (-1.5,1)-- (-1.72,0.11);
\draw (-1,1)-- (0.24,-1.98);
\draw (0.27,1.03)-- (0.24,-1.98);
\draw (-1.27,-0.53)-- (0,-2);
\draw (2.43,1)-- (0.56,-2) -- (3.2,1);
\draw (3.2,2.04) node[anchor=north] {$ y_2$};
\draw [shift={(3,4)}] plot[domain=3.61:4.25,variable=\t]({1*6.71*cos(\t r)+0*6.71*sin(\t r)},{0*6.71*cos(\t r)+1*6.71*sin(\t r)});
\draw [shift={(-4.74,8.55)}] plot[domain=5.21:5.6,variable=\t]({1*12.01*cos(\t r)},{1*12.01*sin(\t r)});
\draw [shift={(-1.12,8.23)}] plot[domain=4.92:5.52,variable=\t]({1*10.44*cos(\t r)},{1*10.44*sin(\t r)});
\draw (6.06,2.04) node[anchor=north west] {$y_4 $};
\end{tikzpicture}
     \caption{ The group above, pictured with the graphical convention of Figure \ref{fig:exm-intro}, has the same first, second, and third Betti numbers as the ones of its associated Carnot graded group, namely $7$, $21$, and $34$, but the two groups are not quasiisometric by Corollary~\ref{cor:QI}.}
     \label{fig:fourth-betti}
 \end{figure}

 One can quantify the failure of existence of a quasiisometry in Corollary~\ref{cor:QI} using \emph{sublinear bilipschitz equivalences}. Recall that two metric spaces $X$ and $Y$ are called \emph{$O(r^e)$-bilipschitz equivalent},  for $e \in [0,1)$, if there exist $L>1$, $c \geqslant 0$, a pair of maps $(f\colon X \to Y,~g\colon Y \to X)$, and some $x_0 \in X$ and $y_0 \in Y$, such that for all $r>0$ the map $f$ (resp. $g$) is a $(L,cr^e)$-quasiisometric embedding when restricted to $B(x_0,r)$ (resp. to $B(y_0,r)$), and for all $x\in X,~ y\in Y$
 \[ d(g(f(x)), x) \leqslant c(1+d(x,x_0)^e), \quad  d(f(g(y)),y)  \leqslant c(1+d(y,y_0)^e).
 \]
 This notion does not depend on the choice of $x_0$ and $y_0$, up to changing $c$. In \cite[\S 6]{cornulier2017sublinear} Cornulier introduced a computable algebraic invariant $e_G \in \left[0,1\right)$ for every simply connected nilpotent Lie group such that $G$ is $O(r^{e_G})$-bilipschitz equivalent to its associated Carnot graded group. He asked in \cite[Question 1.23]{cornulier2017sublinear} whether this invariant is optimal (for a given fixed group). Here we show that for the groups $L_{k+1}\times_Z L_3$ Cornulier's invariant is asymptotically optimal as the nilpotency class $k$ goes to infinity.

 \begin{theoremintro}
 \label{th:sbe}
    Let $k \geqslant 2$.
     Let $\underline e$ be the infimum of the exponents $e$ such that there exists a $O(r^e)$-sublinear bilipschitz equivalence between $G = L_{k+1} \times_Z L_3$ and its associated Carnot graded group. Then 
     \begin{equation*}
         \frac{1}{2k+2} \leqslant \underline e \leqslant \frac{2}{k} = e_G.
     \end{equation*}
 \end{theoremintro}
The upper bound is due to Cornulier \cite{cornulier2017sublinear} and the lower bound follows from our work. Note that asymptotically as $k\to \infty$ the two bounds coincide up to a ratio of $4$, showing that Cornulier's general upper bound is optimal in this asymptotic sense. To our knowledge these are the first examples of this kind. In \cite{lipt} the authors produced non-zero lower bounds on $\underline{e}$ for all groups $L_{k+1}\times_Z L_{k}$ with $k\geqslant 3$, but in this case the asymptotics of their lower bound did not coincide with those of Cornulier's upper bound as $k\to \infty$.

We finish by observing that the family of pairwise non-isomorphic groups covered by Theorem \ref{th:general-factor} and Corollary \ref{cor:QI} is uncountable in the strongest possible sense, even if we restrict to groups of nilpotency class 4. 

\begin{theoremintro}\label{thm:main-uncountable}
        There exists an uncountable family $\left\{(G_i,H_i)\right\}_{i\in I}$ of pairs of 4-nilpotent groups such that the following hold:
    \begin{enumerate}
        \item $H_i$ and $G_i$ have isomorphic associated Carnot graded group for every $i\in I$;
        \item $\delta_{G_i}(n)\asymp n^4 \prec n^5 \asymp \delta_{H_i}(n)$,  for every $i\in I$; and
        \item the associated Carnot graded groups of $H_i$ and $H_j$ for all $i,j\in I$ with $i\neq j$ are not isomorphic. In particular, $H_i$ and $H_j$ are not isomorphic.
    \end{enumerate}
\end{theoremintro}

As the proof will show, the groups satisfying the conclusion of Theorem \ref{thm:main-uncountable} are not very explicit. A weaker version of Theorem \ref{thm:main-uncountable}, where the $H_i$ are all isomorphic and only the $G_i$ form an uncountable pairwise non-isomorphic family, is true for a very explicit family of 14 dimensional nilpotent Lie groups. This also provides the smallest dimensional family of this kind that we are aware of, while the smallest dimension in which we know Theorem \ref{thm:main-uncountable} is satisfied is $15$.

\begin{example} \label{exm:nolattices}
    For $\lambda \neq 0$, we denote by $K_{7,\lambda}$ the simply connected Lie group whose Lie algebra is the unique real form of $\mathcal G_{7,1.1(i_\lambda)}$ in \cite{Magnin}. The $14$-dimensional Lie group $L_8 \times_Z K_{7,\lambda}$ has Dehn function $n^7$ as a consequence of Theorem \ref{th:general-factor}. For uncountably many values of $\lambda$, this group has no lattices. Since $ L_8 \times_Z K_{7,\lambda_1} \cong  L_8 \times_Z K_{7,\lambda_2}$ if and only if $\lambda_1= \lambda_2$, Corollary \ref{cor:QI} thus implies that there is an uncountable family of 14-dimensional nilpotent groups whose Dehn function differs from the one of the associated Carnot graded group (See Section \ref{subsec:central-product-wo-lattice} for details). 
\end{example}

The presence of a lattice in a nilpotent group is sometimes required in order to be able to evaluate its filling invariants. 
It is the case, for instance, in Pittet's proof of the Gersten--Holt--Riley upper bound for Carnot gradable groups in \cite{PittetIsopNilHom} and in Gruber's thesis on the higher filling invariants \cite{GruberFilling}.

As noted above, going from finitely presented groups to compactly presented ones, one is led to consider uncountably many groups. This raises the following question, with which we end this introduction.

\begin{ques}
Are there uncountably many possible growth types for the Dehn functions of compactly presented groups?
\end{ques}

\subsection*{Structure of the paper}

We start with some preliminary material in \S\ref{sec:prelim}, where we recall the basic definitions, our approach to Dehn functions using compact presentations and efficient words, and the specificities imposed by the groups we study in this paper. We prove Theorems~\ref{th:general-factor}, \ref{th:model-filiform}, and \ref{th:lowdim} by separately deducing a lower and upper bound for the respective Dehn functions.

For the lower bound, we shall make use of the same argument as in \cite[Section 8]{lipt}, but reframe it in a more qualitative and general way in \S\ref{sec:lowerbounds}.

The upper bounds required for Theorems~\ref{th:general-factor} and \ref{th:model-filiform} are dealt with in \S\ref{sec:upper-bounds}, where we also prove Theorem \ref{th:sbe}. To obtain them we generalize the methods from \cite[Section 6]{lipt}. This is achieved by replacing the first part of the proof given there by a simpler more direct argument, which enables us to cover the more general class of groups considered here. In particular, we replace the proof in \cite{lipt} of the First and Second Commuting lemmas \cite[Lemma 6.5 \& 6.6]{lipt} involving the highly technical Fractal Form Lemma \cite[Lemma 6.17]{lipt} by a much simpler induction argument. We will explain this simplification in more detail in \S\ref{sec:upper-bounds}. To showcase the strength of this new approach we cover the additional case of the groups $L_{k+1}^{\lrcorner}$ and we expect that a refinement of this approach will also enable an extension to other factors in the future.

The Dehn function in Theorem~\ref{th:lowdim} is obtained by a case-by-case study. Quite a few cases were either already known or are consequences of Theorems~\ref{th:general-factor} and ~\ref{th:model-filiform}. The remaining cases can be dealt with using relatively elementary arguments (without using compact presentations); we collect them in \S\ref{subsec:lowdim-case-by-case}.

Although, the lower bounds and the upper bounds on the Dehn functions are both part of the proofs of Theorems~\ref{th:general-factor}, \ref{th:model-filiform}, and \ref{th:lowdim}, the corresponding section where we deal with them can be read separately.

In \S\ref{sec:uncountable} we first prove Theorem \ref{thm:main-uncountable} and then explain Example \ref{exm:nolattices} in detail.

Our paper ends with two appendices. Appendix \ref{sec:QI} concerns other quasiisometry invariants of nilpotent groups. In Appendix \ref{sec:GHR} we check that the Gersten--Holt--Riley upper bound extends to Lie groups irrespective of whether they have lattices or not.

\subsection*{Notations and conventions}
We fix some notation and conventions that we use throughout this work.

When manipulating words with respect to a presentation $\calp:=\langle S|R \rangle$ of a group $G$ it is useful to distinguish between identities that hold in the free group $F_S$, called \emph{free identities} and denoted by $=_F$, and identities that hold in $\calp$, which we refer to as \emph{identities holding in $G$} and denote them by $=_\calp$ or $=_G$. When an equality follows from using a definition we simply write $=_\text{def}$. For a word $w$ in the generating set $S$ of a group $G$ with respect to a presentation $\calp$ we denote by $|w|_S$ the \emph{word length of $w$}. If there is no need to specify the generating set $S$ we simply write $|w|$.

We use the convention 
$[g,h] := g^{-1}h^{-1}gh$
and the notation \[ \left[x_1,~x_2,\dots,x_{k-1},~x_k\right]:=\left[x_1,\left[x_2,\dots,\left[x_{k-1},x_k\right]\right]\right] \] to denote a simple $k$-fold commutator of elements $x_1,\dots,x_k$ in a group or Lie algebra. Moreover, for $g\in G$ an element of a simply connected nilpotent Lie group $G$ and $a\in\R$, we denote $g^a:=\exp(a\cdot \log(g))$.

Given two functions $f,~g:\mathbf{R}_{\geqslant 0}\to \mathbf{R}_{\geqslant 0}$, we use the notation $f \lesssim_a g$ to mean that there exists some constant $C < \infty$ that only depends on $a$ such that $f \leqslant C \cdot g$. And $f \simeq_a g$ if $f \lesssim_a$ and $g \lesssim_a f$ hold. We also say that $f$ is in $O_a (g)$ if $f \lesssim_a g$ and $f = O_a (g)$ if $f \simeq_a g$. Moreover, we write $f\preccurlyeq g$ (or rather, by a convenient abuse of notation, $f(n) \preccurlyeq g(n)$) if there is a $C\geqslant 1$ such that $f(n)\leqslant Cg(Cn+C) + Cn+C$ for all $n$, and $f\asymp g$, if $f\preccurlyeq g\preccurlyeq f$. 

Unless otherwise stated, by {\em factor} of a central product $G = K\times_Z L$ we mean $K$ or $L$. When we consider direct factors of groups splitting as direct products instead, this will be explicitly specified.

All the Lie groups we consider in this paper are assumed to be connected Lie groups.

\subsection*{Acknowledgements}

We would like to thank Yves Cornulier, Pierre Pansu, Xiangdong Xie, and the anonymous referee for helpful comments.

\section{Preliminaries}

\label{sec:prelim}

In this section we first summarise some background on Dehn functions, nilpotent groups and compact presentations. We then introduce explicit compact presentations for some of the groups we consider later in the paper.

\subsection{Dehn functions of compactly presented groups}

Generalising finite generation and finite presentability for discrete groups, a locally compact topological group $G$ is \emph{compactly generated} if it admits a compact generating set $S$ and \emph{compactly presented} if, in addition, it admits a presentation with a bounded set of relations $R$ with respect to $S$. An important class of examples of compactly presented groups are simply connected Lie groups, see \cite[Theorem 2.6]{TesseraMetricLC}. This allows us to study the large-scale geometry of simply connected Lie groups via their compact presentations.

Let $G $ be a group defined by a compact presentation $\mathcal{P}:=\left\langle S\mid R\right\rangle$. We call a word $w \in F_S$ \emph{null-homotopic in} $G$ if it represents the trivial element in $G$. The \emph{(combinatorial) area} of a null-homotopic word $w$ in $G$ with respect to $\calp$ is defined by
\[
    {\rm Area}_{\mathcal P}(w):= {\rm min}\left\{k\in \mathbb{N}\mid w=_{F} \prod_{i=1}^k u_i r_i^{\epsilon_i}u_i^{-1}, \mbox{ with } u_i\in F_S,~ r_i\in R,~ \epsilon_i\in \left\{\pm 1\right\}\right\}.
\]

The Dehn function of $G$ is then defined as the function
\[
    \delta_G(n):={\rm max}\left\{{\rm Area}_{\calp}(w)\mid w=_G 1,~ \vert w \vert \leqslant n\right\}\in\mathbb{N}.
\]

Similar as in the case of finite presentations, one can show that the Dehn function is a well-defined quasiisometry invariant of compactly presented groups up to the asymptotic equivalence $\asymp$ of functions.
We refer to \cite[\S 2.B]{CoTesDehn} for further details.

Assume now that $G$ is a Lie group equipped with a left-invariant proper geodesic metric. To translate between words in $S$ and paths in $G$, we fix for every element $s\in S$ a continuous path $\overline{s}$ from $1_G$ to $s$ such that the set $\left\{  \text{length}(\overline{s}) \mid s\in S\right\}$ is bounded. For a word $w$ in $S$ we denote by $\overline{w}$ the path in $G$, obtained by concatenating paths corresponding to the generators. Moreover, for $g\in G$ and a path $\gamma:\left[a,b\right]\to G$ we denote by $g_\ast \gamma$ the path obtained from $\gamma$ by left translation by $g$ in $G$. 

To compute lower bounds on Dehn functions we use the following version of Stokes' formula for simply connected Lie groups equipped with compact presentations. 

\begin{lemma}
\label{combinatorial-stokes}
    Let $G$ be a simply connected Lie group and let $\calp:=\langle S \mid R \rangle$ be a compact presentation for $G$.
    Let $\alpha \in \Omega^1(TG)$ be a continuous 1-form on $G$.
    Assume that there exists a constant $C$ such that $\left\vert \int_{g_{\ast}\overline{r}} \alpha \right\vert \leqslant C$ for all $r\in R$.
    Then for any null-homotopic word $w$ over the presentation
    \[ \operatorname{Area}_\calp(w) \geqslant \frac{1}{C} \left\vert \int_{\overline{w}} \alpha \right\vert. \]
\end{lemma}
\begin{proof}
    This is, up to reformulation, \cite[Proposition 8.2]{lipt}. Let us mention that although a Riemannian metric is mentionned in \cite[Proposition 8.2]{lipt}, the area is the combinatorial area and not the Riemannian area of a filling disk of $\overline w$.
\end{proof}

\subsection{Nilpotent groups}

The \emph{lower central series} of a group $G$ is the descending sequence of subgroups which is defined inductively by $C^1(G):= G$ and $C^{i}(G):= \left[G,C^{i-1}(G)\right]$ for $i>1$. We call $G$ \emph{nilpotent of class $k$} if $C^{k}(G)\neq \left\{1\right\}$ and $C^{k+1}(G)=\left\{1\right\}$.

Among Lie groups, nilpotent Lie groups of class $k$ can also be characterised as those Lie groups $G$ which have a \emph{nilpotent} Lie  algebra $\mathfrak{g}$, that is, a Lie algebra whose lower central series becomes trivial precisely in the $(k+1)$-th term. Here the \emph{lower central series} of a Lie algebra is defined analogously to the lower central series of a group: $C^1(\mathfrak{g}):=\mathfrak{g}$ and $C^{i}(\mathfrak{g}):=\left[\mathfrak{g},C^{i-1}(\mathfrak{g})\right]$ for $i>1$. Conversely, we can completely characterise a simply connected nilpotent Lie group by giving the Lie bracket structure with respect to a basis of its Lie algebra. We often use this fact here without further mention.

In this work we study Dehn functions of  simply connected finite-dimensional nilpotent Lie groups and their lattices. Every finitely generated nilpotent group is a co-compact lattice in a simply connected finite-dimensional real nilpotent Lie group, uniquely defined up to isomorphism, called its real Malcev completion \cite{MalcevNilvarietes}. Their Dehn functions coincide up to asymptotic equivalence and it turns out to be technically not harder to manipulate words over compact presentations than over finite ones, while it offers some additional flexibility. In addition, the Lie group structure makes some tools from differential geometry readily available. So in this work, we mainly approach the problem of computing the Dehn function of nilpotent groups by considering directly that of simply connected nilpotent Lie groups. 
A notable exception to this standpoint is in \S\ref{subsec:lowdim-case-by-case} where no tool from the theory of Lie groups and Lie algebras is required in the upper bounds, since we can argue more directly using a suitable finite presentation in this case.

\subsection{Compact presentations for central products of nilpotent groups}

\begin{lemma}\label{lem:central-prod-is-cglc}
    Let $K$ and $L$ be compactly presented nilpotent groups of class $k$ and $\ell$ respectively, such that $Z(K)={C^{k}(K)}$ and $Z(L) = {C^{\ell}}(L)$. Let $\theta$ be a continuous isomorphism between $Z(K)$ and $Z(L)$.
    Then the group $G = K \times_\theta L$ is compactly presented.
    {Moreover, given two compact presentations of $K$ and $L$, there is a compact presentation of $G = K\times_Z L$, whose generating system and set of relators contain those of $K$ and $L$, when viewed as closed subgroups of $G$.}
\end{lemma}
\begin{proof}
    Let $\mathcal P_{K} = \langle \mathcal S_K \mid \mathcal R_K \rangle$ and $\mathcal P_L = \langle \mathcal S_L \mid \mathcal R_L \rangle$ be compact presentations of $K$ and $L$ respectively.
    $Z(K)$ is a closed subgroup of $K$, and $K$ is nilpotent, so $Z(K)$ is compactly generated by \cite[Proposition 5.A.7]{CoHarpe}.
    Let $\mathcal S_{Z(K)}$ be a compact generating set of $Z$. By the assumption that  $Z(K) = C^{k} (K)$, there is $N < +\infty$ such that $\mathcal S_{Z(K)} \subset \mathcal S_K^N$ \cite[Lemma 5.A.5]{CoHarpe}. Similarly, if $\mathcal S_{Z(L)}$ is a compact generating set of the centre $Z(L)$ of $L$, then there is $M < +\infty$ such that $\mathcal S_{Z(L)} \subset \mathcal S_L^M$. Now set 
    \[ \mathcal S = \mathcal S_K \sqcup \mathcal S_L, \quad \mathcal R = \mathcal R_K \sqcup \mathcal R_L \sqcup \mathcal R_+  \sqcup \{ [a,b]: a \in \mathcal S_K, b \in \mathcal S_L \} \]
    where $\mathcal R_{+}$ is constructed as follow: for every $z \in \mathcal S_{Z(K)}$, write $z$ as a product of at most $N$ elements of $\mathcal S_K$, write $\theta(z) \in Z(L)$ as a product of at most $M$ elements of $\mathcal S_L$; let $v$ and $w$ be the resulting words, and add $vw^{-1}$ in $\mathcal R_+$. Clearly these additional relations are of length bounded by $N+M$.
    Hence, $\mathcal S$ is a compact subset of $G$, $\mathcal R$ is a bounded subset of $F_S$ and all the relation that hold true in $G$ are products of conjugates of relations in $\mathcal R$. This shows that $\left\langle S \mid R \right\rangle$ is a compact presentation for $G$.
\end{proof}

\begin{definition}\label{dfn:adapted}
   We call the presentation of the central product $K \times_Z L$ constructed in the proof of Lemma~\ref{lem:central-prod-is-cglc} {\em adapted} to the central product decomposition. 
\end{definition}

\begin{lemma}\label{lem:back-first-factor}
    Let $K$ and $L$ be compactly presented nilpotent groups with isomorphic centres and let $G$ be their central product.
    Assume that $K$ has a closed subgroup $K_0$ containing the centre which is isomorphic to $L$, and that $\delta_{G_0}(n) \preccurlyeq n^p$ where $G_0 = L \times_Z L$ for some $p\geqslant 2$.
    If every word of length $n$ over the generators of $K<G$ representing the trivial element has area at most $n^p$ in $G$, then
    $\delta_G(n) \preccurlyeq n^p.$
\end{lemma}
\begin{proof}
Let $\mathcal P$ be an adapted compact presentation for $G$.
Without loss of generality, assume that there is an isomorphism $L \to K_0$ such that $\mathcal S_K$ contains the image of $\mathcal S_L$.
Start from a null homotopic word $u$ of length $n$ over the generating set $\mathcal S$. Using $\leqslant Cn^2$ commutator relations, we obtain that\[ u =_G w v \]
where the length of $w$ and $v$ is at most $n$. Since $u$ is null-homotopic, $w=_G v^{-1}$ and, in particular, $w$ and $v$ represent central elements.
Using that $\delta_{G_0}(n) \preccurlyeq n^p$ we may rewrite $v$ as $v'$, where $v'$ is a word in the generating set $\mathcal S_K$ of the same length as $v$. Since $\delta_K(n)\preccurlyeq n^p$, this completes the proof.
\end{proof}

\subsection{Some explicit compact presentations}
\label{subsec:explicit-compact-presentations}

At various points of our proofs, we work with explicit compact presentations of the nilpotent groups under consideration. For this reason we provide them here. They are of two kinds: filiform and low-dimensional. Adapted compact presentation for the central products can then be deduced using Lemma~\ref{lem:central-prod-is-cglc}.

\subsubsection{Two families of filiform groups}

Recall that a nilpotent Lie group of dimension $p$ is called filiform if it has maximal possible (finite) nilpotency class, namely $p-1$, among all Lie groups of dimension $p$.
We consider two infinite families of filiform Lie groups.
Each family has one member in every dimension $p \geqslant 3$ ($p \geqslant 5$ for the second one), that we denote $L_p$ and $L^{\lrcorner}_p$ respectively.

$L_p$ is the model filiform group, whose Lie algebra $\mathfrak l_p$ is generated by $X_1,X_2,\ldots, X_{p}$ subject to the nonzero brackets
\begin{equation*}
     [X_1,X_i] = X_{i+1} \ \text{for} \ 2 \leqslant i \leqslant p-1.
\end{equation*}

The next Proposition supports the definition of $L^{\lrcorner}_p$:

\begin{proposition}
\label{prop:presentation-lattice}
For every $p \geqslant 5$, there exists a Lie algebra with basis $X_1,\ldots ,X_p$ and nonzero brackets
\begin{equation}
\label{eq:brackets-in-filiform-non-model}
[X_1,X_i] = X_{i+1} \ \text{for} \ 2 \leqslant i \leqslant n-1, \ [X_2,X_3] =X_p.
\end{equation}
This Lie algebra is not isomorphic to $\mathfrak l_p$, and the associated simply connected Lie group contains a lattice,  with the following presentation:
\[ \left\langle x_1, \ldots, x_p \left\vert
[x_1,x_i] = x_{i+1}\ \text{for} \leqslant i \leqslant p-1,\
[x_2,x_3] =x_p, ~[x_i,x_j]=1~ \text{for all other}\ \{i,j\}
 \right. \right\rangle. \]
\end{proposition}

\begin{remark}
    Although we are not aware of publications considering $\mathfrak l^{\lrcorner}_p$ for all $p$ simultaneously, this Lie algebra can be found in the classifications of nilpotent Lie algebras (up to dimension 7) and of filiform nilpotent Lie algebras (up to dimension $11$).
    The correspondences are as follows. In \cite{deGraafclass}, $\mathfrak l^{\lrcorner}_5$ and $\mathfrak l^\lrcorner_6$ are $L_{5,6}$ and $L_{6,17}$ respectively.
    In \cite{Magnin}, $\mathfrak l^{\lrcorner}_5, \mathfrak l^{\lrcorner}_6$ and $\mathfrak l^{\lrcorner}_7$ are the unique real forms of $\mathcal G_{5,3}$, $\mathcal G_{6,17}$ and $\mathcal G_{7,1.6}$ respectively.
    Finally, in \cite{GomezJimenezFiliform} where filiform Lie algebras of dimension less or equal $11$ are classified, $\mathfrak l^{\lrcorner}_5, \ldots , \mathfrak l^{\lrcorner}_{10}$ and $\mathfrak l_{11}^{\lrcorner}$ are $\mu_{5}^{2}, \mu_{6}^{3}, \mu_{7}^{3}, \mu_{8}^{18}, \mu_9^{37}, \mu_{10}^{50}$ and $\mu_{11}^{105}$ respectively.
\end{remark}

\begin{remark}
    The Lie algebra $\mathfrak l^{\lrcorner}_p$ does admit a grading for all $p \geqslant 5$. It is as follows: $X_1$ has weight $1$, $X_2$ has weight $p-3$, and the weight of $X_k$ is $p-5+k$ for $3 \leqslant k \leqslant p$. 
\end{remark}

\begin{proof}[Proof of Proposition~\ref{prop:presentation-lattice}]
    The fact that $\mathfrak l_p^\lrcorner$ defines a Lie algebra can be checked by observing that it is a central extension of $\mathfrak l_{p-1}$. Namely, let $(\theta_i)_{1 \leqslant i \leqslant p-1}$ be the basis of $\mathfrak l_{p-1}^\star$ dual to $(X_i)$. Then the two-form $\theta_2 \wedge \theta_3$ is a cocycle, since 
    \[ d(\theta_2 \wedge \theta_3) = d\theta_2 \wedge \theta_3 - \theta_2 \wedge d\theta_3 =  \theta_2 \wedge \theta_1 \wedge \theta_2 = 0. \]
    The Lie algebra $\mathfrak{l}_p^{\lrcorner}$ is then the central extension corresponding to the cocycle $\theta_1\wedge \theta_{p-1} + \theta_2\wedge \theta_3$. 
    
    To see that $\mathfrak l^\lrcorner _p$ is not isomorphic to $\mathfrak l_p$, observe that $\mathfrak l_p$ has an element, namely $X_2$, which is not in the derived subalgebra, and whose adjoint $\operatorname{ad}_{X_2}$ has rank one, while $\mathfrak l_p^\lrcorner$ has no such element.
    Now, consider the exponential mapping $\exp\colon \mathfrak l^\lrcorner _p \to L^\lrcorner _p$, and define
    \[ x_i = \begin{cases} \exp(X_i) & 1\leqslant i \leqslant 2 \\
    [x_1,x_{i-1}]& 3 \leqslant i \leqslant p
    \end{cases} \]
    It remains to explain why $[x_2,x_3] = x_p$. 
    We first show by induction on $i$ that for all $i \in \{1, \ldots, p \}$, 
    \begin{equation}
    \label{eq:log-xi}
        \log x_i = X_i + \sum_{j>i} c_{i,j} X_j
    \end{equation}
    for some real coefficients $c_{i,j}$. It is clearly true by definition for $1 \leqslant i \leqslant 2$, with vanishing $c_{1,j}$ and $c_{2,j}$.
    Assume that \eqref{eq:log-xi} holds for some $2\leqslant i \leqslant p-1$.
    Using the Baker--Campbell--Hausdorff formula and \eqref{eq:brackets-in-filiform-non-model}, one gets that    
    {\small\begin{align}
        \log (x_1 x_i) & = X_1 + X_i + c_{i,i+1} X_{i+1} +  \frac{1}{2}X_{i+1} + \frac{1}{12}X_{i+2} - \frac{\delta_{i,2}}{12} X_p + \sum_{j=i+2}^p d_{i,{j-1}} X_j; \label{eq:first-BCH-induction} \\
        \log (x_1^{-1} x_i^{-1}) & = -X_1 -X_i -c_{i,i+1} X_{i+1} + \frac{1}{2}X_{i+1} - \frac{1}{12}X_{i+2} + \frac{\delta_{i,2}}{12} X_p + \sum_{j=i+2}^p d'_{i,{j-1}} X_j \label{eq:second-BCH-induction}
    \end{align}}%
    where $d_{i,j}$ and $d'_{i,j}$ are real constants that depend on $c_{i,j}$ and the higher  coefficients in the Baker--Campbell--Hausdorff formula.
    Now, using \eqref{eq:first-BCH-induction}, \eqref{eq:second-BCH-induction} and the Baker-Campbell-Hausdorff formula once again,
    {\small \begin{align*}
    \log x_{i+1} & = \log \left((x_1^{-1} x_i^{-1}) \cdot ( x_1 x_i) \right) \\
    & = -X_1 -X_i -c_{i,i+1} X_{i+1} + \frac{1}{2}X_{i+1} - \frac{1}{12}X_{i+2} + \frac{\delta_{i,2}}{12} X_p + \sum_{j=i+2}^p d'_{i,{j-1}} X_j \\
    & \quad + X_1 + X_i + c_{i,i+1} X_{i+1} +  \frac{1}{2}X_{i+1} + \frac{1}{12}X_{i+2} - \frac{\delta_{i,2}}{12} X_p + \sum_{j=i+2}^p d_{i,{j-1}} X_j \\
    & \quad + \frac{1}{2} \left[ -X_1, X_i + c_{i,i+1} X_{i+1} +  \frac{1}{2}X_{i+1} + \frac{1}{12}X_{i+2} - \frac{\delta_{i,2}}{12} X_p + \sum_{j=i+2}^p d_{i,{j-1}} X_j \right] \\
    & \quad - \frac{1}{2} \left[ X_1, -X_i -c_{i,i+1} X_{i+1} +  \frac{1}{2}X_{i+1} - \frac{1}{12}X_{i+2} + \frac{\delta_{i,2}}{12} X_p + \sum_{j=i+2}^p d'_{i,{j-1}} X_j \right] \\
    & \quad + \sum_{j=i+2}^p e_{i,j} X_j
    \end{align*}
    }%
where the $e_{i,j}$ depend on the higher coefficients in the Baker--Campbell--Hausdorff formula and the $d_{i,j}$ and $d'_{i,j}$ (we used that the derived subalgebra is abelian in order to omit the commutators $[X_{i+a}, X_{i+b}]$ for $a,b \ge 0$ in the above).
    Summing this yields \eqref{eq:log-xi} for some coefficients $c_{i,j}$ depending on $e_{i,j}$, $d_{i,j-1}$ and $d'_{i,j-1}$, ending the proof of \eqref{eq:log-xi} by induction on $i$.
    Finally, we successively check that
    {\small
    \begin{align*}
        \log(x_2x_3) & = X_2 + X_3 + \frac{1}{2} X_p + \sum_{j=4}^p c_{3,j} X_j; \\
        \log(x_2^{-1}x_3^{-1}) & = -X_2 -X_3 + \frac{1}{2} X_p -  \sum_{j=4}^p c_{3,j} X_j; \\
        \log [x_2,x_3] & = X_p \underset{\eqref{eq:log-xi}}{=} \log(x_p).
    \end{align*}
    }%
    This concludes the proof that the presentation of the lattice in $L^{\lrcorner}_p$ is the one given.
\end{proof}

\begin{proposition}\label{prop:compactpresentations}
    The group $L_p$ (for $p \geqslant 3$), respectively $L^{\lrcorner}_p$ (for $p \geqslant 5$) admits the compact presentation $\mathcal P_p = \langle \widehat S \mid R \rangle$, resp. $\mathcal P^\lrcorner_p = \langle \widehat S \mid R^\lrcorner \rangle$, over the generating set 
    $\widehat S = \{ x_i^a: 1 \leqslant i \leqslant p, -1 \leqslant a \leqslant 1 \}$, where
    \begin{align*}
        R & = \biggl\{ x_i^{a}x_i^{b} x_{i}^{-(a+b)}, [x_1^a,x_i^b] = x_{i+1}^{ab} x_{i+2}^{-\binom{a}{2}b}\cdots x_p^{(-1)^{p+i+1} \binom{a}{p-i}b},  \\
        & \qquad  i \in \{2, \ldots, p-1\}, a,b \in [-1,1],  \\
        & \qquad [x_k^a,x_\ell^b]=1~ \text{for all other pairs}\ \{k,\ell\}\ \text{and for all}~ a,b \in [-1,1] \biggr\} \text{ and} \\
     R^\lrcorner & = \left\{ x_i^{a}x_i^{b} x_{i}^{-(a+b)}, [x_1^a,x_i^b] = x_{i+1}^{ab} x_{i+2}^{-\binom{a}{2}b}\cdots x_p^{(-1)^{p+i+1} \binom{a}{p-i}b}, i \in \{3, \ldots, p-1\},  \right. \\
        & \qquad \left. [x_2^a,x_3^b] = x_p^{ab} , [x_1^a,x_2^b]= x_{3}^{ab} x_{4}^{-\binom{a}{2}b}\cdots x_p^{(-1)^{p+1} \binom{a}{p-2}b-a \binom{b}{2}}, a,b \in [-1,1] \right. \\
        & \qquad \left. [x_k^a,x_\ell^b]=1~ \text{for all other pairs}\ \{k,\ell\}\ \text{and for all}~ a,b \in [-1,1] \right\}
    \end{align*}
\end{proposition}

\begin{proof}
In both cases, $x_1, \ldots, x_p$ forms a Malcev basis of the simply connected Lie group (See e.g. \cite[\S 5.1]{lipt} for the definition of a Malcev basis). To obtain a compact presentation it suffices to check that the same relations hold between the generators $x_i$ for the Zariski dense subset of integers $a,~b$ with $a\geqslant 0$.{\footnote{The reader familiar with algebraic groups may derive the argument in the following way. From the rationality of the structure constants in $\mathfrak l_p$ and $\mathfrak l_p^\lrcorner$ one can deduce that the elements $x_1$ and $x_2$ together generate subgroups of finite index in the groups of integer points of $L_n = \operatorname{L}_p(\mathbf R)$ or $L_p^\lrcorner  = \operatorname L_p^\lrcorner(\mathbf R)$, where $\operatorname{L}_p$ and $\operatorname{L}^\lrcorner_p$ are certain unipotent algebraic groups defined over $\mathbf Q$. These lattices are Zariski-dense, and so a polynomial identity is valid over the group as soon as it is valid over the lattice. Concretely a polynomial identity is one that involves products of polynomial powers of elements taken in a fixed Malcev basis, or equivalently which is polynomial in the exponential charts \cite[15.32]{MilneAlgebraicGroups}.}} For $L_p$ this was proved in \cite[Proposition 5.2]{lipt}. For $L_p^{\lrcorner}$ we use Proposition \ref{prop:presentation-lattice} above to check that the identities
\begin{equation}
\label{eq:commutator-x1xi-corner}
    [x_1^a,x_i^b] = x_{i+1}^{ab} x_{i+2}^{-\binom{a}{2}b}\cdots x_p^{(-1)^{p+i+1} \binom{a}{p-i}b}
\end{equation}
for $3\leqslant i \leqslant p-1$
and 
\begin{equation}
\label{eq:commutator-x1x2-corner}
    [x_1^a,x_2^b] = x_{3}^{ab} x_{4}^{-\binom{a}{2}b}\cdots x_p^{(-1)^{p+i+1} \binom{a}{p-i}b-a\binom{b}{2}}
\end{equation}
hold for all $a,b\in \Z$.

Consider first \eqref{eq:commutator-x1xi-corner}. Since real powers of $x_1$ and $x_i$ together generate a group isomorphic to $L_{p-i-2}$, \eqref{eq:commutator-x1xi-corner} is a consequence of \cite[Proposition 5.1]{lipt}.
Turning to \eqref{eq:commutator-x1x2-corner}, let us first prove that
\begin{equation}
\label{eq:fix-x1-x2}
    [x_1,x_2^b] = x_3^b x_p^{-\binom{b}{2}},
\end{equation}
by induction on $b \in \mathbf Z_{\geqslant 0}$. For $b=1$ it holds by definition of $x_3$; assuming it holds for some $b$, we have that
{\small
\begin{align*}
    x_1 x_2^{b+1} = x_1 x_2^b x_2 = x_2^b x_1 x_3^b x_p^{-\binom{b}{2}} x_2  = x_2^b x_1 x_3^b  x_2  x_p^{-\binom{b}{2}} & =  x_2^b x_1   x_2 x_3^b  x_p^{-b-\binom{b}{2}} \\ & =   x_2^{b+1} x_1   x_3^{b+1}  x_p^{-\binom{b+1}{2}}.
\end{align*}}%
From this point, we can proceed as in the proof of \cite[Proposition 5.2]{lipt}, proving the formula by induction on $a$, replacing $[x_1,x_2^b] = x_3^b$ by \eqref{eq:fix-x1-x2}. We provide the details below; in the computation, ``I.H.'' means that we use the induction hypothesis. 
{\small
\begin{align*}
x_1^{a+1}x_2^b  = x_1^a x_1x_2^b 
& \stackrel{\eqref{eq:fix-x1-x2}}{=} x_1^a x_2^b x_1 x_3^b x_p^{-\binom{b}{2}}  
\\
& \stackrel{\mathrm{I.H.}}= x_2^b x_1^a x_{3}^{ab} x_{4}^{-\binom{a}{2}b} x_{5}^{\binom{a}{3}b} \cdots x_p^{(-1)^{p+1}\binom{a}{p-2}b  -a \binom{b}{2}} x_1 x_3^b x_p^{-\binom{b}{2}}  
\\
& = x_2^b x_1^{a+1} x_{3}^{ab} x_{4}^{-\binom{a+1}{2}b} x_{5}^{\binom{a+1}{3}b} \cdots x_p^{(-1)^{p+1}\binom{a+1}{p-2}b -a \binom{b}{2}} x_3^b x_p^{-\binom{b}{2}} 
\\
& = x_2^b x_1^{a+1} x_{3}^{(a+1)b} x_{4}^{-\binom{a+1}{2}b} x_{5}^{\binom{a+1}{3}b} \cdots x_p^{(-1)^{p+1}\binom{a+1}{p-2}b -(a+1) \binom{b}{2} }. \qedhere 
\end{align*}}%
\end{proof}

From now on, whenever a compact presentation has been provided for the factors, we equip their central product with an adapted presentation as produced by Lemma~\ref{lem:central-prod-is-cglc}. We denote by $y_i$ the generators of the right-hand side factor. As an example, the group $L_p^\lrcorner \times_Z L_3$ has as generating set
\[ \{ x_i^a, y_j^b : 1 \leqslant i \leqslant p,\ 1 \leqslant j \leqslant 3, \vert a \vert \leqslant 1, \vert b \vert \leqslant 1 \}, \]
and a set of relators including $R^{\lrcorner}$, $[y_1^a, y_{2}^b] = y_3^{ab}$ for all $\vert a \vert, \vert b \vert \leqslant 1$, the relation $x_p=y_3$, and the relations $[x_i^a,y_j^b]$ for all $i \in \{1,\ldots, p-1\}$, $j \in \{1,2\}$ and $-1\leqslant a,b \leqslant 1$.

The following consequences of Proposition \ref{prop:compactpresentations} is used in \S\ref{sec:upper-bounds} when we give upper bounds on the Dehn functions of $L_p \times_Z L_3$ and $L^\lrcorner_p \times_Z L_{3}$.

\begin{remark}\label{rem:identitis-lpl32} 
    The results are stated for $L_p^\lrcorner \times_Z L_{3}$ for all $p \geqslant 5$ but the reader should note that the same results hold for $L_{p} \times_Z L_{3}$  for all $p \geqslant 3$ with the same area estimates, the only minor change is that there is no error term of the form $[y_1, y_{p-1}]^{\pm 1}$.
\end{remark}

\begin{corollary}\label{cor:commutingxkxl}
    Let $p \geqslant 5$.
    \begin{enumerate}
        \item For all $b \in \R$ the identities $$[x_1, x_j ^b ] =_\calp \begin{cases}
        x_{j+1}^b, \ & \text{if} \ j>2, \\
        x_3^b \cdot [y_1^{\tilde{b}}, y_{p-1}^{\tilde{b}}]^{\pm 1}, \ & \text{if} \ j=2.
        \end{cases}$$ hold in $L^\lrcorner_p \times_Z L_{3}$ with area $\lesssim_p |b|^2$, where $\tilde{b} \in \R$ with $|\tilde{b}| \lesssim_p |b|$.
        \item For all $b \in \R$ the identities 
        $$
        [x_1^{-1}, x_j^b] =_\calp \begin{cases}
            x_{j+1}^{-b} \ldots x_p^{-b}, \ & \text{if} \ j>2,\\
            x_3^{-b} \ldots x_p^{-b} \cdot [y_{1}^{\tilde{b}}, y_{p-1}^{\tilde{b}}]^{\pm 1}, \ & \text{if} \ j=2.
        \end{cases} 
        $$
        hold in $L^\lrcorner_p \times_Z L_{3}$ with area $\lesssim_p |b|^2$ where $\tilde{b} \in \R$ with $|\tilde{b}|\lesssim_p |b|$.
    \end{enumerate}
\end{corollary}
\begin{proof}
    In case (1) the identity for $j>2$ is a direct consequence of Proposition \ref{prop:compactpresentations} since $\binom{1}{m}= 0$ for all $m \in \mathbf Z_{\geqslant 2}$. Now suppose $j=2$. For $\beta := \lfloor b \rfloor - b$ we have that the identities  {\small$$x_2^{-b}x_1 =_F x_2 ^{\beta} x_2^{-\lfloor b \rfloor} x_1 =_\calp x_2 ^{\beta} x_1 (x_2 x_3^{-1})^{-\lfloor b \rfloor}$$}%
    hold in $L^\lrcorner_p \times_Z L_{3}$  with area $\lesssim_p |b|$. By Proposition \ref{prop:compactpresentations} it follows that 
    {\small$$
    x_2 ^{\beta} x_1 (x_2 x_3^{-1})^{-\lfloor b \rfloor} =_\calp x_1 x_3^{-\beta} x_p ^{\binom{\beta}{2}} x_2^\beta (x_2 x_3^{-1})^{-\lfloor b \rfloor}
    $$}%
    holds in $L^\lrcorner_p \times_Z L_{3}$ with area $\lesssim_p |b|$.
    Finally, since  $\langle x_2, x_3, y_1, y_{p-1} \rangle$ generate a subgroup isomorphic to the integral 5-Heisenberg group, it follows that 
    {\small$$
    x_1 x_3^{-\beta} x_p ^{\binom{\beta}{2}} x_2^\beta (x_2 x_3^{-1})^{-\lfloor b \rfloor} =_F  x_1 x_3^{-\beta} x_2^\beta (x_2 x_3^{-1})^{-\lfloor b \rfloor} x_p ^{\binom{\beta}{2}} =_\calp x_1 x_3^{b} \cdot [y_1 ^{\tilde{b}}, y_{p-1}^{\tilde{b}}] \cdot x_2^{-b}
    $$}%
    holds in $L^\lrcorner_p \times_Z L_{3}$ with area $\lesssim_p |b|^2$ and $|\tilde{b}|\lesssim_p |b|$.

    The proof of Case (2) is similar, only that now $\binom{-1}{m}=\pm 1$ leading to the slightly modified formulas. We briefly explain the case $j=2$, the case $j>2$ being similar but without the error term $[y_{p-1}^{\tilde{b}}, y_1^{\tilde{b}}]$. It follows from Proposition \ref{prop:compactpresentations} that {\small$$x_2^bx_1^{-1}= x_1^{-1} x_2^b x_3^b \ldots x_p^{b -\binom{b}{2}}$$}%
    holds in $L_p^\lrcorner \times_Z L_{3}$. Finally by applying the identity {\small$x_p^{\binom{b}{2}} =_\calp [y_{p-1}^{\tilde{b}}, y_1^{\tilde{b}}]$}, which has area $\lesssim_p |b|^2$ for suitable $|\tilde{b}| \lesssim |b|$, we obtain the desired statement.
    \end{proof}

\begin{corollary}\label{cor:oneandjnotequalsoneplusj}
    Let $p \geqslant 5$. Let $a,b \in \R$. 
    
    The identities
    {\small\begin{equation}
        [x_{1}^{a},x_{j}^{b}] =_{\calp} \begin{cases}
            x_{j+1}^{- b} \cdot [x_{j+1}^{-b},x_{1}^{a-1}] \cdot [x_{1}^{a-1},x_{j}^{b}],\\[10pt]
            \left(\prod\limits_{i=j+1}^{p} x_i^{-b} \cdot [x_i^{-b}, x_1^{a+1}] \right) \cdot [x_1^{a+1},x_j^{b}].
        \end{cases}
    \end{equation}}%
    hold in $L^\lrcorner_p \times_Z L_{3}$ for $2 < j \leqslant p-1$ with area $\lesssim_p |b|^2$. 
    
    Moreover, there exists $\tilde{b} \in \R$ with $|\tilde{b}| \lesssim_p |b|$, such that the identities 
    {\small\begin{equation}\label{eq:omega2}
    [x_{1}^{a},x_{2}^{b}] =_\calp 
    \begin{cases}
        x_{3}^{-b} \cdot [x_{3}^{- b},x_{1}^{a-1}] \cdot [x_{1}^{a-1},x_{2}^{b}] \cdot [y_{1}^{\tilde{b}},y_{p-1}^{\tilde{b}}]^{\pm 1}, \\[10pt]
        \left( \prod\limits_{i=3}^{p} x_i^{-b}[x_i^{-b}, x_1^{a+1}] \right) \cdot [x_1^{a+1}, x_2^{b}] \cdot [y_{1}^{\tilde{b}}, y_{p-1}^{\tilde{b}}]^{\pm 1}.
    \end{cases}
    \end{equation}}%
    hold in $L^\lrcorner_p \times_Z L_{3}$ with area $\lesssim_p |b|^2$.
\end{corollary}
\begin{proof}
    We prove the second identity in \eqref{eq:omega2}. The other cases are analogous, but slightly simpler. It follows from Corollary \ref{cor:commutingxkxl} (2) that the identity 
    {\small$$
    x_1^{-a}x_2^{-b}x_1^{a}x_2^{b} =_\calp x_1^{-(a+1)} x_3^{-b} \ldots x_p^{-b}\cdot [y_{1}^{\tilde{b}}, y_{p-1}^{\tilde{b}}]^{\pm 1}x_2^{-b} x_1^{a+1}x_2^{b}
    $$}%
    holds in $L_p^\lrcorner \times_Z L_{3}$ with area $\lesssim_p |b|^2$ where $|\tilde{b}| \lesssim_p |b|$. Therefore, since the free identity 
    {\small
    $$
    x_1^{-(a+1)}x_i^{-b} =_F x_i^{-b} \cdot [x_i^{-b}, x_1^{a+1}] \cdot x_1^{-(a+1)}
    $$
    }%
    holds for all $i\geqslant 3$, we can shuffle $x_1^{-(a+1)}$ to the right in front of the suffix-word $x_2^{-b}x_1^{a+1}x_2^b$ to obtain the desired identity with area $\lesssim_p |b|^2$.
\end{proof}

We record a rewriting of the identities of Corollary \ref{cor:commutingxkxl} and Corollary \ref{cor:oneandjnotequalsoneplusj} that are used in \S\ref{sec:upper-bounds}. For this we introduce the following notation.

\begin{notation}\label{notation:omega-long}
	For $j \geqslant 2$, $k \geqslant 1$, and $\underline{n} := (n_1, \ldots, n_k) \in \R^k$ we denote by
	\[
		\Omega_k^j (\underline{n}) := [x_1 ^{n_1}, \ldots , x_1 ^{n_{k-1}}, x_j ^{n_k}].
	\]
\end{notation}

\begin{addendum}\label{add:omega3s}
     Let $p\geqslant 5$ and let $3 \leqslant i \leqslant p$. Since for all $i \geqslant j$ and for all $b \in \R$ the identity $$x_i^{-b} =_\calp \Omega_{i-j}^{j+1}(1,\ldots, 1, -b)$$ holds in $L^\lrcorner_p \times_Z L_{3}$  with area $\lesssim_p |b|^2$, we can rewrite the identities from Corollary \ref{cor:commutingxkxl} (2) as
   {\small$$
    [x_1^{-1}, x_j^b] =_\calp \begin{cases}
            \Omega_{1}^{j+1}(-b) \cdot \ldots \cdot \Omega_{p-j}^{j+1}(1, \ldots, 1, -b), \ & \text{if} \ j>2,\\[8pt]
            \Omega_{1}^{3}(-b) \cdot \ldots \cdot \Omega_{p-2}^{3}(1, \dots, 1, -b) \cdot [y_{p-1}^{\tilde{b}}, y_1^{\tilde{b}}]^{\pm 1}, \ & \text{if} \ j=2,
    \end{cases} 
    $$}%
    with area $\lesssim_p |b|^2$. 
    
    Similarly, we can rewrite the two identities corresponding to increasing the $x_1$-exponent by one in Corollary \ref{cor:oneandjnotequalsoneplusj} as
    {\small\begin{align*}
    [x_1^a, x_j^b] &=_\calp\left(\prod\limits_{i=j+1}^{p} \Omega_{i-j}^{j+1}(1, \ldots, 1, -b) \cdot \Omega_{i-j+1}^{j+1}(a+1,1, \ldots, 1, -b)^{-1} \right) \cdot [x_1^{a+1},x_j^{b}],
    \intertext{for $j >2$, and}
    [x_1^a, x_2^b] &=_\calp \left( \prod\limits_{i=3}^{p} \Omega_{i-2}^3(1, \ldots, 1, -b)\cdot\Omega_{i-1}^{3}(a+1,1, \ldots, 1, -b)^{-1} \right) \cdot [x_1^{a+1}, x_2^{b}] \cdot [y_{p-1}^{\tilde{b}}, y_{1}^{\tilde{b}}]^{\pm 1},
    \end{align*}}%
    each with area $\lesssim_p |b|^2$.
\end{addendum}

We finish this section by recalling the following well-known result that is used in \S\ref{sec:upper-bounds}.

\begin{lemma}[{\cite[Lemma 5.7]{lipt}}]\label{lem:freeidentities}
    Let $G$ be a group, and let $u,v,w$ be words in some generating set of $G$. The following free identities hold in $G$
    \begin{enumerate}
        \item $[u\cdot v, w] =_F [u,w]^v \cdot [v,w]$. 
        \item $[u, v \cdot w] =_F [u, w] \cdot [u,v]^w$.
        \item $u^w =_F u \cdot [u,w]$.
    \end{enumerate}
\end{lemma}

\subsubsection{Groups of low dimension}
\label{subsubsec:grp-low-dim-prelim}

Low-dimensional nilpotent Lie algebras over any field of characteristic not equal to $2$ have been classified up to dimension 6 by de Graaf in \cite{deGraafclass}. We recall below the list of nilpotent Lie algebras of dimension at most $5$ with one-dimensional centres, and the names as in \cite[\S4]{deGraafclass}, where we write $\mathfrak l_{d,i}$ instead of $L_{d,i}$. ($d$ is the dimension and $i$ an integer used as an index for the list in any given dimension).
\begin{itemize}
    \item The three-dimensional $\mathfrak l_{3,2} = \mathfrak l_3$, namely the first Heisenberg Lie algebra.
    \item The four-dimensional $\mathfrak l_{4,3} = \mathfrak l_4$.
    \item The five-dimensional $\mathfrak l_{5,4} = \mathfrak l_{3,2} \times_Z \mathfrak l_{3,2}$,  namely the second Heisenberg Lie algebra.
    \item The five-dimensional $\mathfrak l_{5,5}$ with Lie brackets
    \[ [X_1,X_2] = X_3, ~[X_1,X_3] = [X_2,X_5] = X_4; \]
    Note that $\{X_1,X_2,X_3,X_4 \}$ generate a $\mathfrak l_{4,3}$ subalgebra which is not an ideal.
    \item The five-dimensional filiform Lie algebras $\mathfrak l_{5,7} = \mathfrak l_5$ and $\mathfrak l_{5,6} = \mathfrak l^{\lrcorner}_5$.
\end{itemize}

These Lie algebras are either filiform of the form $\mathfrak l_p$ or $\mathfrak l_p^\lrcorner$, or central products of such algebras, with the exception of $\mathfrak l_{5,5}$. We provide a lattice and a compact presentation of the latter separately below.

\begin{proposition}
\label{prop:pres-lattice-L55}
    The simply connected Lie group with Lie algebra $\mathfrak l_{5,5}$ contains a lattice with presentation
    \begin{equation*}
        \left\langle x_1,x_2,x_3,x_4,x_5 \mid [x_1,x_2] = x_3, [x_1,x_3] = x_4, [x_2,x_5]=x_4, [x_i,x_j]=1~ \text{for all other}\ \{i,j\} \right\rangle
    \end{equation*}
    and the subgroup generated by $x_1,\ldots, x_4$ is a lattice in the Lie subgroup with Lie algebra spanned by $X_1,\ldots, X_4$.
    \end{proposition}

    \begin{proof}
        Set $x_i = \exp X_i$ for $1 \leqslant i \leqslant 5$. A computation using the Baker--Campbell--Hausdorff formula, very similar to that in \eqref{eq:first-BCH-induction}, yields that
        {\small
        \begin{equation*}
            \log(x_1 x_2) = X_1 +X_2 +\frac{1}{2} X_3 + \frac{1}{12} X_4\; \text{and}\,  \log(x_1^{-1} x_2^{-1}) = -X_1 -X_2 +\frac{1}{2} X_3 - \frac{1}{12} X_4\
        \end{equation*}
        }%
        so that $X_3 = \log x_3 = \log [x_1,x_2]$, while
        {\small
        \begin{equation*}
            \log(x_1 x_3) = X_1 +X_3 +\frac{1}{2} X_4\; \text{and}\,  \log(x_1^{-1} x_3^{-1}) = -X_1 -X_3 +\frac{1}{2} X_4,
        \end{equation*}
        }%
        {\small
        \begin{equation*}
            \log(x_2 x_5) = X_2 +X_5 +\frac{1}{2} X_4\; \text{and}\,  \log(x_2^{-1} x_5^{-1}) = -X_2 -X_5 +\frac{1}{2} X_4,
        \end{equation*}
        }%
        so that $X_4 = \log x_4 = \log[x_1,x_3] = \log[x_2,x_5]$.
    \end{proof}

\begin{proposition}
    The group $L_{5,5}$ admits the compact presentation $\mathcal P_{5,5} = \langle S \mid R_{5,5} \rangle$, where
    \[ S = \{ x_i^a: i \in \{1, \ldots, 5\}, \ a \in [-1,1] \}\]
    and 
    \begin{align*}
        R_{5,5} = \big\{ x_i^a x_i^b x_i^{-a-b}&, \ i \in \{1, \ldots 5 \}, \ [x_1^a, x_2^b] = x_3^{ab} x_4^{-\binom{a}{2}b}, \\
        & \quad  [x_1^a, x_3^b] = x_4^{ab}, \ [x_2^a, x_5^b]=x_4^{ab}, \ -1 \leqslant a,b \leqslant 1,  \\
        & \quad [x_j^a, x_k^b] = 1\ \text{for all other}\ \{j,k\}\ \text{and for all}\ a,b \in [-1,1] \big\}.
    \end{align*}
    \begin{proof}
        Since the elements $x_1, \ldots, x_5$ form a Malcev basis for $L_{5,5}$, this follows from Proposition~\ref{prop:pres-lattice-L55} using the same arguments as in the proof of \cite[Proposition 5.2]{lipt}.
    \end{proof}
\end{proposition}

\subsubsection{Reduction to products of efficient words}\label{sec:efficient}

We start this section by establishing Proposition \ref{prop:reductiontrick} which allows us to obtain upper bounds for Dehn functions by bounding the area of particular words belonging to so-called \emph{efficient sets} (Definition \ref{def:efficient}). This idea originates from an observation of Gromov \cite[$5.A_{3}''$]{AsInv}, for details we refer to \cite{CornulierTesseraDehnBaums} and \cite{lipt}. We then construct efficient sets for the groups  $L_p \times_Z L_{3}$ and $L^\lrcorner_p \times_Z L_{3}$ (Corollary \ref{cor:efificentcentral}) which we  require in \S\ref{sec:upper-bounds} when we prove the upper bounds on their Dehn functions.

Throughout this section $G$ denotes a compactly presented group, $\calp:=\langle S \mid R \rangle$ a compact presentation of $G$, and $F_S$ the free group generated by $S$. Given a subset $\mathcal{F} \subset F_S$ and an integer $k \geqslant 1$, we denote by $\mathcal{F}[k]$ the set of concatenations of at most $k$ elements of $\mathcal{F}$.

\begin{definition}\label{def:efficient}
    Given an integer $r \geqslant 1$, a subset $\mathcal{F} \subset F_S$ is called \emph{$r$-efficient with respect to $\calp$}, if there exists a constant $C \geqslant 1$ such that for every $w \in F_S$ there exists $w' \in \mathcal{F}[r]$ such that $w =_\calp w'$ and $|w'|_S \leqslant C|w|_S$.
\end{definition}

Note that this generalises the definition of efficient sets given in \cite{CornulierTesseraDehnBaums}. What they call efficient corresponds to a $1$-efficient set in our definition.

\begin{proposition}\label{prop:reductiontrick}
    Let $G$ be a compactly presented group and $\calp:=\langle S \mid R \rangle$ a compact presentation for $G$. Suppose that $\mathcal{F}$ is $r$-efficient with respect to $\calp$ for some $r\geqslant 1$ and that there exists $d>1$ such that for all $k \geqslant 1$ and all $n \geqslant 0$, the area of each null-homotopic word in $\mathcal{F}[k]$ of length at most $n$ is $\lesssim_{k} n^d$. Then, the Dehn function of $G$ satisfies $\delta_G(n) \preccurlyeq n^d$.
\end{proposition}
\begin{proof}
    This is an immediate consequence of applying \cite{CornulierTesseraDehnBaums} to the 1-efficient set $\mathcal{F}[r]$.
\end{proof}

We now explain the existence of $O_{p}(1)$-efficient sets for the groups $L_p \times_Z L_{3}$ and $L^\lrcorner_p \times_Z L_{3}$ with respect to the compact presentations obtained from Proposition \ref{prop:compactpresentations} and Lemma \ref{lem:central-prod-is-cglc}. To establish these results we first need to recall some notation from \cite{lipt}: 
$$
\Sigma := \{ x_1 ^{a_1}, x_2 ^{a_2} \mid |a_1|, |a_2| \leqslant  1\}, \quad \mathcal{F} := \{ s^n \mid s \in \Sigma, n \in \N \},
$$
$$
T := \{ x_1 ^{a_1}, x_2 ^{a_2}, y_{1}^{a_3}, y_{p-1}^{a_4} \mid |a_1|, |a_2|, |a_3|, |a_4| \leqslant 1\}, \ \text{and} \ \mathcal{G} := \{ s^n \mid s \in T, n \in \N \}.
$$
\begin{proposition}[cf. {\cite[Proposition 5.11]{lipt}}]\label{prop:efficientset}
    The subset $\mathcal{F}$ is $O_{p}(1)$-efficient with respect to the compact presentation $\mathcal{P}^\lrcorner_p$ (respectively $\mathcal{P}_p$) of $L^\lrcorner_p$ for all $p \geqslant 5$ (respectively $L_p$ for all $p \geqslant 3$).
\end{proposition}
\begin{proof}
    The proof is analogous to the one of \cite[Proposition 5.11]{lipt}, so we only sketch it here. We first observe that the set $\Sigma$ is a generating set for $L_p$ (respectively $L^\lrcorner_p$), because for each $b \in \R$ and $3 \leqslant i \leqslant p$ the element $x_i^b$ can be expressed as a word $u_i$ in $\mathcal{F}[O_p(1)]$ such that {\small$|u_i| = O_p(b^{\frac{1}{i-1}}) + O_p(1)$}. We then show that every word $w$ in the letters $x_1$ and $x_2$ can be expressed in normal form $x_1^{b_1} \ldots x_p^{b_p}$ with $|b_1|\lesssim |w|$ and $|b_i|\lesssim |w|^{i-1}$ for $2\leqslant i \leqslant p$. Finally, combining these observations, we can express each of the words {\small$x_j^{b_j}$}, in the normal form of $w$, as words $u_j\in\mathcal{F}[O_p(1)]$. Thus, $u := u_1 \ldots u_p =_\calp w$ with {\small$|u| = O_p(|w|) + O_p(1)$}, completing the proof.
\end{proof}

\begin{corollary}[cf. {\cite[Corollary 5.12]{lipt}}]\label{cor:efificentcentral}
    For all $p \geqslant 5$ the subset $\mathcal{G}$ is $O_p(1)$-efficient with respect to the compact presentation $\calp$ of $L^\lrcorner_p \times_Z L_{3}$ (respectively of $L_p \times_Z L_{3}$ for all $p\geqslant 3$).     
\end{corollary}
\begin{proof}
    Proposition \ref{prop:efficientset} gives us an efficient set for $L^\lrcorner_p$ for all $p \geqslant 5$ (respectively $L_p$ for all $p \geqslant 3$). In particular, we have an efficient set for the product of $L^\lrcorner_p \times L_{3}$ (respectively $L_p \times L_{3}$ for all $p \geqslant 3$) and therefore for the quotient  $L^\lrcorner_p \times_Z L_{3}$ (respectively of $L_p \times_Z L_{3}$ for all $p\geqslant 3$).
\end{proof}

\section{Lower bounds}\label{sec:lowerbounds}

We begin this section by explaining a general principle to obtain lower bounds for Dehn functions or filling invariants in degree one. 
There are two ways to state this principle, one combinatorial (in terms of words, on the left of \Cref{fig:forms-distorsion}), the other differential-geometric (in terms of loops and differential forms, on the right of \Cref{fig:forms-distorsion}); those are equivalent via the correspondence between $2$-forms and central extensions.


\begin{figure}
\begin{tikzpicture}
\draw
    (-4,0)  node[align=center, minimum size=1.5cm]		(A) {Maximally distorted \\ central extensions of $G$}
    (4,0)  node[align=center, minimum size=1.5cm]     (B) {Certain cohomology classes \\ $[d\beta] \in H^2(G,\mathbf R)$, \\ $\beta$ not $G$-invariant}
    (0,-5) node[align=center, minimum size=1cm]		(C) {Lower bound on the Dehn function};

  \draw[thick,<->] (A) to (B);
  \draw[thick,->] (A) to (C);
  \node[align=center, node font=\small, left] at (-2.2,-2.5) {Words  in $G$ lifting  \\ to the highly distorted kernel \\ of the extension have large area};
  \draw[thick,->] (B) to (C);
  \node[align=center, node font=\small, right] at (2.2,-2.5) {Integration of $\beta$ \\ along certain loops is large};
\end{tikzpicture}
\caption{Relation between maximally distorted central extensions and second cohomology classes.}
\label{fig:forms-distorsion}
\end{figure}

The main idea in this section is that the right side of \Cref{fig:forms-distorsion} is flexible enough so that it can be strengthened by replacing $\beta$ with a form with bounded exterior derivative.
This turns out to sometimes yield better lower bounds on the Dehn function, since it can pick up area contributions that cancel when working with invariant forms, as was already observed in \cite{lipt}.  However, our method for evaluating the required integrals is quite different. Indeed, instead of explicitly integrating the form $\beta$ in coordinates along a loop, we adopt the viewpoint via distortion provided by the left side of \cref{fig:forms-distorsion} to evaluate the integral. This is formalised by \Cref{lem:main-ingredient-lower-lie} in \S\ref{sec:main-lemma}.

In \S\ref{subsec:non-inv}, we then implement this method to obtain lower bounds on the groups considered in this paper. For this we first construct explicit 1-forms with piecewise invariant derivative. We then integrate them along suitable loops, producing the desired lower area bounds. We refer to \S\ref{subsec:non-inv} for an explanation of the heuristics of our choice of forms and loops.

Finally, in \S\ref{subsec:dehn-lower-lie} we pass to a more general setting and prove a lower bound on Dehn functions coming from forms with invariant derivatives corresponding to maximally distorted central extensions (see \cref{prop:dehn-lower-lie}). This is of independent interest and generalises the analogous result for groups with lattices, which is classical, see for instance \cite{BW97} and also \cite{lipt}.

\subsection{An integral identity in central extensions}\label{sec:main-lemma}

\begin{lemma}
    \label{lem:main-ingredient-lower-lie}
    Let $L > 0$.
    Let $G$ be a simply connected nilpotent Lie group.
    Let $\omega \in Z^2(G,\mathbf R)$ and let 
    \[ 1 \to \mathbf R \overset{\iota}{\longrightarrow} \widetilde G \overset{\pi}{\longrightarrow} G \to 1 \]
    be the central extension associated to $\omega$.
    Let $V$ be a left-invariant subbundle of $T\widetilde G$, transverse to the cosets of $\iota(\mathbf R)$.
    Let $\lambda \colon [0,L] \to G$ be a piecewise $C^1$ loop in $G$ and let $\widetilde \lambda \colon [0,L] \to \widetilde G$ be the lift of $\lambda$ in $\widetilde G$ with derivative in $V$ whenever it is defined, and $\widetilde \lambda (0) = 1$.
    Further, let $\alpha$ be a one-form on $G$ such that $d\alpha = \omega$.
    Then 
    \[ \int_{\lambda} \alpha = \iota^{-1}(\widetilde \lambda(L)). \]
\end{lemma}
\begin{proof}
    Let $Y_1, \ldots, Y_n$ form a basis of $\operatorname{Lie}(G)$, and let $\widetilde Y_1, \ldots, \widetilde Y_n$ be the frame of $V$ that lifts them. A basis of $\operatorname{Lie}(\widetilde G)$ is then given by $\widetilde Y_1, \ldots, \widetilde Y_n$ and $Z$, where $Z$ is such that $\iota(\mathbf R)$ is the one-parameter subgroup generated by $Z$; if we decompose $[Y_i, Y_j] = \sum \beta^{ij}_k Y_k$ for all $i,j$ then the Lie brackets in $\widetilde {\mathfrak g}$ are given by $[\widetilde Y_i, \widetilde Y_j] = \sum \beta^{ij}_k \widetilde Y_k + \omega(Y_i, Y_j)Z$.
    Denote $\widetilde \alpha = \pi^\ast \alpha$. Finally, define $\mu(t) = \iota(t)$ for $t \in [0, \iota^{-1}(\widetilde \lambda (L))]$.
    
    In $\widetilde G$, $\pi^\ast \omega = - d\zeta$ for the left-invariant form $\zeta$ on $\widetilde G$ such that $\ker(\zeta) = V$ and $\langle \zeta, Z \rangle = 1$. In particular, $d(\widetilde \alpha + \zeta) = 0$ and therefore, 
    \begin{equation*}
        \int_\lambda \alpha = \int_{\widetilde \lambda} \widetilde \alpha \notag 
         \stackrel{(1)}{=} \int_{\widetilde \lambda} \widetilde \alpha + \zeta 
         \stackrel{(2)}{=} \int_{\mu} \widetilde \alpha + \zeta. \notag
    \end{equation*}
    
Let us justify the previous equalities.
In (1) we used that $\int_{\widetilde \lambda} \zeta = 0$, as $V = \ker(\zeta)$ and $\widetilde \lambda$ is almost everywhere tangent to $V$.
In (2) we used that $d(\widetilde \alpha + \zeta) = 0$, and that $\widetilde \lambda$ and $\mu$ are homotopic with fixed endpoints.
Now, 
\[ \int_{\mu} \widetilde \alpha + \zeta = \int_\mu \zeta = \iota^{-1}(\mu(\widetilde \lambda(L))),\]
where we used successively that $Z \in \ker \widetilde \alpha$, and that $\mu'(t) = Z$ for all $t$ in $[0, \iota^{-1}(\widetilde \lambda (L))]$. \qedhere
\end{proof}

\subsection{Lower bound from non-invariant forms}\label{subsec:non-inv}

Throughout this subsection, $G$ denotes a central product of the form
\[ G = K \times_Z H \]
where $K$ is either $L_p$ for $p \geqslant 3$ or $L_{p}^{\lrcorner}$ for $p \geqslant 5$, and $H$ is a simply connected nilpotent group with a one-dimensional centre.
We prove the following.

\begin{proposition} \label{prop:lower-bound-general}
    Let $G$ be as above.
    Then 
    $\delta_G(n) \succcurlyeq n^{p-1}$.
\end{proposition}

Before we give a proof, we illustrate our strategy by using it to give a lower area bound on a loop in the Heisenberg group $L_3$ by integrating its projection to $\mathbf{R}^2$ along a suitable form with piecewise invariant exterior derivative.

Consider the plane $\mathbf R^2$ with coordinate frame $(x_1,x_2)$. 
The area form $dx_1 \wedge dx_2$ has a primitive $\alpha_0 = x_1 dx_2$. The integral of $\alpha_0$ along a positively oriented circle of diameter $\ell$ in the half-plane $\{ x_1 > 0 \}$ is of order $\ell^2$, while the integral of $\alpha_0$ along a negatively oriented circle $\Gamma^{-}$ of diameter $\ell$ in the half-plane  $\{ x_1 < 0 \}$ is of order $-\ell^2$. This means that if we form a loop by connecting the two circles via a path as in \Cref{fig:heuristic} their signed enclosed areas cancel and thus the integral of this loop along $\alpha_0$ vanishes. This means that the loop lifts to a loop in $L_3$, whose area can, however, not be detected by the form $\alpha_0$. To obtain a lower area bound on this loop we instead consider the form $\alpha_1=\operatorname{sign}(x_1) x_1 dx_2$. Indeed, the integral of $\alpha_1$ along this loop is large: it is about $(1+1) \ell^2 = 2\ell^2$ instead of $(1-1)\ell^2 = 0$. 

In this approach, we used a distorted central extension of a quotient (namely, $\mathbf R^2$) in order to build a differential form and a loop in the Heisenberg group, with quadratic filling. This is not optimal, but to prove \cref{prop:lower-bound-general} we consider a case where the same strategy does yield an optimal lower bound. For this we replace the Heisenberg group by $G$, as above, its quotient $\mathbf R^2$ by a filiform group, the form $\alpha_0$ by a form $\beta_0$ such that $d\beta_0$ defines a maximally distorted central extension of the quotient, the coordinate $x_1$ by a function $u_1$, and the circles $\Gamma_{\pm}$ by loops corresponding to $\Omega$-words defined in \cref{notation:omega-long}. The loop we consider is then a natural generalization of the lift of the loop of Figure~\ref{fig:heuristic} to the Heisenberg group (See Figure~\ref{fig:enter-label} in the case $p=3$).

\begin{figure}
    \centering
    
\begin{tikzpicture}[line cap=round,line join=round,>=angle 45,x=0.7cm,y=0.7cm]
\clip(-10.,-4.5) rectangle (10.,3.);
\draw [shift={(7.,0.)},line width=1.pt]  plot[domain=-2.89:2.89,variable=\t]({1.*2.06*cos(\t r)+0.*2.06*sin(\t r)},{1.*2.06*sin(\t r)});
\draw [shift={(-7.,0.)},line width=1.pt]  plot[domain=0.244:6.03,variable=\t]({1.*2.06*cos(\t r)},{1.*2.06*sin(\t r)});
\draw [line width=1.pt] (5.,0.5)-- (0.,0.);
\draw [line width=1.pt] (2.39,0.239) -- (2.48,0.372);
\draw [line width=1.pt] (2.39,0.239) -- (2.512,0.127);
\draw [line width=1.pt] (0.,0.)-- (-5.,-0.5);
\draw [line width=1.pt] (-2.602,-0.260) -- (-2.512,-0.127);
\draw [line width=1.pt] (-2.602,-0.260) -- (-2.487,-0.372);
\draw [line width=1.pt] (0.995,-0.0995)-- (5.,-0.5);
\draw [line width=1.pt] (3.099,-0.3099) -- (2.98,-0.42);
\draw [line width=1.pt] (3.09,-0.309) -- (3.00,-0.177);
\draw [line width=1.pt] (-5.,0.5)-- (-0.995,0.0995);
\draw [line width=1.pt] (-2.895,0.2895) -- (-3.00,0.177);
\draw [line width=1.pt] (-2.895,0.2895) -- (-2.985,0.422);
\draw [line width=1.pt] (9.06,0.)-- (9.2,-0.2);
\draw [line width=1.pt] (9.06,0.)-- (8.91,-0.17);
\draw [line width=1.pt] (-9.06,0.)-- (-9.2,-0.2);
\draw [line width=1.pt] (-9.06,0.)-- (-8.915,-0.174);
\draw [->,line width=1.pt] (5.,-3.) -- (9.,-3.);
\draw [->,line width=1.pt] (9.,-3.) -- (5.,-3.);
\draw [->,line width=1.pt] (-9.,-3.) -- (-5.,-3.);
\draw [->,line width=1.pt] (-5.,-3.) -- (-9.,-3.);
\draw (7,-3.5) node[anchor=center] {$\ell$};
\draw (-7,-3.5) node[anchor=center] {$\ell$};
\draw (7,2.5) node[anchor=center] {$\Gamma^+$};
\draw (-7,2.5) node[anchor=center] {$\Gamma^{-}$};

\draw [line width=1.pt,dash pattern=on 2pt off 2pt] (-7.,0.) circle (2.061);
\draw [line width=1.pt,dash pattern=on 2pt off 2pt] (7.,0.) circle (2.061);
\draw (5.5,0.123) node[stuff_fill, anchor=north west] {$\int_{\Gamma^+} x_1dx_2 \approx + \ell^2$};
\draw (-7.958,0.123) node[stuff_fill, anchor=north west] {$\int_{\Gamma^-} -x_1dx_2 \approx + \ell^2$};
\begin{scriptsize}
\draw [fill=black] (0.,0.) circle (2.0pt);

\draw [line width = 1pt, dash pattern = on 2pt off 4pt] (0,2) -- (0,-2);

\draw (0,0) node[anchor=south west, rotate = 90] {$\{ x_1 < 0 \}$};

\draw (1,2) node[anchor=south east, rotate = 90] {$\{ x_1 > 0 \}$};

\end{scriptsize}
\end{tikzpicture}
    \caption{Heuristic for the choice of form and loop (with $p=3$).}
    \label{fig:heuristic}
\end{figure}

We now proceed with the proof of \Cref{prop:lower-bound-general}.
Let $\mathfrak g$ be the Lie algebra of $G$.
A basis for $\mathfrak g$ is given by
\[ X_1, \ldots, X_{p-1}, X_p = Z = Y_q, Y_1, \ldots, Y_{q-1} \]
where $X_1,\ldots, X_{p-1}$ is an adapted basis of $L_p$ or $L_p^\lrcorner$ and $Y_1, \ldots, Y_q$ is a (copy of) a basis of $H$.
Let $\xi_1, \ldots, \xi_{p-1},  \zeta, \eta_1,\ldots, \eta_{q-1},$ be the basis of $\operatorname{Hom}(\mathfrak g,\mathbf R)$ dual to the former one. Then
{\small\[
\begin{cases}
d \xi_1 = 0 \vspace{0.5ex}\\
d \xi_2 = 0\vspace{0.5ex} \\
d\xi_i = - \xi_1 \wedge \xi_{i-1} & 2 \leqslant i \leqslant p-1 \vspace{0.5ex}\\
d \eta_j = - \sum e_{ik}^j \eta_i \wedge \eta_k & 1 \leqslant j \leqslant q -1 \vspace{0.5ex}\\
d \zeta = - \xi_1 \wedge \xi_{p-1} - \sum f_{ik} \eta_i \wedge \eta_{k}, & \text{or} \ - \xi_1 \wedge \xi_{p-1} - \sum f_{ik} \eta_i \wedge \eta_{k} - \xi_2 \wedge \xi_3.
\end{cases}
\]}%

The left invariant form $\xi_1$ is closed, hence it has a (non-left-invariant) primitive in $G$, that we denote $u_1$, with the convention that $u_1(1) = 0$.

Consider the following pair of continuous one-forms on $G$:
\begin{equation}\label{eq:definition-of-beta0}
    \beta_0 =    u_1 \xi_{p-1} + \frac{u_1^2}{2} \xi_{p-2} + \frac{u_1^3}{6} \xi_{p-3} + \cdots  + 
    \frac{u_1^{p-2}}{(p-2)!} \xi_2 
\end{equation}
and 
\begin{equation}\label{eq:definition-of-beta1}
    \beta_1 = \operatorname{sign} (u_1)\beta_0.
\end{equation}
Note that $\beta_1$ is continuous because $\beta_0$ vanishes on $\{ u_1 = 0 \}$.
Moreover, one computes that
{\small\begin{align*}
d \beta_0 & = \xi_1 \wedge \xi_{p-1} - u_1 \xi_1 \wedge \xi_{p-2}   + u_1 \xi_1 \wedge \xi_{p-2} - \frac{u_1^2}{2} \xi_1 \wedge \xi_{p-3} \notag \\
& \quad + \cdots - \frac{u_1^{p-2}}{(p-2)!} \xi_1 \wedge \xi_2 + \frac{u_1^{p-2}}{(p-2)!} \xi_1 \wedge \xi_2 \\
& = \xi_1 \wedge \xi_{p-1},
\end{align*}}%
and then, outside of the subvariety $\{u_1 = 0 \}$, $\beta_1$ is smooth and its exterior derivative can be expressed as
\begin{equation*}
    d \beta_1 = \operatorname{sign}(u_1) \xi_1 \wedge \xi_{p-1}.
\end{equation*}

\begin{remark}
    $\beta_1$ is going to play the role of the form that was named $\beta_0$ in \cite[\S8.3]{lipt}; although we avoid using a linear representation of $L_p$ or $L_p^\lrcorner$.
\end{remark}

Next, we are going to construct a loop $\Lambda$ in $G$ (actually, in the first factor of the central product). This loop is the concatenation of four paths. Two of those paths are defined below, and the remaining two paths follow integral curves of $X_1$.

\begin{definition}
    \label{def:lp}

Let $\ell > 0$ be a parameter. For $3 \leqslant k \leqslant p$ we construct inductively a number $\ell_k$ and a curve $\lambda_k\colon [0,\ell_k] \to L_p$ as follows:
Set $\ell_3:=4\ell$, and define the following curve $\lambda_3\colon[0,\ell_3] \to L_p$:
{\small
\[ \lambda_3(t) = \begin{cases}
    \exp(tX_1) & 0 \leqslant t \leqslant \ell \vspace{0.5ex}\\
    \exp(\ell X_1)\exp( (t-\ell) X_2) & \ell \leqslant t \leqslant 2 \ell \vspace{0.5ex}\\
    \exp(\ell X_1) \exp(\ell X_2) \exp(-(t-2\ell) X_1) & 2 \ell \leqslant t \leqslant 3\ell \vspace{0.5ex}\\
    \exp(\ell X_1) \exp(\ell X_2) \exp(-\ell X_1) \exp(-(t-3\ell) X_2) & 3 \ell \leqslant t \leqslant 4\ell.
\end{cases} \]
}
Assuming that $\ell_k$ and $\lambda_k:[0,\ell_k] \to L_p$ have been constructed, define 
{\small
\begin{align*}
\ell_{k+1}& = 2(\ell_k+\ell),\text{ and} \\
\lambda_{k+1}(t) & = \begin{cases}
    \exp(tX_1) & 0 \leqslant t \leqslant \ell \vspace{0.5ex}\\
    \exp(\ell X_1) \lambda_k(t-\ell) & \ell \leqslant t \leqslant \ell + \ell_k \vspace{0.5ex}\\
    \exp(\ell X_1) \lambda_k(\ell_k) \exp(-(t-\ell - \ell_k) X_1) & \ell + \ell_k \leqslant t \leqslant 2\ell + \ell_k \vspace{0.5ex}\\
    \exp(\ell X_1) \lambda_k(\ell_k) \exp(-\ell X_1) \lambda_k(t - 2\ell - \ell_k)^{-1} & 2 \ell + \ell_k \leqslant t \leqslant 2\ell + 2\ell_k.
\end{cases}     
\end{align*}
}%
\end{definition}
\begin{remark}
\label{rem:lp}
    In \S\ref{sec:upper-bounds} we consider certain families of words in $L_p$ and $L^{\lrcorner}_p$, and the loop $\lambda_p$ is closely related to the word that we denote $\Omega^2_{p-1}(\ell, \ldots, \ell)$. 
    More precisely, one can embed the Cayley graph of the lattice generated by $x_1^\ell, x_2^\ell$ in $L_p$; then, $\lambda_p$ is the loop in $L_p$ corresponding to the word $\Omega^2_{p-1}(\ell, \ldots, \ell)$ seen as a loop in the Cayley graph. 
    It will follow from the definition of $\Omega_{p-1}(\ell, \ldots, \ell)$ as an iterated commutator that it has length $\ell_p$ exactly.
\end{remark}

It follows from the construction, by induction in $k$ (compare to \eqref{eq:log-xi}), that we have
\begin{equation*}
    \log \lambda_k (\ell_k) \equiv \ell^{k-1} X_k \; \operatorname{mod} \; \oplus_{j=k+1}^p \langle X_j \rangle
\end{equation*}
for all $k \in \{ 3, \ldots, p \}$. In particular, we record for further use
\begin{equation}
\label{eq:-lambda-p-ell-p}
    \lambda_p(\ell_p) = \exp(\ell^{p-1} X_p).
\end{equation}
Moreover, by induction on $k$, we get that 
\begin{equation}\label{eq:exact-lp}
    \ell_p = (2^{p-1}  + 2^{p-2} -2)\ell.
\end{equation}
Now, let $L > 0$ be a second parameter. Consider the following loop in $G$:
{\small
\[
\Lambda(t) =
\begin{cases}
    \exp(tX_1) & 0 \leqslant t \leqslant L, \vspace{0.5ex}\\
    \exp(L X_1) \lambda_p(t-L) & L \leqslant t \leqslant L + \ell_p, \vspace{0.5ex}\\
    \exp(L X_1+ \ell^{p-1} X_p) \exp(-(t-L-\ell_p)X_1) & L+\ell_p \leqslant t \leqslant 3L+\ell_p, \\
    \exp(-L X_1) \lambda_p(3L+2\ell_p-t) & 3L+\ell_p \leqslant t \leqslant 3L + 2\ell_p, \vspace{0.5ex}\\
    \exp(-L X_1) \exp((t-3L-2\ell_p)X_1) & 3L+2\ell_p \leqslant t \leqslant 4L + 2\ell_p, \vspace{0.5ex}\\
\end{cases}
\]
}%
Here we used \eqref{eq:-lambda-p-ell-p} and the fact that $[X_1,X_p] = 0$ so that 
\[ \exp(LX_1) \exp(\ell^{p-1}X_p)=\exp(LX_1 + \ell^{p-1} X_p) \] 
in order to simplify the expression of $\Lambda$.

\begin{lemma}
\label{lem:evaluate-integral-beta1}
    Let $\ell$, $L$ and $\Lambda\colon [0,4L+2\ell_p] \to G$ be as above. There is a constant $C_p$ only depending on $p$ such that if $L \geqslant C_p \ell$, then
    \begin{equation}
        \int_\Lambda \beta_1 = 2 \ell^{p-1}.
    \end{equation}
\end{lemma}

\begin{figure}
    \begin{tikzpicture}[line cap=round,line join=round,>=angle 45,x=1.0cm,y=1.0cm]
\clip(-4.5,-3.4) rectangle (5,4);
\draw [dash pattern=on 2pt off 2pt,domain=3.5:-4] plot(\x,{(-0-1*\x)/-2}) node[below]{$\langle X_1 \rangle$};
\draw [dash pattern=on 2pt off 2pt,domain=-3:2.8] plot(\x,{(-0-1*\x)/1}) node[below right]{$\langle X_2 \rangle$};
\draw [dash pattern=on 2pt off 2pt] (0,-2) -- (0,3.3) node[above]{$\langle Z \rangle$};
\draw [->,line width=1.0pt] (-3,-1.5)-- (-2.5,0);
\draw [->,line width=1.0pt] (-2.5,0)-- (-1.5,0.5);
\draw [->,line width=1.0pt] (-1.5,0.5)-- (-2,1);
\draw [->,line width=1.0pt](-2,1)-- (2,3);
\draw [->,line width=1.0pt] (0,0)-- (-2,-1);
\draw [->,line width=1.0pt] (-2,-1)-- (-3,-1.5);
\draw [->,line width=1.0pt] (2,3)-- (2.5,2.5);
\draw [->,line width=1.0pt] (2.5,2.5)-- (1.5,2);
\draw [->,line width=1.0pt] (1.5,2)-- (1,0.7);
\draw [->,line width=1.0pt] (1,0.7)-- (2,1.2);
\draw [->,line width=1.0pt](2,1)-- (0,0);
\draw [->,line width=1.0pt] (2,1)-- (0,0);
\begin{scriptsize}
\fill [color=black] (0,0) circle (1.5pt);
\fill [color=black] (-2,-1) circle (2.5pt) node[below right]{$LX_1$};
\fill [color=black] (-2,1) circle (2.5pt) node[above left]{$LX_1+\ell^{p-1}Z$};
\fill [color=black] (2,3) circle (2.5pt) node[above right]{$-LX_1+\ell^{p-1}Z$};
\fill [color=black] (2,1) circle (2.5pt)node[below right]{$-LX_1$};

\draw (2.7,-1.5) node[stuff_fill] {$\{ u_1 < 0 \}$};
\draw (0.2,-1.5) node[stuff_fill] {$\{ u_1 > 0 \}$};

\end{scriptsize}

\draw (1,0)-- (1,-2);
\draw (2,-3)-- (1,-2);
\draw (1,0)-- (2,-1);
\draw (2,-3)-- (2,-1);

\end{tikzpicture}
    \caption{The loop $\log \Lambda$ in $\mathfrak l_p$ for $p=3$. (For readability, the line segment that is the part of the loop ending at $-LX_1$ is pictured with a slight shift.) }
    \label{fig:enter-label}
\end{figure}

\begin{proof}
Let $\pi\colon L_p \to L_{p-1}$ be the quotient by the centre of $L_p$.
Note that\footnote{Here $\overline{\beta}_0$ is defined as in \eqref{eq:definition-of-beta0} with respect to the analogous choices of primitives and dual basis for $L_{p-1}$ viewed as a quotient of $L_p$.} $\beta_0 = \pi^\ast \overline \beta_0$, where $d \overline \beta_0$ is the closed one-form $\xi_1 \wedge \xi_{p-1}$ on $L_{p-1}$ determining the $(p-1)$-central extension $L_{p} \to L_{p-1}$. It then follows from Lemma~\ref{lem:main-ingredient-lower-lie} that
\begin{equation}
    \int_{\lambda_p} \beta_0 = \ell^{p-1}.
\end{equation}
For $s,e \in \mathbf R$ denote by $\sigma(s,e)$ the (oriented) portion of the one-parameter subgroup $\{ \exp(tX_1) \}$ between $s$ and $e$.
In view of the definition of $\Lambda$,
{\small\begin{align}
    \int_\Lambda \beta_1 & = \int_{\sigma(0,L)} \beta_1 + \int_{\exp (LX_1) \lambda_p} \beta_1 + \int_{\exp(\ell^{p-1}X_p)\sigma(L,-L)} \beta_1  \notag \\
    & - \int_{\exp(-LX_1)\lambda_p} \beta_1 + \int_{\sigma(-L,0)} \beta_1 \notag \\
    & = \int_{\exp (LX_1) \lambda_p} \beta_1 - \int_{\exp(-LX_1)\lambda_p} \beta_1, \label{eq:remaining-portions}  
\end{align}}%
where we used that $\beta_1(X_1)=0$ to deduce that
{\small\[ \int_{\sigma(0,L)} \beta_1 + \int_{\exp(\ell^{p-1}X_p)\sigma(L,-L)} \beta_1 + \int_{\sigma(-L,0)} \beta_1 =0. \]}%

In order to evaluate the integral of $\beta_1$ along $\Lambda$, we must determine the sign of $u_1$ on the portions of the loop that remain in \eqref{eq:remaining-portions}. Start by observing that $\vert(u_1 \circ \Lambda)'(t)\vert \leqslant 1$ for all $t\in[0,4L + 2 \ell_p]$ while $u_1(\Lambda(L))=L$ and $u_1(3L+2\ell_p) = -L$. So if $L>\ell_p$ then 
{\small$$
(u_1\circ \Lambda)(t) \geqslant L - \int_{L}^t \vert (u_1\circ \Lambda)'(s) \vert ds \geqslant L - \ell_p > 0
$$}%
for all $t \in [L,L+\ell_p]$ and $(u_1\circ \Lambda)(t)<0$ for all $t \in [3L+\ell_p,3L+2\ell_p]$, ensuring that 
\begin{equation*}
    \int_{\exp (LX_1) \lambda_p} \beta_1 = \int_{\exp (LX_1) \lambda_p} \beta_0 = \ell^{p-1}
    \end{equation*}
and
\begin{equation*}    
    \int_{\exp (-LX_1) \lambda_p} \beta_1 = \int_{\exp (-LX_1) \lambda_p} (-\beta_0) = - \ell^{p-1}.
\end{equation*}
Plugging these two equalities into \eqref{eq:remaining-portions} yields the conclusion, still under the assumption $L>\ell_p$. Finally, by \eqref{eq:exact-lp}, $\ell_p$ is bounded by a multiple of $\ell$.
So the conclusion holds under the assumption that $L \geqslant C_p \ell$, for some $C_p >0$.
\end{proof}

To complete the proof of Proposition \ref{prop:lower-bound-general}, we need the following auxiliary Lemma.

\begin{lemma}\label{lem:u1-is-lipschitz}
    The function $u_1$ is lipschitz on $G$.
\end{lemma}

\begin{proof}
    Consider the projection $\pi\colon G \to \mathbf R^2$ given by quotienting by the normal subgroup generated by $[K,K]$ and $H$. Then $u_1 = \pi^\ast \underline u_1$ where $\underline u_1$ is a linear form on $\mathbf R^2$.
\end{proof}

\begin{proof}[Proof of Proposition~\ref{prop:lower-bound-general}]
Equip $K$ with the adapted compact presentation given in Proposition~\ref{prop:compactpresentations}.
After choosing any compact presentation for the group $H$, we can produce an adapted presentation of $G$ (Definition \ref{dfn:adapted}) using Lemma~\ref{lem:central-prod-is-cglc}. This is the presentation $\langle S \mid R \rangle$ we use.
In particular, we denote two subsets of the generating system,
$S_1 = \{ \exp(tX_1): 0 \leqslant t \leqslant 1 \}$ and $S_2 = \{ \exp(tX_2): 0 \leqslant t \leqslant 1 \}$.
According to the setting of Lemma~\ref{combinatorial-stokes}, for every $s \in S_i$ we let 
\[ \overline s: [0,1] \to G, \sigma \mapsto \exp(\sigma t X_i), \]
where $t$ is such that $s = \exp(tX_i)$. Note that they have bounded length. We do not specify the other $\overline s$, aside from asking that they are smooth of bounded length, as the rest of the proof does not depend on them.

For integers $L$ and $\ell$ the loop $\Lambda$ in $G$ is then the path $\overline w$, where
    \[ w = x_1^L \Omega^2_{p-1}(\ell, \ldots, \ell) x_1^{-2L} \Omega^2_{p-1}(\ell, \ldots, \ell)^{-1} x_1^L. \]
    The word $w$ has length $4L+2\ell_p$ in $S$, as can be seen by observing that $\Omega_{p-1}^2(\ell, \ldots, \ell)$ has length $\ell_p$ in view of Remark \ref{rem:lp}.

    Set $L = \lceil C_p \ell \rceil$, where $C_p$ is the constant from Lemma \ref{lem:evaluate-integral-beta1}.
    From this point, fix the left-invariant Riemannian metric on $G$ for which our chosen basis of $\mathfrak{g}$ is orthonormal (this is not strictly necessary, but simplifies some of the following estimates).
    Since $S$ is compact and every relator in $R$ has bounded length in $S$, there is a uniform upper bound on the length and diameter of the $\overline{r}$ with $r\in R$. By left-invariance, the same is true (with the same bounds) for the left translates $g_\ast \overline r$ with $g\in G$.

    We use the following case distinction to deduce a uniform upper bound on $\int_{g_\ast \overline r} \beta_1$ with $g\in G$ and $r\in R$:
    \begin{itemize}
        \item If $g_\ast \overline r$ does not intersect $\{u_1 = 0 \}$, then $d\beta_1$ coincides with a form with left-invariant derivative (namely, $\pm d\beta_0$) on $g_\ast \overline r$. By the usual Stokes formula (for surfaces with piecewise smooth boundary), for every disk $\Delta$ filling $\overline r$ we have
        \[ \int_{g_\ast \overline r} \beta_1 = \int_{g_\ast \Delta} \pm d\beta_0, \]
        and the latter is bounded by a constant times the Riemannian area of $\Delta$. This area is itself bounded by a constant that does not depend on $r$.
        \item If $g_\ast \overline r$ intersects $\{u_1 = 0 \}$, by Lemma~\ref{lem:u1-is-lipschitz}, there is a constant $M$ independent of $r$ such that $\sup_{g_\ast \overline r} \vert u_1 \vert \leqslant M$. In view of the definition of $\beta_1$ given in \eqref{eq:definition-of-beta1}, 
        {\small\[ \left\vert \int_{g_\ast \overline r} \beta_1 \right\vert \leqslant \operatorname{length}(\overline r) \cdot \left( M+ \frac{M^2}{2} + \cdots + \frac{M^{p-2}}{(p-2)!} \right) \leqslant \operatorname{length}(\overline r) \cdot (e^M-1), \]}%
    \end{itemize}
    where by length, we mean the Riemannian length.
    We may now set 
    \[ C:= \sup \left\{ \left\vert \int_{g_\ast \overline r} \beta_1 \right\vert \colon r \in R \right\}<+\infty \] and apply Lemma~\ref{combinatorial-stokes} with the word $w$ defined above. We conclude that $\operatorname{Area}(w)$ is at least of the order $\ell^{p-1}$. Since the word length of $w$ is of the order of $\ell$, this finishes the proof.
\end{proof}

\subsection{Lower bound using maximal distortion: Lie-group version}
\label{subsec:dehn-lower-lie}

\begin{proposition}
\label{prop:dehn-lower-lie}
Let $k \geqslant 2$.
    Let $G$ be a simply connected nilpotent Lie group {of nilpotency class $k-1$}.
    Assume that for some $\omega \in Z^2(G,\mathbf R)$  the central extension 
    \[ 1 \to \mathbf R \overset{\iota}{\longrightarrow} \widetilde G \overset{\pi}{\to} G \to 1 \]
    has nilpotency class $k$.
    Then the Dehn function of $G$ satisfies $\delta_G(n) \succcurlyeq n^k$.
\end{proposition}
To prove Proposition \ref{prop:dehn-lower-lie} we require two preliminary lemmas.

\begin{lemma}\label{lem:first-lemma-simple-k-fold}
Let $\widetilde G$ be a simply connected nilpotent Lie group of class $k$ with a one-dimensional centre.
There exists $x_1,\ldots, x_k \in \widetilde G$ such that 
$[x_1,\ldots,x_k] \in Z(\widetilde G) \setminus 1$.
\end{lemma}
\begin{proof}
The proof is inspired by \cite[Exercise 3 page 291]{MKS66}.
Using Lemma \ref{lem:freeidentities}(2) and the fact that a conjugate of a commutator is a commutator of conjugates, we prove by induction on $c$ that every element of $C^c \widetilde G$ is a product of simple $c$-fold commutators. This is especially true for $c=k$.
The assumptions imply that $C^k \widetilde G = Z(\widetilde G)$. Consider $z \in C^k \widetilde G$, not equal to $1$ and write it as a product of $k$-fold simple commutators.
Assume for a contradiction that all $k$-fold commutators in this product are trivial; then $\widetilde G$ is of class less or equal $k-1$. So one of them, say $[x_1,\ldots , x_k]$, is non-trivial.
\end{proof}

\begin{lemma}\label{lem:second-lemma-simple-k-fold}
Let $\widetilde G$ be a simply connected nilpotent Lie group of class $k$, and $x_1,\ldots, x_k \in \widetilde G$ such that 
$x:=[x_1,\ldots,x_k] \in Z(\widetilde G) \setminus 1$.
Then for every $s \in \mathbf R$,
\[ [x_1^s, x_2^s,\ldots ,x_k^s] = x^{s^k}. \]
\end{lemma}
\begin{proof}
    By repeated application of the Baker-Campbell-Hausdorff formula,
    \[ \log[x_1^s, \ldots, x_k^s] = [\log x_1^s, \ldots , \log x_k^s] = s^k \log [x_1, \ldots, , x_k]. \qedhere \]
\end{proof}

\begin{proof}[Proof of Proposition~\ref{prop:dehn-lower-lie}]
Let $V$ be as in the statement of Lemma~\ref{lem:main-ingredient-lower-lie}.
Fix a compact generating set $S$ of $G$; let $x_1,\ldots, x_k$ be as provided by Lemma~\ref{lem:first-lemma-simple-k-fold}. Moreover, assume without loss of generality in view of Lemma~\ref{lem:second-lemma-simple-k-fold} that $[x_1, \ldots ,x_k] = \iota(1)$.
We also consider a compact presentation of $\widetilde G$ in which all the $x_i$ and their powers with exponents between $0$ and $1$ are in the generating system.
Set $X_i = \log (\pi(x_i))$ for all $i$ and denote by $\widetilde X_i$ the lift of $X_i$ to $V$. Note that $X_i \neq 0$ for all $i$, otherwise $[x_1,\ldots , x_k]$ would be trivial.

Let $\ell \geqslant 0$ be a parameter, and consider the loop $\lambda = \overline w$, where
$w = [\pi(x_1)^\ell,\ldots, \pi(x_k)^\ell]$ and $ \overline {\pi(x_i)}  $
is the segment $\{ \exp(sX_i) \}$ for $ 0 \leqslant s \leqslant 1$.
A lift of $\lambda$ in $\widetilde G$ that is everywhere parallel to $V$ is $\widetilde \lambda = \overline \Omega$,
where $\Omega$ is the word $[x_1^\ell, \ldots, x_k^\ell]$ in $\widetilde G$ and $ \overline {x_i}  $
is the segment $\{ \exp(s\widetilde X_i) \}$ for $ 0 \leqslant s \leqslant 1$.

Let $\alpha$ be a primitive for $\omega$.
By Lemmas \ref{lem:main-ingredient-lower-lie} and \ref{lem:second-lemma-simple-k-fold} we have that
\begin{equation*}
    \left\vert \int_{\lambda} \alpha \right\vert = \ell^k. 
\end{equation*}
Since $\omega={\rm d}\alpha$ is a left-invariant form, there is a uniform bound on the absolute values of its integrals over $g_{\ast}\overline r$ for $r \in R$ and $g\in G$. It thus follows from Lemma~\ref{combinatorial-stokes} that the combinatorial area of $w$ is at least of the order of $\ell^k$. Since its word length is of the order of $\ell$, this concludes the proof.
\end{proof}

\begin{corollary}
\label{cor:dehn-lower-lie}
Let $k \geqslant 2$.
    Let $G$ be a simply connected nilpotent Lie group {of nilpotency class $k-1$}.
    Assume that for some $\omega \in Z^2(G,\mathbf R)$  the central extension 
    \[ 1 \to \mathbf R \overset{\iota}{\longrightarrow} \widetilde G \overset{\pi}{\to} G \to 1 \]
    has nilpotency class $k$.
    Then $\delta_G(n) \asymp n^k$.
\end{corollary}
\begin{proof}
    The lower bound is given by Proposition \ref{prop:dehn-lower-lie} and the upper bound is the Lie-group version of the Gersten--Holt--Riley upper bound (see Theorem~\ref{th:GHR-lie}).
\end{proof}

\section{Upper bounds}
\label{sec:upper-bounds}
In this section we treat upper bounds on Dehn functions. With these results together with the ones in \S\ref{sec:lowerbounds} we are able to precisely compute the Dehn functions of the groups we consider. 

The first part, \S\ref{sec:upper-gp3-gp3corner}, deals with the central products $$G_{p,3} := L_p \times_Z L_{3} \quad  \text{and} \quad G^\lrcorner_{p,3} := L^\lrcorner_p \times_Z L_{3}$$ introduced in \S\ref{sec:efficient}. The main steps of the proof are similar to the ones in \cite[Section 6]{lipt}, but the actual proofs of some of the steps require significant changes that involve new ideas. We thus follow the overall structure and notation used in \cite{lipt} and refer the reader to that work in places where our proofs are completely analogous, while explaining in detail the steps where our proofs are different. We mostly focus on $G_{p,3}^{\lrcorner}$. The proofs for $G_{p,3}$ are similar, but slightly easier, since one of the relations simplifies. An important takeaway from this part is that the approach developed in \cite[Section 6]{lipt} holds with much higher generality. We achieve this by significantly simplifying the first part of the proof given there. The most important simplification is that we avoid the use of the highly technical Fractal Form Lemma  \cite[Lemma 6.17]{lipt} by instead using an inductive argument to prove the First and Second commuting Lemmas and then deduce the Main Commuting Lemma (see below for the statements of these results). Further refining this approach seems to be a promising route towards extending the upper bound in Conjecture \ref{conj:main} to very general classes of groups in the factor of higher nilpotency class. The fact that we simultaneously cover the groups $L_p$ and $L_{p}^{\lrcorner}$ showcases this idea, which is one of our motivations for proving this conjecture for both of these cases.

In the second part, \S\ref{sec:fromp3topqandbeyond}, we show how to obtain the upper bounds for a bigger class of central products of nilpotent groups using the results obtained for $G_{p,3}$ and  $G^\lrcorner_{p,3}$, namely central products of the form $$L_p \times_Z H \quad \text{and} \quad L^\lrcorner_p \times_Z H,$$ where $H$ is a simply connected nilpotent Lie group of class $q-1$ with one-dimensional centre. This provides the upper bounds in Theorem \ref{th:general-factor} and Theorem \ref{th:model-filiform}.

\subsection{Notation and conventions}\label{not:omegawords}
We start this section by recalling notation from \S\ref{subsec:explicit-compact-presentations} that is used throughout the whole section.

For $j \geqslant 2$, $k \geqslant 1$, and $\underline{n} := (n_1, \ldots, n_k) \in \R^k$ we denote by 
	\[
		\Omega_k^j (\underline{n}) := [x_1 ^{n_1}, \ldots , x_1 ^{n_{k-1}}, x_j ^{n_k}].
	\]
For $j=2$ we only denote it by $\Omega_k(\underline{n})$. We refer to them as \emph{$\Omega_k^j$-words}, and when we do not need to specify the parameters $(k,j)$ we shall refer to them simply as \emph{$\Omega$-words}. Moreover, we extend the notation of $\Omega$-words to the generators $y_i$ as follows:
	\[
		\widetilde{\Omega}_k^j(\underline{n}) :=[y_1^{n_1},\ldots, y_1^{n_{k-1}}, y_j^{n_k}].
	\] 

Recall the efficient sets introduced in \S\ref{sec:efficient}:
\[
\Sigma := \{ x_1 ^{a_1}, x_2 ^{a_2} \mid |a_1|, |a_2| \leqslant  1\}, \quad \mathcal{F} := \{ s^n \mid s \in \Sigma, n \in \N \},
\]
\[
T := \{ x_1 ^{a_1}, x_2 ^{a_2}, y_{1}^{a_3}, y_{p-1}^{a_4} \mid |a_1|, |a_2|, |a_3|, |a_4| \leqslant 1\}, \ \text{and} \ \mathcal{G} := \{ s^n \mid s \in T, n \in \N \}.
\]

\subsection{Upper bounds on the Dehn functions of $G^\lrcorner_{p,3}$ and $G_{p,3}$} \label{sec:upper-gp3-gp3corner}
We fix once and for all the compact presentation $\calp$ for $G^\lrcorner_{p,3}$ (respectively for $G_{p,3}$) obtained using Lemma \ref{lem:central-prod-is-cglc} and the compact presentation for $L^\lrcorner_p$ (respectively for $L_p$) given in Proposition \ref{prop:compactpresentations}. By abuse of notation we denote both presentations by $\calp$ and only specify to which groups it refers whenever confusion may arise. We start by stating the main results of this section.
\begin{proposition}\label{prop:upperbound-dehn-Gp3}
        Let $p \geqslant 3$. The Dehn function of the group $G_{p,3} := L_p \times_Z L_{3}$ satisfies $ \delta_{G_{p,3}}(n) \preccurlyeq n^{p-1}$. 
\end{proposition}

\begin{proposition}\label{prop:upperbound-dehn-Gcornerp3}
    Let $p \geqslant 5$. The Dehn function of the group $G^{\lrcorner}_{p,3}:= L^{\lrcorner}_p \times_Z L_{3}$ satisfies $\delta_{G^\lrcorner _{p,3}}(n) \preccurlyeq n^{p-1}$.
\end{proposition}

\subsubsection*{Advice to the reader} On a first read, the structure of the proofs of the upper bounds on the Dehn functions in this section and the motivation for the auxiliary results may be easier to follow for a reader by first looking at the end of the proofs of \Cref{prop:upperbound-dehn-Gp3,prop:upperbound-dehn-Gcornerp3} to see how the various auxiliary results are used. In particular, we recommend that the reader first looks at \Cref{remark-general-proof} to see how the results of the section are combined to derive an upper bound on the Dehn function of $G_{p,3}$ and $G_{p,3}^{\lrcorner}$. From there they can work their way backwards. In the process, it may also be helpful to consider \Cref{fig:commuting-lemmas} to understand the structure of the induction arguments.

\subsubsection*{Strategy of the proofs of Propositions \ref{prop:upperbound-dehn-Gp3} and \ref{prop:upperbound-dehn-Gcornerp3}} 

The proof of Proposition \ref{prop:upperbound-dehn-Gp3} and Proposition \ref{prop:upperbound-dehn-Gcornerp3} are done simultaneously by ascending induction in $p$, where we emphasise that the \emph{induction steps from $p-1$ to $p$  only require the induction hypothesis for $G_{q,3}$ with $q<p$}. 

For the proof of the induction step for $G^\lrcorner_{p,3}$ (respectively $G_{p,3}$) we need to show that the area of null-homotopic words $w := w(x_1, x_2)$ in $G^\lrcorner_{p,3}$ (respectively $G_{p,3}$) is $\lesssim_p n^{p-1}$. By Proposition \ref{prop:reductiontrick} it suffices to show that null-homotopic words $w(x_1,x_2)\in \mathcal{F}[\alpha]$ have area $\lesssim_{p,\alpha} n^{p-1}$ for every $\alpha>0$. To show this we first prove a result about commuting certain words in $G^\lrcorner_{p,3}$ (respectively $G_{p,3}$), which we call the Main commuting Lemma \ref{lem:maincommuting} for $G^\lrcorner_{p,3}$ (respectively Lemma \ref{lem:gp3maincommuting} for $G_{p,3}$). 

It is in the proof of the Main commuting Lemmas for $G_{p,3}$ and $G^\lrcorner_{p,3}$ that our proof significantly simplifies the one in \cite{lipt}, as we shall now explain. Each one of the Main commuting Lemmas is a consequence of two corresponding  results about commutators involving the $\Omega_k^j$-words -- namely the First and Second commuting $(k,j)$-Lemmas -- and the Reduction Lemma \ref{lem:reductionlemma}. The First commuting $(k,j)$-Lemma is a generalisation of \cite[Lemma 6.5 and Lemma 6.7]{lipt}, while the Second commuting $(k,j)$-Lemma is a generalisation of \cite[Lemma 6.4 and Lemma 6.6]{lipt}. We prove the commuting $(k,j)$-Lemmas by descending induction in $j$ -- more details on this induction will be given at the beginning of \Cref{sec:proof-of-1st-and-2nd-commuting-lemma} and in \Cref{fig:commuting-lemmas}. Comparing with \cite{lipt}, the first induction step from $(k,2)$ to $(k,3)$ in our work is similar to the reduction of \cite[Lemmas 6.6 and 6.7]{lipt} to \cite[Lemmas 6.4 and 6.5]{lipt} in \cite{lipt}. However, instead of then continuing this inductive procedure to further reduce to \cite[Lemmas 6.4 and 6.5]{lipt} for $j>3$, the proof in \cite{lipt} uses a technical result, called the Fractal Form Lemma \cite[Lemma 6.17]{lipt}. It decomposes a word of the form $\Omega_k^j(\underline{n})$ into $\lesssim_p |\underline{n}|^k$ copies of $\Omega_k^j(\frac{\underline{n}}{\log_2(|\underline{n}|)})$ and some error terms which are interlaced in a fractal form. It then uses this decomposition to directly prove versions of the First and Second commuting $(k,j)$-Lemmas by descending induction in $k$ instead of $j$. The Fractal Form Lemma \cite[Lemma 6.17]{lipt}, as well as several auxiliary results used in its proof, make up a significant part of the proof in \cite{lipt}, which we were able to remove in our work. This led to a simplification, allowing us to compute the Dehn functions of $G_{p,q}$ and $G_{p,q}^\lrcorner$ for all $p$ and $q$.

Once we have proven the First and Second commuting $(k,j)$-Lemmas, our proof follows the arguments in \cite{lipt}: in particular, key steps include the Reduction Lemma \ref{lem:reductionlemma}, which allows us to deduce the Main Commuting Lemmas from the First and Second Commuting $(k,j)$-Lemmas, and the Cancelling $k$-Lemma \ref{lem:cancelling} for $G_{p,3}$ (respectively $G^\lrcorner_{p,3}$). We start by stating the main results.\vspace{.3cm}

\noindent \textbf{Disclaimer.} Throughout Section \S\ref{sec:upper-gp3-gp3corner} we state the results for both of the groups $G_{p,3}$ and $G^\lrcorner_{p,3}$ and prove them for the more complicated case of $G^\lrcorner_{p,3}$. In some cases the corresponding statements for $G_{p,3}$ are simpler due to the absence of the relation $[x_2,x_3]=_\calp x_p$ (see for instance the Main commuting Lemma or the Second commuting $(k,j)$-Lemma). Whenever it may not be clear that a proof for $G^\lrcorner_{p,3}$ translates easily to one for $G_{p,3}$ we offer an explanation.\vspace{.3cm}

\begin{lemma}[Main commuting Lemma for $G_{p,3}$, {cf. \cite[Lemma 6.2]{lipt}}]\label{lem:gp3maincommuting}
    Let $p \geqslant 3$, $\alpha \geqslant 1$, $n \geqslant 1$, and let $w_1 (x_1, x_2)$, $w_2(x_1, x_2)$ be either powers of $x_2$ or words in $\mathcal{F}[\alpha]$ representing elements in the derived subgroup of $G_{p,3}$. If $|w_1|,|w_2| \leqslant n$, then the identity $$[w_1, w_2] =_\calp 1$$ holds in $G_{p,3}$ with area $\lesap n^{p-1}$.
\end{lemma}

\begin{lemma}[First commuting $(k,j)$-Lemma for $G_{p,3}$, cf. {\cite[Lemma 6.5 and Lemma 6.7]{lipt}}]\label{lem:gp3firstcommuting}
    Let $p \geqslant 3$, $n \geqslant 1$, $j \geqslant 2$, $k \geqslant 1$, $\underline{n} \in \R^k$ with $|\underline{n}| \leqslant n$, and $m \in \R$. The identity $$[\Omega_k ^j (\underline{n}), x_2^m]=_\calp 1$$ holds in $G_{p,3}$ with area $\lesssim_p |m|n^{p-j} + n^{p-j+1}$ if $j< p$ and area $\lesssim_p |m|n + n^2$ if $j = p$.
\end{lemma}

\begin{lemma}[Second commuting $(k,j)$-Lemma for $G_{p,3}$, cf. {\cite[Lemma 6.4 and Lemma 6.6]{lipt}}]\label{lem:gp3secondcommuting}
    Let $p \geqslant 3$, $n \geqslant 1$, $j \geqslant 2$, $k \geqslant 1$, $\underline{n} \in \R^k$ with $|\underline{n}| \leqslant n$, $\alpha \geqslant 1$ and  $w:= w(x_1, x_2) \in \mathcal{F}[\alpha]$ be a word of length at most $n$ representing an element in the derived subgroup of $G_{p,3}$. The identity $$[\Omega_k ^j (\underline{n}), w] =_\calp 1$$ holds in $G_{p,3}$ with area $\lesap n^{p-j+1}$ if $j<p$ and area $\lesap n^2$ if $j=p$.
\end{lemma}

\begin{lemma}[Main commuting Lemma for $G^\lrcorner _{p,3}$, cf. {\cite[Lemma 6.2]{lipt}}]\label{lem:maincommuting}
    Let $p \geqslant 5$, $\alpha \geqslant 1$, and $n \geqslant 1$. If $w_1$ and $w_2$ are either powers of $x_2$ or words of length at most $n$ in $\mathcal{F}[\alpha]$ representing elements of the derived subgroup of $G^\lrcorner_{p,3}$ with $|w_1|,|w_2| \leqslant n$, then the identities
    \begin{equation*}
        [w_1, w_2] =_\calp \left\{\arraycolsep=1.5pt\def\arraystretch{1.8}
    \begin{array}{ll}
            1, &\ w_1, w_2 \ \text{are both powers of} \ x_2,\\
            1, &\ w_1, w_2 \ \text{both represent elements in } \ [G^{\lrcorner}_{p,3}, G^{\lrcorner}_{p,3}],\\
            \prod\limits_{i=1}^{D} \Omega_3 ^{p-2}(\underline{\eta}_i)^{\pm 1}, &\ w_1 \ \text{represents a word in } [G^{\lrcorner}_{p,3}, G^{\lrcorner}_{p,3}] \ \text{and} \ w_2 := x_2 ^k. 
     \end{array}\right.
    \end{equation*} hold in $G^{\lrcorner}_{p,3}$ with area $\lesap n^{p-1}$ for some $D = O_{\alpha, p}(1)$, and $\underline{\eta}_i \in \R^3$ with $|\underline{\eta}_i| \lesssim_p n$. 
\end{lemma}

\begin{lemma}[First commuting $(k,j)$-Lemma for $G^\lrcorner_{p,3}$, cf. {\cite[Lemma 6.5 and Lemma 6.7]{lipt}}
]\label{lem:almostkjcommuting} 
    Let $p \geqslant 5$, $n \geqslant 1, j \geqslant 2$, $k \geqslant 1$, $\underline{n} \in \R^k$, and $m \in \R$ with $|\underline{n}| \leqslant n$. Then, the identities
    \begin{equation*}
        [\Omega^{j}_{k}(\underline{n}), x_{2}^m] =_{\calp} \begin{cases}
            1, & \text{if} \ k+j \neq 4,\\
            \Omega_{2}^{p-1}(m, -n_1), & \text{if} \ (k,j)=(1,3), \\
            \Omega^{p-1}_{2}(m, -n_2)^{n_1}, & \text{if}  \ (k,j)=(2,2),
        \end{cases}
    \end{equation*}
    hold in $G^{\lrcorner}_{p,3}$. with area $\lesssim_p |m|n^{p-j} + n^{p-j+1}$ if $j<p$ and area $\lesssim_p |m|n + n^2$ if $j=p$. 

    If, moreover, $|m| \leqslant n$ and $(k,j)=(2,2)$, then an identity of the form $$[\Omega_2(\underline{n}), x_2^m]=_\calp \Omega^{p-2}_{3}(\underline{\tilde{m}})$$ holds in $G^\lrcorner_{p,3}$ with area $\lesssim_p n^{p-1}$ for some suitable $\underline{\tilde{m}} \in \R^3$ with $|\underline{\tilde{m}}| \lesssim_p n$.
\end{lemma}

\begin{lemma}[Second commuting $(k,j)$-Lemma for $G^\lrcorner_{p,3}$, cf. {\cite[Lemma 6.4 and Lemma 6.6]{lipt}}] \label{lem:kjcommuting}
     Let $p \geqslant 5$, $n \geqslant 1, j \geqslant 2$, $k \geqslant 1$, $\underline{n} \in \R^k$, and  $\alpha \geqslant 1$ with $|\underline{n}| \leqslant n$. If $w:=w(x_1, x_2) \in \mathcal{F}[\alpha]$ is a word of length at most $n$ representing an element in the derived subgroup of $G^{\lrcorner}_{p,3}$, then the identity
     \begin{equation*}
         [\Omega^{j}_{k}(\underline{n}), w] =_{\calp} 1,
     \end{equation*}
     holds in $G^{\lrcorner}_{p,3}$ with area $\lesap n^{p-j+1}$ if $j<p$ and area $\lesap n^2$ if $j=p$.
\end{lemma}

A key step in the proof of the Main Commuting Lemmas is the Reduction Lemma; the proof of the latter relies on the First and Second commuting Lemmas.  

\begin{lemma}[{\cite[Lemma 6.8]{lipt}}, Reduction Lemma]\label{lem:reductionlemma}
    Let $\alpha \geqslant 1$, let $w=w(x_1, x_2)$ be a word of length at most $n$ in $\mathcal{F}[\alpha]$ representing an element in the derived subgroup of $G^\lrcorner_{p,3}$ for $p \geqslant 5$ (respectively $G_{p,3}$ for $p \geqslant 3$). There exists $L = O_{\alpha,p}(1)$ such that the identity
    $$
    w =_\calp \prod_{i=1}^L \Omega_{l_i}(\underline{\eta}_i)^{\pm 1} 
    $$
    holds in $G^{\lrcorner}_{p,3}$ (respectively $G_{p,3}$) with area $\lesap n^{p-1}$ for some $2 \leqslant l_i \leqslant p-1$ and some $\underline{\eta}_i \in \R^{l_i}$ with $|\underline{\eta}_i| \lesssim_p n$. 
\end{lemma}

Finally, we deduce the Cancelling $k$-Lemma from the Main Commuting Lemmas and then use it to complete the proofs of Proposition \ref{prop:upperbound-dehn-Gp3} and Proposition \ref{prop:upperbound-dehn-Gcornerp3}.

\begin{lemma}[Cancelling $k$-Lemma cf. {\cite[Lemma 6.9]{lipt}}]\label{lem:cancelling}
    Let $n \geqslant 1$, $2 \leqslant k \leqslant p-1$, $M_j$ be a positive integer for all $k \leqslant j \leqslant p-1$, and $M:=\max \left\{M_j\mid k\leqslant j \leqslant p-1\right\}$. Consider the word
    \[
    w(x_1, x_2) := \left(\prod_{i=1}^{M_k} \Omega_k (\underline{n}_{k,i})^{\pm 1} \right) \left(\prod_{i=1}^{M_{k+1}} \Omega_{k+1} (\underline{n}_{k+1,i})^{\pm 1}\right) \ldots \left(\prod_{i=1}^{M_{p-1}} \Omega_{p-1} (\underline{n}_{p-1,i})^{\pm 1}\right),
    \] 
    for some $\underline{n}_{l,i} \in \R^j$ with $|\underline{n}_{l,i}| \leqslant n$. If $w:=w(x_1,x_2)$ is null-homotopic in $G^{\lrcorner}_{p,3}$ for $p \geqslant 5$ (respectively $G_{p,3}$ for $p \geqslant 3$), then it has area $\lesssim_{M,p} n^{p-1}$.
\end{lemma}

We now proceed with the proofs of Proposition \ref{prop:upperbound-dehn-Gp3} and Proposition \ref{prop:upperbound-dehn-Gcornerp3}.

\subsubsection{Start of the induction on $p$}

As mentioned above, the proofs of Proposition \ref{prop:upperbound-dehn-Gp3} and Proposition \ref{prop:upperbound-dehn-Gcornerp3} are done in parallel by ascending induction on $p$. The corresponding bases of induction are as follows

\subsubsection*{Base of induction on $p$} By \cite{Allcock} and \cite{OlsSapCombDehn} the Dehn function of $G_{3,3}$ satisfies $\delta_{G_{3,3}}(n) \asymp n^2$ and by \cite{lipt} the Dehn function of the group $G_{4,3}$ satisfies $\delta_{G_{4,3}}(n) \asymp n^3$.

\subsubsection*{Induction hypothesis for $p$} Throughout \S\ref{sec:upper-gp3-gp3corner} these are our standing assumptions. As mentioned above, the induction steps from $p-1$ to $p$ only requires the induction hypothesis for $G_{q,3}$ with $q<p$

\begin{enumerate}[leftmargin=4em, font=\bfseries]
    \item[{\crtcrossreflabel{(IH-p)}[item:IH-fromp-1top]}] Let $p-1 \geqslant 3$. Suppose Proposition \ref{prop:upperbound-dehn-Gp3} holds for all $G_{q,3}$ with $q \leqslant p-1$. Moreover, suppose the Main commuting Lemma for $G_{q,3}$ along with the First commuting $(k,j)$-Lemma \ref{lem:gp3firstcommuting} and the Second commuting $(k,j)$-Lemma \ref{lem:gp3secondcommuting} for $G_{q,3}$ hold. 
\end{enumerate}

For the induction step from $p-1$ to $p$ for $G^\lrcorner_{p,3}$ (respectively $G_{p,3}$) we first prove the First and Second commuting $(k,j)$-Lemmas for $G^\lrcorner_{p,3}$ (respectively $G_{p,3}$).

\subsubsection{Proof of the First and Second Commuting $(k,j)$-Lemmas}\label{sec:proof-of-1st-and-2nd-commuting-lemma}

\subsubsection*{Strategy of the proofs} We prove the First and Second commuting $(k,j)$-Lemmas for $G^\lrcorner_{p,3}$ and $G_{p,3}$ in \emph{parallel by descending induction on $j$} for all $(k,j)\neq (1,2)$. For $(1,2)$, we have by definition that $\Omega_1^2(n)=x_2^n$, so the First commuting $(1,2)$-Lemma is trivially true. The Second commuting $(1,2)$-Lemma is a particular case of the Main Commuting Lemma, whose proof only relies on the First and Second commuting $(k,j)$-Lemmas for $(k,j) \neq (1,2)$.

The induction in $j$ for $G^\lrcorner_{p,3}$ is done for all pairs $(k,j)$ such that \emph{$k + j \geqslant 5$}, while for $G_{p,3}$ it is done for \emph{all} pairs $(k,j)$. In the induction step for $j$, namely to prove the statement for the pairs $(k,j_0)$ assuming the statements for the pairs $(k,\ell)$ with $j_0 +1 \leqslant \ell \leqslant p$, we distinguish between two cases: 
\begin{itemize}
    \item \emph{Case (i): pairs $(k, j_0)$ with $k\neq 1$}. In this case the First (respectively Second) commuting $(k,j_0)$-Lemma requires the First (respectively Second) commuting $(r,j_0 +1)$-Lemmas for $G^\lrcorner_{p,3}$ for all $r \geqslant k-1$.\vspace{3pt} 
    \item \emph{Case (ii): the pair $(1,j_0)$}. For this pair the First commuting $(1,j_0)$-Lemma will follow easily from the presentation of $G^\lrcorner_{p,3}$, while the Second commuting $(1,j_0)$-Lemma requires the First commuting $(k,j_0)$-Lemma for $G^\lrcorner_{p,3}$ with $k>1$.
\end{itemize}
Finally, the statements for the remaining pairs $(1,3)$ and $(2,2)$ for $G_{p,3}^{\lrcorner}$ require extra arguments to address the appearing ``error terms'' coming from the relation $[x_2, x_3]=_\calp x_p$. For a pictorial description of the implications between the statements see Figure \ref{fig:commuting-lemmas}.

\begin{figure}
\begin{center}
\begin{tikzpicture}[line cap=round,line join=round,>=angle 60,x=0.7cm,y=0.7cm]
\clip(-2.5,-1) rectangle (14.5,13);
\draw[->,color=black] (1,0) -- (13,0) node[right]{$j$};
\foreach \x in {2,3}
\draw[shift={(\x,0)},color=black] (0pt,2pt) -- (0pt,-2pt) node[below] {\footnotesize $\x$};
\draw[shift={(4,0)},color=black] (0,-3pt) node[below] {$\cdots$};
\draw[shift={(5,0)},color=black] (0pt,2pt) -- (0pt,-2pt) node[below] {\footnotesize $j_0$};
\draw[shift={(6,0)},color=black] (0pt,2pt) -- (0pt,-2pt) node[below] {\footnotesize $j_0+1$};
\draw[shift={(7.2,0)},color=black] (0,-3pt) node[below] {$\cdots$};
\draw[shift={(8,0)},color=black] (0pt,2pt) -- (0pt,-2pt) node[below] {\footnotesize $\ell$};
\draw[shift={(9,0)},color=black] (0pt,2pt) -- (0pt,-2pt) node[below] {\footnotesize $\ell+1$};
\draw[shift={(10,0)},color=black] (0,-3pt) node[below] {$\cdots$};
\draw[shift={(11,0)},color=black] (0pt,2pt) -- (0pt,-2pt) node[below] {\footnotesize $p-1$};
\draw[shift={(12,0)},color=black] (0pt,2pt) -- (0pt,-2pt) node[below] {\footnotesize $p$};
\draw[->,color=black] (1,0) -- (1,12) node[above]{$k$};
\foreach \y in {1,2,3}
\draw[shift={(1,\y)},color=black] (2pt,0pt) -- (-2pt,0pt) node[left] {\footnotesize $\y$};
\draw[shift={(1,5)},color=black] (2pt,0pt) -- (-2pt,0pt) node[left] {\footnotesize $p-\ell+1$};
\draw[shift={(1,11)},color=black] (2pt,0pt) -- (-2pt,0pt) node[left] {\footnotesize $p-1$};

\foreach \j in {5,6,8,9,11}
\fill[shift={(\j,1)}, color=black] circle(2pt);
\foreach \j in {3,5,6,8,9}
\fill [color=black] (\j,2) circle (2pt) ;
\foreach \j in {2,3,5,6,8,9}
\fill [color=black] (\j,3) circle (2pt) ;
\foreach \j in {2,3,5,6,8}
\fill [color=black] (\j,5) circle (2pt) ;
\foreach \j in {2,3,5,6}
\fill [color=black] (\j,7) circle (2pt) ;
\foreach \j in {2,3,5}
\fill [color=black] (\j,8) circle (2pt) ;
\foreach \j in {2,3}
{\fill [color=black] (\j,9) circle (2pt) ;
\fill [color=black] (\j,10) circle (2pt) ;}
\fill [color=black] (2,11) circle (2pt) ;
\draw (2,2) node{$\ast$};
\draw (3,1) node{$\ast$};

\draw (6,8) circle(2pt);

\draw (3,11) circle(2pt);

\draw [color=red] (2.7,1.7) -- (3.3,1.7) -- (3.3,10.3) -- (2.7,10.3) -- (2.7,1.7);
\draw [shift={(2,0)},color=red] (2.7,1.7) -- (3.3,1.7) -- (3.3,8.3) -- (2.7,8.3) -- (2.7,1.7);
\draw [shift={(3,0)},color=red] (2.7,1.7) -- (3.3,1.7) -- (3.3,7.3) -- (2.7,7.3) -- (2.7,1.7);
\draw [shift={(5,0)},color=red] (2.7,1.7) -- (3.3,1.7) -- (3.3,5.3) -- (2.7,5.3) -- (2.7,1.7);
\draw [shift={(6,0)},color=red] (2.7,1.7) -- (3.3,1.7) -- (3.3,4.3) -- (2.7,4.3) -- (2.7,1.7);


\draw [color=blue] (2.8,11) -- (2.8,6.8) -- (3.2,6.8) -- (3.2,11);
\draw [color=blue, dash pattern = on 2pt off 2pt] (2.8,11) --(2.8,11.5);
\draw [color=blue, dash pattern = on 2pt off 2pt] (3.2,11) --(3.2,11.5);

\draw [color=blue] (5.8,8) -- (5.8,6.8) -- (6.2,6.8) -- (6.2,8);
\draw [color=blue, dash pattern = on 2pt off 2pt] (5.8,8) --(5.8,8.5);
\draw [color=blue, dash pattern = on 2pt off 2pt] (6.2,8) --(6.2,8.5);

\draw [color=blue] (2.8,11) -- (2.8,6.8) -- (3.2,6.8) -- (3.2,11);
\draw [color=blue, dash pattern = on 2pt off 2pt] (2.8,11) --(2.8,11.5);
\draw [color=blue, dash pattern = on 2pt off 2pt] (3.2,11) --(3.2,11.5);

\draw [color=blue] (2.8,5) -- (2.8,0.8) -- (3.2,0.8) -- (3.2,5);
\draw [color=blue, dash pattern = on 2pt off 2pt] (2.8,5) --(2.8,5.5);
\draw [color=blue, dash pattern = on 2pt off 2pt] (3.2,5) --(3.2,5.5);

\draw [color=blue] (8.8,5) -- (8.8,1.8) -- (9.2,1.8) -- (9.2,5);
\draw [color=blue, dash pattern = on 2pt off 2pt] (8.8,5) --(8.8,5.5);
\draw [color=blue, dash pattern = on 2pt off 2pt] (9.2,5) --(9.2,5.5);

\foreach \k in {2,8}
\draw [->,color=blue] (2.7,\k) -- (2.3,\k);
\foreach \k in {8}
\draw [->,color=blue] (5.7,\k) -- (5.3,\k);
\foreach \k in {8}
\draw [->,color=blue] (5.7,\k-0.7) -- (5.3,\k-0.3);
\foreach \k in {2,8}
\draw [->,color=blue] (2.7,\k-0.7) -- (2.3,\k-0.3);

\draw [->,color=blue] (8.7,3) -- (8.3,3);
\draw [->,color=blue] (8.7,2.3) -- (8.3,2.7);

\foreach \j in {3,5,6,8,9,11}
\draw [->,shift={(\j,0)}, color=red] (0,1.7) -- (0,1.3);


\fill (12,1) circle(2pt);
\fill (11,2) circle(2pt);

\draw (13,5) node {base of};
\draw (13,4.5) node {induction};

\draw [->] (12.8,4) -- (11.3,2.3);

\draw [->] (13.2,4) -- (12.3,1.3);

\draw (7,10) node {descending induction};
\draw (7,9.5) node {on $j$ happens in};
\draw (7,9) node {this region};

\draw [->] (4.8,9.3) -- (4,7.5);

\foreach \k in {1,2,3,5}
\draw (4,\k) node {$\cdots$};

\foreach \k in {1,2,3,5}
\draw (7,\k) node {$\cdots$};

\foreach \k in {1,2}
\draw (10,\k) node {$\cdots$};

\foreach \j in {0.5,2,3,5,6,8,9}
\draw (\j,4) node {$\vdots$};

\foreach \k in {2,3,5,7}
\draw (4,\k) node {$\cdots$};

\end{tikzpicture}
\end{center}
\caption{\small Implications of the Commuting $(k,j)$-Lemmas: the points $\bullet, \ast,$ and $\circ$ correspond to both the First and Second commuting $(k,j)$-Lemmas \ref{lem:almostkjcommuting} and \ref{lem:kjcommuting} respectively (see also Addendum \ref{add:nullhomotopicomegawords}). The drawn blue arrows/boxes encode the fact that for the corresponding statement for $(k,j)$ with $k \neq 1$ we need the statements corresponding to the pairs $(r, j+1)$ for all $r \geqslant k-1$ (see Lemma \ref{lem:commutatortinyletter}). The red arrows/boxes correspond to the fact that for the statement corresponding to $(1,j)$ we need the statements for the pairs $(k,j)$ for all $k \geqslant 2$ (see Lemma \ref{lem:omega-u}). All the points beside $\ast$ correspond to the proof by descending induction on $j$; the points $\ast$ require special treatment and can only be done once we are done with the proof of the induction step.}
\label{fig:commuting-lemmas}
\end{figure}
\subsubsection*{Bases of induction on $j$}
Let $W \in \{x_2^m, w\}$ where $w:=w(x_1, x_2) \in \mathcal{F}[\alpha]$ is a word in the derived subgroup of $G_{p,3}^{\lrcorner}$ (respectively $G_{p,3}$). We start the induction with the pairs $(1, p)$ and $(2, p-1)$.

For $(1, p)$ we have that $\Omega_1^p(\underline{n}) := x_p^{n_1}$ is central and it follows from the presentation $\calp$ that the identity $[x_p^{n_1},W] =_\calp 1$ has area $\lesap |W|n$ in $G^{\lrcorner}_{p,3}$ (respectively $G_{p,3}$). 

For $(2, p-1)$, we first use the fact that $\langle x_1, x_{p-1}, y_1, y_{p-1} \rangle$ is a 5-Heisenberg subgroup of $G^\lrcorner_{p,3}$ to rewrite $\Omega_2^{p-1}(\underline{n}) =_\calp \widetilde{\Omega}_2^{p-1}(\underline{n})$ with area $\lesssim_p n^2$ in $G^\lrcorner_{p,3}$. Since $y_i$ commutes with $x_i$, we deduce that the identity $[\Omega_2^{p-1}(\underline{n}), W]=_\calp 1$ holds in $G^\lrcorner_{p,3}$ with area $\lesap n^2$.

\subsubsection*{Induction hypotheses for $j$}
\begin{enumerate}[leftmargin=6em, labelindent=2em, font=\bfseries]
    \item[{\crtcrossreflabel{(IH-$G_{p,3}$)}[item:IH12-gp3]}] Let $3 \leqslant j_0 +1 \leqslant p$. Suppose that for all pairs $(k,\ell)$ such that $j_0 +1 \leqslant \ell \leqslant p$ the First commuting $(k,\ell)$-Lemma \ref{lem:gp3firstcommuting} and the Second commuting $(k,\ell)$-Lemma \ref{lem:gp3secondcommuting} for $G_{p,3}$ hold.
    
    \item[{\crtcrossreflabel{(IH-$G^\lrcorner_{p,3}$)}[item:IH-gp3corner]}] Let $3 \leqslant j_0 +1 \leqslant p$. Suppose that for all pairs $(k,\ell) \neq (1,3)$ such that $j_0 +1 \leqslant \ell \leqslant p$ the First commuting $(k,\ell)$-Lemma \ref{lem:almostkjcommuting} and the Second commuting $(k,\ell)$-Lemmas \ref{lem:kjcommuting} for $G^\lrcorner_{p,3}$ hold. 
\end{enumerate}

\subsubsection*{Auxiliary results for the induction steps for $j$}

\begin{lemma}[cf. {\cite[Lemma 6.12(3)]{lipt}}]\label{lem:laderlemma}
    Let $\ell \geqslant j_0$, and $\beta, n \in \R$ with $|\beta| \leqslant 1$. The identity $[x_{1}^{\beta}, x_{\ell}^{n}] =_{\calp} x_{\ell+1}^{\beta \cdot n} x_{\ell+2}^{s_{\ell+2}} \ldots x_{p-1}^{s_{p-1}}x_{p}^{s_p}$ holds in $G^\lrcorner_{p,3}$ for all $p \geqslant 5$ (respectively $G_{p,3}$ for all $p \geqslant 3$) for suitable $|s_i| \lesssim_p n$ with area $\lesssim_p n^2$.
\end{lemma}
\begin{proof}
     This is a straight-forward consequence of Proposition \ref{prop:compactpresentations} and its proof. Since the arguments are completely analogous to the ones employed in the proof of \cite[Lemma 6.12(3)]{lipt} we omit them here.
\end{proof}

\begin{lemma}[cf. {\cite[Lemma 6.15]{lipt}}]\label{lem:omegacommuteswithproducts}
    Let $k \geqslant 2$, $\ell \geqslant j_0 +1$, $n \geqslant 1$, and $|n_i|, |s_r| \leqslant n$ for $1 \leqslant i \leqslant k-1$, $\ell \leqslant r \leqslant p$. The identity  \begin{equation*}
        \left[ x_1 ^{n_1}, \ldots, x_1 ^{n_{k-1}}, x_\ell ^{s_\ell} \ldots x_p ^{s_p} \right] =_\calp \prod_{r = \ell} ^{p-1} \Omega_k^r (n_1, \ldots, n_{k-1}, s_r)
    \end{equation*} holds in $G^\lrcorner_{p,3}$ for all $p \geqslant 5$ (respectively $G_{p,3}$ for all $p \geqslant 3$)
    with area $\lesssim_p n^{p-\ell+1}$.
\end{lemma}
\begin{proof}
    Since by assumption $\ell >2$, the statement is a direct consequence of the induction hypothesis \ref{item:IH-fromp-1top} using the fact that the Dehn function of $G_{p-\ell +2,3} \hookrightarrow G^\lrcorner_{p,3}$ is $\preccurlyeq n^{p-\ell+1}$. 
\end{proof}

\begin{lemma}\label{lem:omegainverse}
        Let $(k,\ell)$ with $k\neq1$ and $\ell \geqslant j_0 +1$, and $\underline{n}:=(n_1, \ldots, n_k) \in \R^k$. The identity
        $$[x_{1}^{n_1}, \Omega_{k-1}^{\ell}(n_2, \ldots, n_k)^{-1}] =_{\calp} \Omega_{k}^{\ell}(\underline{n})^{-1}$$
        holds in $G^{\lrcorner}_{p,3}$ for all $p \geqslant 5$ (respectively $G_{p,3}$ for all $p \geqslant 3$) with area $\lesssim_p \max\{|\underline{n}|^{p-\ell+1}, |\underline{n}|^2\}$.
\end{lemma}
\begin{proof}
    For $\ell=p$ this is an easy consequence of $x_p$ being central, so assume $\ell<p$. Since by assumptions $\ell>2$, the statement is then a direct consequence of (i) the induction hypothesis \ref{item:IH-fromp-1top} that the Dehn function of $G_{p-\ell+2,3} \hookrightarrow G^\lrcorner_{p,3}$ is $\preccurlyeq n^{p-\ell+1}$ and (ii) the following free identity in $G_{p-\ell+2,3}$: 
    \[ \left[x_{1}^{n_1},\Omega_{k-1}^{\ell}(n_2, \ldots, n_k)^{-1}\right] =_{F} \left[\Omega_{k-1}^{\ell}(n_2, \ldots, n_k), x_{1}^{n_1}\right]^{\Omega_{k-1}^{\ell}(n_2, \ldots, n_k)^{-1}}.\qedhere \] 
\end{proof}

\begin{corollary}\label{cor:growingomega}
    Let $k \geqslant 2$, $\ell \geqslant j_0+1$, $\underline{n} \in \R^k$, and $|l| \leqslant |\underline{n}|$. The identity
    $$
    [\Omega_k ^\ell (\underline{n})^{\pm 1},  x_1 ^l] =_\calp \Omega_{k+1}^\ell (l, \underline{n})^{\mp 1}
    $$
    holds in $G^{\lrcorner}_{p,3}$ for all $p \geqslant 5$(respectively $G_{p,3}$ for all $p \geqslant 3$ ) with area $\lesssim_p \max\{|\underline{n}|^{p-\ell+1}, |\underline{n}|^2\}$.
\end{corollary}
    \begin{proof} 
        Lemma \ref{lem:omegainverse} implies that the identity $$[\Omega_k ^j (\underline{n})^{\pm 1}, x_1 ^l] =_F [x_1 ^l , \Omega_k ^j (\underline{n})^{\pm 1}]^{-1} =_{\calp} \Omega_{k+1}^j (l, \underline{n})^{\mp 1}$$ holds in $G^{\lrcorner}_{p,3}$ with area $\lesssim_p \max\{|\underline{n}|^{p-\ell+1}, |\underline{n}|^2\}$. 
    \end{proof}

\begin{lemma}[cf. {\cite[Lemma 6.23]{lipt}}]\label{lem:separatingproducts}
    Let $j_0 \leqslant \ell \leqslant p-1$, $ n\geqslant 1, ~\nu \geqslant 1$, and $\eta \in \R$ with $|\eta| \lesssim_p n$. For words $v_i := \Omega^{\ell +1}_{k_i}(\underline{n_i})^{\pm 1}$, with $1\leqslant k_i\leqslant p-\ell$ and $|\underline{n}_i|\lesssim_p n$ ($1\leqslant i \leqslant \nu$), and  $w := \Omega^{\ell}_{k}(\underline{m})^{\pm 1}$, with $2\leqslant k \leqslant p-\ell +1$, the identity
    \begin{equation*}
        [x_{1}^{\eta}, v_1\cdots v_{\nu} \cdot w] =_{\calp} [x_{1}^{\eta},w] \cdot [x_{1}^{\eta},v_{\nu}] \cdots [x_{1}^{\eta},v_1]
    \end{equation*}
    holds in $G^\lrcorner_{p,3}$ for all $p \geqslant 5$ (respectively $G_{p,3}$ for all $p\geqslant 3$). Moreover, if $\ell \leqslant p-2$, then it has area $\lesssim_{\nu, p} n^{p-\ell}$; and if $\ell=p-1$, then it has area $\lesssim_{\nu, p} n^2$.
    \end{lemma}
    \begin{proof}
        Assume $\ell=2$. (If $p-1>\ell >2$ one can either apply the same reasoning in $G_{p-\ell+2,3} \hookrightarrow G^\lrcorner_{p,3}$ or observe that this assertion was already proved in the induction step from $G_{p-\ell+1,3}$ to $G_{p-\ell +2,3}$. In the special case $\ell=p-1$ we use that the power of $x_p$ appearing in $v_1,~ \cdots, v_{\nu}$ is central.)
        We assume that $v_i=  \Omega^{\ell +1}_{k_i}(\underline{n_i})^{-1}$ for $1\leqslant i \leqslant \nu$; the case where some exponents are $+1$ is analogous, but slightly easier. Since $\ell +1 >2$, the words $v_1,~\cdots,~ v_{\nu}$ represent elements in $G_{p-\ell +1,3} \hookrightarrow G^\lrcorner_{p,3}$. 

        In particular, it follows from the induction hypothesis \ref{item:IH-fromp-1top} and Lemma \ref{lem:omegainverse} that the identities 
        \begin{equation}\label{eq:inversion-of-Omega-words}
            [x_{1}^\eta, v_i] =_{F} (\Omega^{\ell +1}_{k_i+1}(\eta, \underline{n_i})^{-1})^{v_i^{-1}} =_{\calp} \Omega^{\ell +1}_{k_i+1}(\eta, \underline{n_i})^{-1}
        \end{equation}
        hold in $G^{\lrcorner}_{p,3}$ with area $\lesssim_p n^{p-\ell}$ for $1\leqslant i \leqslant \nu$. 
        
        Finally, the identities
        \begin{align*}
        [x_{1}^{\eta}, v_1\cdots v_{\nu} \cdot w] &=_F
        [x_{1}^{\eta}, w] \cdot \prod_{i=0}^{\nu-1}[x_{1}^{\eta}, v_{\nu-i} ]^{v_{\nu-i+1}\cdots v_{\nu} w}\\ 
        & =_\calp [x_{1}^{\eta}, w] \cdot \prod_{i=0}^{\nu-1}[x_{1}^{\eta}, v_{\nu-i} ]^{w}\\ 
        & =_\calp [x_{1}^{\eta}, w] \cdot [x_{1}^{\eta}, v_{\nu} ] \cdots [x_{1}^{\eta}, v_1]
        \end{align*}
        have area $\lesssim_{\nu, p} n^{p-\ell}$ in $G^\lrcorner_{p,3}$, where the first identity is free, the second holds in $G_{p-\ell+1,3}\hookrightarrow G_{p,3}^{\lrcorner}$, and for the last identity we apply the identity \eqref{eq:inversion-of-Omega-words} and the Second commuting $(k_i+1, \ell+1)$-Lemmas in $G^\lrcorner_{p,3}$ at most $\nu$ times.
    \end{proof}

    \begin{remark}
        In our applications of Lemma \ref{lem:separatingproducts} later we will always have $k_i \geqslant k-1$ meaning that in those case the above proof only relies on the Second commuting $(r, \ell+1)$-Lemmas for $r \geqslant k$. 
    \end{remark}

The following statement is a key ingredient of the induction step that allows us to decompose $\Omega_k^\ell$-words into a product of $\Omega^{\ell+1}$-words (see blue arrows/boxes in Figure \ref{fig:commuting-lemmas}).

\begin{lemma}[cf. {\cite[Proposition 6.21]{lipt}}]\label{lem:commutatortinyletter}
    Let $(k,\ell) \neq (2,2)$ with $k\neq 1$ and $j_0 \leqslant \ell \leqslant p-2$, $\underline{n}:= (n_1, \ldots, n_k) \in \R^k$, and $\beta:=n_{k-1}-\lfloor n_{k-1} \rfloor$. 
    
    If $n_{k-1} \geqslant 0$ , then the identity
    \begin{multline}\label{eq:tinypos}
        \Omega^{\ell}_{k}(\underline{n}) =_\calp \Omega_{k}^{\ell}(n_1, \ldots, \beta, n_k) \\
        \cdot \Big(\prod_{r=0}^{\lfloor n_{k-1} \rfloor-1} \Omega_{k}^{\ell+1}(n_1, \ldots, r+\beta, n_k)^{-1} \cdot \Omega_{k-1}^{\ell+1}(n_1, \ldots, n_{k-2}, n_k) \Big)
    \end{multline}
    holds in $G^\lrcorner_{p,3}$ for all $p \geqslant 5$ (respectively $G_{p,3}$ for all $p \geqslant 3$) with area $\lesssim_p |\underline{n}|^{p-\ell+1}$.

    If $n_{k-1} < 0$, then the identity \footnote{The formula for $n_{k-1}<0$ in the corresponding result in \cite{lipt} was slightly flawed and should read like the one given here. Besides producing a constant number of terms for every $r$ instead of only $2$ terms this has no impact on the proofs given there, which remain correct as written for the modified formula.}
    
    \begin{align}\label{eq:tinyneg}
        \Omega^{\ell}_{k}(\underline{n}) =_\calp \Omega_{k}^{\ell}(n_1, \ldots, \beta, n_k) \cdot \prod_{r=0}^{-\lfloor n_{k-1} \rfloor -1}\prod_{i=\ell +1}^{p} &\Omega_{k+ i-(\ell +1)}^{\ell + 1}(n_1, \ldots,n_{k-2}, \beta - r, 1, \ldots 1, -n_k)^{-1} \\ &\cdot \Omega_{k+ i-(\ell +1)-1}^{\ell +1}(n_1, \ldots, n_{k-2}, 1, \ldots,1, -n_k)\notag
    \end{align}
    holds in $G^\lrcorner_{p,3}$ for all $p \geqslant 5$ (respectively $G_{p,3}$ for all $p \geqslant 3$) with area $\lesssim_p |\underline{n}|^{p-\ell+1}$.
\end{lemma}

\begin{addendum}\label{add:nullhomotopicomegawords}
    For $3 \leqslant j_0 +1 \leqslant \ell \leqslant p$ and $k \geqslant p-\ell +2$ the words $\Omega_k^\ell(\underline{n})$ are null-homotopic in $G_{p-\ell+2,3} \hookrightarrow G^\lrcorner_{p,3}$ (respectively $G_{p,3}$). By the induction hypothesis \ref{item:IH-fromp-1top} their area is $\lesssim_p \delta_{G_{p-\ell+2,3}}(n) \lesssim_p n^{p-\ell+1}$.

    This observation has the following consequences in $G^\lrcorner_{p,3}$ and $G_{p,3}$:
    \begin{enumerate}
        \item For $j_0 +1 \leqslant \ell \leqslant p$ and $k \geqslant p-\ell +2$, the First and Second commuting $(k,\ell)$-Lemmas may be interpreted as the process of removing and creating the $\Omega_k^\ell$-words with area $\lesssim_p |m|n^{p-\ell} + n^{p-\ell+1}$ and $\lesssim_p n^{p-\ell+1}$ respectively;
        \item For $j_0 +1 \leqslant \ell \leqslant p$ and the pairs $(p-\ell+1,\ell)$, in the identity \eqref{eq:tinypos} (respectively \eqref{eq:tinyneg}) the $|\lfloor n_{k-1} \rfloor|$ appearances of the $\Omega_{p-\ell+1}^{\ell +1}$-words (respectively the $\Omega_{p-\ell+1}^{\ell+1}$-,$\ldots, \Omega_{2p-1}^{\ell+1}$-words) can be removed with area $\lesssim_p |\lfloor n_{k-1} \rfloor| \cdot \delta_{G_{p-\ell+1,3}}(n) \lesssim_p n^{p-\ell+1}$. For notational convenience we do not distinguish cases and leave these terms in the identities; 
        \item  For all $k \geqslant p$ by Lemma \ref{lem:commutatortinyletter} and using the induction hypothesis \ref{item:IH-fromp-1top} at most $n$ times we get that $\Omega_k(\underline{n})$ is a null-homotopic word with area $\lesssim_p n^{p-1}$.
    \end{enumerate}
\end{addendum}

\begin{proof}[Proof of Lemma \ref{lem:commutatortinyletter}]  
      We assume $n_{k-1} \geqslant 0$. The proof for $n_{k-1} <0$ is analogous, but produces slightly longer terms when we use the corresponding formula from Addendum \ref{add:omega3s}. The proof is by induction in $\left\lfloor n_{k-1}\right\rfloor$. The case $\left\lfloor n_{k-1}\right\rfloor=0$ is trivial and we now assume that we have proven the result for $\left\lfloor n_{k-1}\right\rfloor-1$.
     
     We deduce from $k\neq 1$ and Corollary \ref{cor:oneandjnotequalsoneplusj} that
    {\small
    \begin{align*}
        \Omega^{\ell}_{k}(\underline{n}) &=_\calp \left\{ \begin{array}{ll}\left[x_{1}^{n_1},\ldots, x_{1}^{n_{k-2}},x_{\ell+1}^{n_k} \cdot [x_{\ell+1}^{n_{k}},x_{1}^{n_{k-1}-1}] \cdot [x_{1}^{n_{k-1}-1},x_{\ell}^{n_k}]\cdot [y_{1}^{\tilde{n}},y_{p-1}^{\tilde{n}}]\right],~&\ell=2\\
        &\\
        \left[x_{1}^{n_1},\ldots, x_{1}^{n_{k-2}},x_{\ell+1}^{n_k} \cdot [x_{\ell+1}^{n_{k}},x_{1}^{n_{k-1}-1}] \cdot [x_{1}^{n_{k-1}-1},x_{\ell}^{n_k}]\right],~ & \ell\neq 2
        \end{array}\right. \label{eq:identity-(2,2)}\\
        &=_\calp \left[x_{1}^{n_1},\ldots, x_{1}^{n_{k-2}},x_{\ell+1}^{n_k} \cdot [x_{\ell+1}^{n_{k}},x_{1}^{n_{k-1}-1}] \cdot [x_{1}^{n_{k-1}-1},x_{\ell}^{n_k}]\right], \nonumber
    \end{align*}}%
    where the last identity holds in $G^\lrcorner_{p,3}$ at cost $\lesssim_p |\underline{n}|^2$, since the $y_i$ commute with the $x_i$. Therefore, by applying Lemma \ref{lem:separatingproducts} to the above equality for $\nu=2$ and for the words $v_1:=x_{\ell+1}^{n_k}$, $v_2 := [x_{1}^{n_{k-1}-1}, x_{\ell+1}^{n_{k}}]^{-1}$, and $w := [x_{1}^{n_{k-1}-1},x_{\ell}^{n_k}]$,\footnote{Note that in the case $n_{k-1}<0$, $\nu \leqslant p$.}  we get that the identity
     {\small  
    \begin{align*}
        &\left[x_{1}^{n_1},\ldots, x_{1}^{n_{k-2}},x_{\ell+1}^{n_k} \cdot [x_{\ell+1}^{n_{k}},x_{1}^{n_{k-1}-1}] \cdot [x_{1}^{n_{k-1}-1},x_{\ell}^{n_k}]\right] =_\calp\\
         &\left[x_{1}^{n_1},\ldots, x_{1}^{n_{k-3}},[x_{1}^{n_{k-2}}, x_{\ell+1}^{n_k}] \cdot [x_{1}^{n_{k-2}},[x_{1}^{n_{k-1}-1}, x_{\ell+1}^{n_{k}}]^{-1}] \cdot [x_{1}^{n_{k-2}},[x_{1}^{n_{k-1}-1},x_{\ell}^{n_k}]]\right].
    \end{align*}
    }%
    holds in $G^{\lrcorner}_{p,3}$ with area $\lesssim_p |\underline{n}|^{p-\ell}$. By Lemma \ref{lem:omegainverse} the identity
    \[
        \left[x_{1}^{n_{k-2}},[x_{1}^{n_{k-1}-1}, x_{\ell+1}^{n_{k}}]^{-1}\right] =_\calp \left[ x_{1}^{n_{k-2}}, [x_{1}^{n_{k-1}-1}, x_{\ell+1}^{n_{k}}]\right]^{-1}
    \]
    holds with area $\lesssim_p |\underline{n}|^{p-\ell}$.
    
    Overall, on the first iteration of applying Lemma \ref{lem:separatingproducts} we get that the identity
    {\small  
    \begin{align*}
        \Omega^{\ell}_{k}&(\underline{n}) =_\calp \left[x_{1}^{n_1},\ldots, x_{1}^{n_{k-3}},[x_{1}^{n_{k-2}}, x_{\ell+1}^{n_k}] \cdot \left[ x_{1}^{n_{k-2}}, [x_{1}^{n_{k-1}-1}, x_{\ell+1}^{n_{k}}]\right]^{-1} \cdot [x_{1}^{n_{k-2}},[x_{1}^{n_{k-1}-1},x_{\ell}^{n_k}]]\right]
    \end{align*}
    }%
    holds in $G^{\lrcorner}_{p,3}$ with area $\lesssim_p |\underline{n}|^{p-\ell}$. Iterating this argument $O_p(1)$-times we get that the identity
    \begin{align*}
        \Omega^{\ell}_{k}(\underline{n}) =_\calp \ &\Omega^{\ell}_{k}(n_{1}, \ldots, n_{k-1}-1,n_k) \cdot \Omega^{\ell+1}_{k}(n_{1}, \ldots, n_{k-1}-1,n_k)^{-1} \\
        &\cdot \Omega^{\ell+1}_{k-1}(n_{1}, \ldots, n_{k-2},n_k)
    \end{align*}
    holds in $G^{\lrcorner}_{p,3}$ with area $\lesssim_p |\underline{n}|^{p-\ell}$. Applying our induction hypothesis for $\left\lfloor n_{k-1}\right\rfloor-1$ to $\Omega_k^\ell(n_1,\dots, n_{k-1}-1,n_k)$ shows that identity \eqref{eq:tinypos} holds in $G_{p,3}^{\lrcorner}$ with area $\lesssim_p |\underline{n}|^{p-\ell+1}$.
\end{proof}

If $(k,j)=(2,2)$ the same arguments as in the proof of Lemma \ref{lem:commutatortinyletter} yield a similar identity with the only difference that now the terms of the form $[y_1^{\tilde{n}},y_{p-1}^{\tilde{n}}]$ do not cancel. We record it below.

\begin{addendum}\label{add:(2,2)}
    Let $p\geqslant 5$, $(k,j)=(2,2)$, $\underline{n}=(n_1,n_2)\in \R^k$ and $\beta=n_1-\left\lfloor n_1\right\rfloor$. 
    
    If $n_1\geqslant 0$, then the identity
    \begin{equation}\label{eq:claimE1}
        \Omega_{2}(\underline{n}) =_{\calp} [x_{1}^{\beta}, x_{2}^{n_{2}}] \cdot \left(\prod_{r=0}^{\lfloor n_{1} \rfloor -1} \Omega^3_1(n_2)\cdot \Omega^3_2(r+\beta,n_2)^{-1} \cdot  [y_{1}^{\tilde{n}}, y_{p-1}^{\tilde{n}}]^{\pm 1} \right),
    \end{equation}
    holds in $G^{\lrcorner}_{p,3}$ for all $p \geqslant 5$ with area $\lesssim_p |\underline{n}|^{p-1}$ and $|\tilde{n}| \lesssim_{p} |n_2|$.
    
    Analogously, if $n_1 <0$, then the identity
    {\small\begin{align}\label{eq:claimE2}
        \Omega_{2}(\underline{n}) =_{\calp} [x_{1}^{\beta}, x_{2}^{n_{2}}] \cdot  \Big( \prod_{r=0}^{- \lfloor n_{1} \rfloor -1} \Big( \prod_{i=3}^{p} & \Omega_{i-1}^{3}(\beta - r,1, \ldots, 1, -n_2)^{-1} \\
        &\cdot \Omega_{i-2}^{3}(1, \ldots, 1, -n_2) \Big)\cdot  [y_{1}^{\tilde{n}}, y_{p-1}^{\tilde{n}}]^{\pm1} \Big),\notag
    \end{align}}%
    holds in $G^{\lrcorner}_{p,3}$ for all $p \geqslant 5$ with area $\lesssim_p |\underline{n}|^{p-1}$ and $|\tilde{n}| \lesssim_{p} |n_2|$.
\end{addendum}

\begin{remark}\label{rem:(2,2)decomposition-gp3}
    For all $p \geqslant 3$ the identity \eqref{eq:claimE1} (respectively \eqref{eq:claimE2}) without the error terms $[y_{1}^{\tilde{n}}, y_{p-1}^{\tilde{n}}]$ holds in $G_{p,3}$ with area $\lesssim_p |\underline{n}|^{p-1}$.
\end{remark}

The following result for $(k,\ell)=(1,j_0)$ is a key ingredient in our derivation of the case $(1,j_0)$ of the Commuting Lemmas from the cases $(k,j_0)$ for $k\geqslant 2$ (see the red arrows/boxes in Figure \ref{fig:commuting-lemmas}.) In particular, in its proof we can assume that the Commuting $(k,j_0)$-Lemmas with $k\geqslant 2$ have already been proved. We prove a more general version for arbitrary $(k,\ell)$, as this is required in other parts of this work (after the proofs of the Commuting $(k,\ell)$-Lemmas).

\begin{lemma}[cf. {\cite[Lemma 6.20]{lipt}}]\label{lem:omega-u}
    Let $p \geqslant 5$, $n \geqslant 1$, $\alpha \geqslant 0$, $j_0 \leqslant \ell \leqslant p-1$, $k \geqslant 1$, and $\underline{n} \in \R^k$ with $|\underline{n}| \lesssim_p n$. For $u:= u(x_1,x_2) \in \mathcal{F}[2\alpha]$ with $|u| \leqslant n$, the identities
    \[
    [\Omega_k ^\ell (\underline{n})^{\pm 1}, u ]=_\calp 
    \left\{\arraycolsep=1.4pt\def\arraystretch{2.2}
    \begin{array}{ll}
         \prod\limits_{i=1}^{\nu} \Omega_{l_i}^{\ell}(\underline{\eta}_{i})^{\pm1}, &\ \text{if} \ k+\ell \geqslant 5, \\
         \left( \prod\limits_{i=1}^{\nu} \Omega_{l_i}^{\ell}(\underline{\eta}_{i})^{\pm1} \right) \left( \prod\limits_{i=1}^{\mu} \Omega_2 ^{p-1}(\underline{\hat{n}_i})^{\pm 1} \right), &\ \text{if} \ (k,\ell)=(1,3),\\
         \left( \prod\limits_{i=1}^{\nu} \Omega_{l_i}(\underline{\eta}_{i})^{\pm1} \right) \left( \prod\limits_{i=1}^{\mu} \Omega_{3}^{p-2}(\tilde{\underline{n}}_i)^{\pm 1} \right), &\ \text{if} \ (k,\ell)=(2,2).
    \end{array}\right.
    \]
    hold in $G^\lrcorner_{p,3}$ with area $\lesap n^{p-\ell+1}$, for suitable $\nu = O_{p,\alpha}(1)$, $\mu \leqslant 2\alpha$, $l_i \geqslant k+1$, $\underline{\eta}_i \in \R^{l_i}$ with $|\underline{\eta}_i| \lesssim_p n$ for $1 \leqslant i \leqslant \nu$, $\hat{n}_i \in \R^2$, and $\tilde{n}_i \in \R^3$ with $|\hat{n}_i|,|\tilde{n}_i| \lesssim_p n$ for $1\leqslant i \leqslant \mu$.
\end{lemma}
\begin{proof}       
    We start by proving the identities for $k+\ell \geqslant 5$ for $\Omega_k ^\ell (\underline{n})$, the proof for $\Omega_k ^\ell (\underline{n})^{-1}$ being similar. The proof for $k+\ell \geqslant 5$ relies only on the First commuting $(k+r,\ell)$-Lemmas for $G^\lrcorner_{p,3}$ where $r \geqslant 1$. 
    
    Since $u \in \mathcal{F}[2\alpha]$, there exists $\mu \leqslant 2\alpha$ such that 
    {\small
    $$
    u := \prod_{i=1}^{\mu} x_1 ^{\beta_i} x_2 ^{\gamma_i}.
    $$
    }%
    Without loss of generality we can assume $\mu = 2\alpha$. The proof is done by induction on $2\alpha$. For the base of induction, $\alpha=0$, the statement is trivially true with $\nu =\mu =0$, with the convention that in this case the product is empty. Suppose that the statement is true for  $2 \alpha \geqslant 0$ and let $u \in \mathcal{F}[2(\alpha +1)]$. The following identities hold in $G^{\lrcorner}_{p,3}$. The area estimates are explained below.
    \begingroup
    \allowdisplaybreaks
    {\small  
    \begin{align*}
        \Omega_k^\ell(\underline{n}) \cdot \prod_{i=1}^{\mu} x_1 ^{\beta_i} x_2 ^{\gamma_i} \stackrel{\ref{eq:1-omega-u}}{=}_\calp& \ x_1^{\beta_1} \cdot  \Omega_k^\ell(\underline{n}) \cdot \Omega_{k+1}^\ell(\beta_1, \underline{n})^{-1} \cdot x_2^{\gamma_1} \cdot \prod_{i=2}^{\mu} x_1 ^{\beta_i} x_2 ^{\gamma_i}\\
        \stackrel{\ref{eq:2-omega-u}}{=}_\calp & \ x_1^{\beta_1}x_2^{\gamma_1} \cdot \Omega_k^\ell(\underline{n}) \cdot \Omega_{k+1}^\ell(\beta_1, \underline{n})^{-1} \cdot \prod_{i=2}^{\mu} x_1 ^{\beta_i} x_2 ^{\gamma_i} \\
        \stackrel{\ref{eq:3-omega-u}}{=}_\calp& \left( \prod_{i=1}^{\mu} x_1 ^{\beta_i} x_2 ^{\gamma_i} \right) \cdot \Omega_k^\ell(\underline{n})\cdot \left( \prod_{i=1}^{\nu} \Omega_{l_{i}}^{\ell}(\underline{\eta}_{i})^{\pm1}\right)
    \end{align*}
    }%
    \endgroup
    \begin{enumerate}[label={(\arabic*)}]
        \item is a consequence of Corollary \ref{cor:growingomega}\footnote{In fact this is a free identity for $\Omega_k^{\ell}(\underline{n})$, but Corollary \ref{cor:growingomega} is required for $\Omega_k^{\ell}(\underline{n})^{-1}$.}, so it has area $\lesssim_p \max\{n^{p-\ell+1}, n^2\} \leqslant n^{p-\ell+1}$; \label{eq:1-omega-u}
        \item is obtained as follows: note that $k +\ell \geqslant 5$, if $k=1$, then this follows because $[x_2^{\gamma_2}, x_\ell^{n_1}]$ has area $\lesssim_p n^2$ for all $\ell \geqslant 4$. Otherwise, by the induction hypothesis \ref{item:IH-gp3corner} we can apply the First commuting $(k, \ell)$-and $(k+1, \ell)$-Lemmas for $G^\lrcorner_{p,3}$. Therefore, it has area $\lesssim_p n^{p-\ell+1}$. In either case it has area $\lesssim_p n^{p-\ell+1}$;\label{eq:2-omega-u}
        \item is a consequence of the induction hypothesis for $2\alpha$. More specifically, we apply the induction hypothesis first to the $\Omega_{k+1}^\ell$-word to obtain the error term {\small$\prod_{i=1}^{\nu_2} \Omega_{l_{2,i}}^{\ell}(\underline{\eta}_{2,i})^{\pm1}$}, we then apply the induction hypothesis on $2\alpha$ to the $\Omega_k^\ell$-word to obtain the error term {\small$\prod_{i=1}^{\nu_1} \Omega_{l_{1,i}}^{\ell}(\underline{\eta}_{1,i})^{\pm1}$}, where $\nu_1, \nu_2 = O_{\alpha,p}(1)$, $l_{1,i} \geqslant k+1$, $l_{2,i} \geqslant k+2$, $\underline{\eta}_{1,i} \in \R^{l_{i,1}}$, and $\underline{\eta}_{2,i} \in \R^{l_{2,1}}$ with $|\underline{\eta}_{1,i}|, |\underline{\eta}_{2,i}| \lesssim_p n$. Finally, we merge these terms into the single product {\small$\prod_{i=1}^{\nu} \Omega_{l_{i}}^{\ell}(\underline{\eta}_{i})^{\pm1}$}. In total, this has area $\lesssim_{p,\alpha} n^{p-\ell+1}$.\label{eq:3-omega-u}
    \end{enumerate}
    This concludes the proof for the pairs $(k, \ell)$ with $k+\ell \geqslant 5$. The proofs for the pairs $(1,3)$ and $(2,2)$ are analogous. However, some additional error terms are being produced. Rather than repeating the line of argument several times, we just give a brief explanation of the main differences and leave the details to the reader.
    \begin{itemize}[label=\bfseries, leftmargin=*,labelindent=6.5em]
        \item[\textbf{Case (1,3):}] In this case we also have to commute terms of the form $x_3^{n_1}$ with terms of the form $x_2 ^{\gamma_i}$, which produces additional error terms of the form $[x_3^{n_1},x_2^{\gamma_i}]$. However, since $\langle x_3,x_2,y_1,y_{p-1} \rangle$  and $\langle x_1,x_{p-1},y_1,y_{p-1} \rangle$ generate subgroups of $G_{p,3}^{\lrcorner}$ isomorphic to the 5-Heisenberg group, each such commutation can be performed by producing an error term of the form $\Omega_2^{p-1}(\hat{\underline{n}}_i)$ on the very right with area $\lesssim_p n^2$. In particular, the total area of the transformations is still $\lesssim_p n^{p-2}$.
        \item[\textbf{Case (2,2):}] This is similar to the case (1,3) only that now the error terms $\Omega_3^{p-2}(\tilde{\underline{n}}_i) $ arise from commutators of the form $[\Omega_2(\underline{n}),x_2^{\gamma_i}]$. Which by the First commuting $(2,2)$-Lemma \ref{lem:almostkjcommuting}\footnote{At the point where we use the identity for (2,2) we have already proved the First commuting $(2,2)$-Lemma for $G^\lrcorner_{p,3}$.} for $G^\lrcorner_{p,3}$ has area $\lesssim_p n^{p-1}$. \qedhere
    \end{itemize}
\end{proof}
\begin{addendum}\label{add:omega-u-gp3}
    The above proof for the pairs $(k,\ell)$ with $k+\ell \geqslant 5$ translates  without modification to $G_{p,3}$ with $p \geqslant 3$ for all pairs $(k,\ell)$. So we get that the identity $$  [\Omega_k ^\ell (\underline{n})^{\pm 1}, u ]=_\calp \prod_{i=1}^{\nu} \Omega_{l_i}^{\ell}(\underline{\eta}_{i})^{\pm1} $$ holds in $G_{p,3}$ for all $p \geqslant 3$ with area $\lesap n^{p-\ell+1}$ for suitable $\nu = O_{p,\alpha}(1)$, $l_i \geqslant k+1$, and $\underline{\eta}_i \in \R^{l_i}$ with $|\underline{\eta}_i| \lesssim_p n$ for all  $1 \leqslant i \leqslant \nu$. 
\end{addendum}

With the required auxiliary results we can now complete the induction step for the First commuting $(k, j_0)$-Lemma \ref{lem:almostkjcommuting} and the Second commuting $(k,j_0$)-Lemma \ref{lem:kjcommuting} for $G^\lrcorner_{p,3}$. 

\subsubsection*{Proof of the induction step for $j$} \label{sec:inductionstep}
Recall that the induction in $j$ for $G^\lrcorner_{p,3}$ is done for all pairs $(k,j)$ such that $k + j \geqslant 5$.
So assume $k+j_0\geqslant 5$. As mentioned before we distinguish the Cases (i) $k\neq 1$ and (ii) $k=1$. \vspace{4pt}

\emph{Case (i) $k \neq 1$.} Let $W \in \{x_{2}^m, w(x_1,x_2)\}$ with $w:= w(x_1,x_2)$ a word in the derived subgroup of $G_{p,3}^{\lrcorner}$. We need to show that the area of the null-homotopic word $[\Omega_{k}^{j_0}(\underline{n}), W]$ is $\lesap |W|\cdot n^{p - j_0} + n^{p - j_0 + 1}$. We assume that $n_k \geqslant 0$, the proof for $n_k <0$ is analogous, but involves slightly more terms. Let $\beta := n_{k-1}-\lfloor n_{k-1} \rfloor$. We have the following identities in $G^\lrcorner_{p,3}$ whose area estimates are explained below.
\begingroup
\allowdisplaybreaks
{\small  
\begin{align} 
    \Omega_{k}^{j_0}(\underline{n}) \cdot W \stackrel{\ref{eq:1-commutatorstinyletter}}{=}_\calp
    & \Big(\Omega_{k}^{j_0}(n_1, \ldots, \beta, n_k) \cdot \label{eq:commuting-omegakj-W}\\
    &\prod_{r=0}^{\lfloor n_{k-1} \rfloor-1} \left(\Omega_{k}^{j_{0}+1}(n_1, \ldots, r+\beta, n_k)^{-1} \cdot \Omega_{k-1}^{j_{0}+1}(n_1, \ldots, n_{k-2}, n_k) \right) \Big) \cdot W \notag \\
    \stackrel{\ref{eq:2-commutatorstinyletter}}{=}_\calp& \ \Omega_{k}^{j_0}(n_1, \ldots, \beta, n_k) \cdot W \notag \\
    \quad & \cdot \Big( \prod_{r=0}^{\lfloor n_{k-1} \rfloor-1} \Omega_{k}^{j_{0}+1}(n_1, \ldots, r+\beta, n_k)^{-1} \cdot \Omega_{k-1}^{j_{0}+1}(n_1, \ldots, n_{k-2}, n_k) \Big) \notag \\
    \stackrel{\ref{eq:3-commutatorstinyletter}}{=}_\calp& \left( \prod_{r = j_{0}+1}^{p-1} \Omega_{k-1}^{r}(n_1, \ldots, n_{k-2}, s_r) \right) \cdot W \notag \\
    \quad & \cdot \Big( \prod_{r=0}^{\lfloor n_{k-1} \rfloor-1} \Omega_{k}^{j_{0}+1}(n_1, \ldots, r+\beta, n_k)^{-1} \cdot \Omega_{k-1}^{j_{0}+1}(n_1, \ldots, n_{k-2}, n_k) \Big) \notag \\
    \stackrel{\ref{eq:4-commutatorstinyletter}}{=}_\calp& \ W \cdot \left( \prod_{r = j_{0}+1}^{p-1} \Omega_{k-1}^{r}(n_1, \ldots, n_{k-2}, s_{r}) \right) \notag \\
    \quad & \cdot \Big( \prod_{r=0}^{\lfloor n_{k-1} \rfloor-1} \left(\Omega_{k}^{j_{0}+1}(n_1, \ldots, r+\beta, n_k)^{-1} \cdot \Omega_{k-1}^{j_{0}+1}(n_1, \ldots, n_{k-2}, n_k) \right) \Big) \notag \\
    \stackrel{\ref{eq:5-commutatorstinyletter}}{=}_\calp& \ W \cdot \Omega_{k}^{j_{0}}(\underline{n}),\notag
\end{align}
}%
\endgroup

\begin{enumerate}[label={(\arabic*)}]
    \item follows from Lemma \ref{lem:commutatortinyletter}, so the area is $\lesssim_p |\underline{n}|^{p-j_0+1}$;\label{eq:1-commutatorstinyletter}
    \item is a consequence of the induction hypothesis \ref{item:IH-gp3corner} applied $2\cdot \lfloor n_k \rfloor$ times \footnote{In the case $n_{k-1}<0$ there are $2(p-j_0)|\lfloor n_{k-1} \rfloor|$ terms appearing.}. Therefore, it has area $\lesap |W|n^{p-j_0} + n^{p-j_{0}+1}$;\label{eq:2-commutatorstinyletter}
    \item follows from applying Lemma \ref{lem:laderlemma}, since $|\beta| \leqslant 1$, and then Lemma \ref{lem:omegacommuteswithproducts}. Therefore, it has area $\lesssim_p n^{p-j_0}$;\label{eq:3-commutatorstinyletter}
    \item follows from the induction hypothesis \ref{item:IH-gp3corner} applied $O_{p}(1)$  times, thus has area $\lesap |W|n^{p-j_{0}-1} + n^{p-j_{0}}$;\label{eq:4-commutatorstinyletter}
    \item follows from Lemma \ref{lem:commutatortinyletter}, Lemma \ref{lem:laderlemma}, and Lemma \ref{lem:omegacommuteswithproducts}, so it has area $\lesssim_p |\underline{n}|^{p-j_0+1}$.\label{eq:5-commutatorstinyletter}
\end{enumerate}
Overall, the null-homotopic word $[\Omega_{k}^{j_0}(\underline{n}), W]$ has area $\lesap |W|n^{p-j_0} + n^{p-j_{0}+1}$. This finishes the proof for all pairs $(k,j_0)$ with $k\neq 1$ and $k + j_0 \geqslant 5$.\vspace{4pt}

\emph{Case (ii) $k=1$.} In this case $\Omega_1^{j_0}(\underline{n}):= x_{j_0}^{n_1}$. Observe that since $1+ j_0 \geqslant 5$ the First commuting $(1,j_0)$-Lemma for $G^\lrcorner_{p,3}$ follows from the relation $[x_{j_0}, x_2]=_\calp 1$. For the Second commuting $(1,j_0)$-Lemma for $G^\lrcorner_{p,3}$, we have that $W=_{\text{def}} w(x_1,x_2) \in \mathcal{F}[\alpha]$ represents an elements in the derived subgroup of $G^\lrcorner_{p,3}$. So from Lemma \ref{lem:omega-u} for $1 +j_0 \geqslant 5$ we obtain the identity 
{\small
\begin{equation*}
    [x_{j_0}^{n_1}, w] =_\calp \prod_{i=1}^{\nu} \Omega_{l_i}^{j_0}(\underline{\eta}_i)^{\pm 1}
\end{equation*}}%
in $G^\lrcorner_{p,3}$ with area $\lesssim_{\alpha,p} n^{p-j_0+1}$, where $\nu = O_{\alpha,p}(1)$, $l_i \geqslant 2$, and $\underline{\eta}_i \in \R^{l_i}$ with $|\underline{\eta}_i| \lesssim_p n$. 

Since $G^\lrcorner_{p,3}$ is metabelian, it follows that the identity $[x_{j_0}^{n_1}, w]=_\calp 1$ holds. Thus, $\prod_{i=1}^{\nu} \Omega_{l_i}^{j_0}(\underline{\eta}_i)^{\pm1}$ is null-homotopic in $G_{p-j_0+2,3}\hookrightarrow G_{p,3}^{\lrcorner}$. It follows from the induction hypothesis \ref{item:IH-fromp-1top} that it has area $\lesssim_{\alpha,p} n^{p-j_0 +1}$. Thus, the identity $[x_{j_0}^{n_1}, w]=_\calp 1$ has area $\lesssim_{\alpha,p} n^{p-j_0+1}$ in $G^\lrcorner_{p,3}$.

This concludes the proof of the induction step of the First and Second commuting $(k,j_0)$-Lemmas for $G^\lrcorner_{p,3}$ with $k+j_0 \geqslant 5$. \hfill \qedsymbol{} \medskip

\begin{addendum}\label{add:proof-first-second-gp3}
    The above proof translates verbatim to the group $G_{p,3}$ \emph{for all pairs $(k,j_0)$}, completing the proof of the First and Second commuting Lemmas for $G_{p,3}$; there is no need to treat the cases $(1,3)$ and $(2,2)$ separately for $G_{p,3}$.
\end{addendum}

We now address the cases $(1,3)$ and $(2,2)$ of the Commuting Lemmas for $G^\lrcorner_{p,3}$. Since the proof of these cases shares some arguments with the above proof, rather than repeating all arguments, we explain the parts where the proofs differ and provide details whenever a new argument is needed.

\subsubsection*{Proof of case $(1,3)$.} The First commuting $(1,3)$-Lemma for $G^\lrcorner_{p,3}$ follows from the fact that $\langle x_3,x_2,y_1,y_{p-1} \rangle$ and $\langle x_1,x_{p-1},y_1,y_{p-1} \rangle$ generate subgroups of $G_{p-1,3}^{\lrcorner}$ isomorphic to the 5-Heisenberg group, so that the identities $[x_3^{n_1}, x_2^m]=_\calp [y_1^{m}, y_{p-1}^{-n_1}] =_\calp \Omega_2^{p-1}(m, -n_1)$ hold in $G^\lrcorner_{p,3}$ with area $\lesssim_p |m|n$. 

For the proof of the Second commuting $(1,3)$-Lemma for $G^\lrcorner_{p,3}$ we proceed as in \emph{Case (ii)}. We first use Lemma \ref{lem:omega-u} for $(1,3)$ to obtain
{\small
\begin{equation} \label{eq:(1,3)}
    [x_3^{n_1}, w] =_\calp \left( \prod_{i=1}^{\nu} \Omega_{l_i}^{3}(\underline{\eta}_{i})^{\pm1} \right) \left( \prod_{i=1}^{\mu} \Omega_2 ^{p-1}(\underline{\hat{n}_i})^{\pm 1} \right)
\end{equation}  
}%
in $G^\lrcorner_{p,3}$ with area $\lesap n^{p-2}$. We then observe that the left side of \eqref{eq:(1,3)} is null-homotopic in $G^\lrcorner_{p,3}$, since $G_{p,3}^{\lrcorner}$ is metabelian, and thus the same is true for the right side. The induction hypothesis \ref{item:IH-fromp-1top} for $G_{p-1,3}\hookrightarrow G_{p,3}^{\lrcorner}$ then implies that the right side has area $\lesssim_{\alpha,p} n^{p-2}$.  Thus, the word $[x_3^{n_1}, w]$ is null-homotopic with area $\lesssim_{\alpha,p} n^{p-2}$ in $G^\lrcorner_{p,3}$. \hfill \qedsymbol{} \medskip

\subsubsection*{Proof of case $(2,2)$.} 
We start with the proof of the First commuting $(2,2)$-Lemma. Arguing as in \emph{Case (i)}, we first prove that the identities 
\begin{equation}\label{eq:step1}
    [\Omega_2(\underline{n}), x_2^m]=_\calp [y_1^{m}, y_{p-1}^{-n_2}]^{n_1}  =_\calp \Omega_2^{p-1}(m,-n_2)^{n_1}
\end{equation} 
hold in $G^\lrcorner_{p,3}$ with area $\lesssim_p |m|n^{p-2} + n^{p-1}$. We first use the decomposition of $\Omega_2(\underline{n})$ given by Addendum \ref{add:(2,2)} (as before we assume that $n_1\geqslant 0$ with the case $n_1<0$ being analogous, but involving slightly more terms). We then follow the chain of identities in \eqref{eq:commuting-omegakj-W} applying the First commuting $(1,3)$- and $(2,3)$-Lemmas at most $\lfloor n_1 \rfloor$ times. Each application of the First commuting $(1,3)$-Lemma produces an error term of the form $[y_1^m, y_{p-1}^{-n_2}]$. The first identity in \eqref{eq:step1} then follows by shifting each of these error terms to the right, with total area $\lesssim_p n \cdot (|m|n^{p-3} + n^{p-2})$. The second identity in \eqref{eq:step1} follows with area $\lesssim_p |m|n^2$ by arguing with $5$-Heisenberg subgroups, similar as in Case $(1,3)$.

For the moreover-part of the First commuting $(2,2)$-Lemma it now suffices to show that the identity
\begin{equation*}
    \Omega_2^{p-1}(m,-n_2)^{n_1} =_\calp \Omega_3^{p-2}(m, n_1, -n_2)
\end{equation*} 
holds in $G^\lrcorner_{p,3}$ with area $\lesssim_p n^{p-1}.$ For this observe that one can first argue as in the proof of Lemma \ref{lem:commutatortinyletter}, expanding the commutator $[x_1^{n_1},x_{p-2}^{-n_2}]$, to obtain that the identity
\[
    \left[x_1^m,[x_1^{n_1},x_{p-2}^{-n_2}]\right]=_\calp \left[x_1^m,x_{p-1}^{-n_2}\right]^{n_1}
\]
holds in $G_{4,3}\hookrightarrow G_{p,3}^{\lrcorner}$ with area $\lesssim_p n^4$, since $|m|\leqslant n$. Combining this with \eqref{eq:step1} and using again that $|m|\leqslant n$, the assertion follows.

The Second commuting $(2,2)$-Lemma is done as follows. For simplicity assume that $n_1 \geqslant 0$, the proof for $n_1 <0$ is analogous but involves slightly more terms and thus uses more instances of the First and Second commuting Lemmas. First, from Addendum \ref{add:(2,2)} we obtain the identity $$\Omega_2(\underline{n})=_\calp [x_{1}^{\beta}, x_{2}^{n_{2}}] \cdot \left(\prod_{r=0}^{\lfloor n_{1} \rfloor -1} \Omega^3_1(n_{2})\cdot \Omega^3_2(r,n_2)^{-1} \cdot  [y_{1}^{\tilde{n}}, y_{p-1}^{\tilde{n}}] \right) $$ 
with area $\lesssim_p n^{p-1}$ in $G^\lrcorner_{p,3}$. 
Then, using the Second commuting $(1,3)$- and $(2,3)$-Lemmas for $G^\lrcorner_{p,3}$, and the fact that the $y_i$ commute with the $x_i$ we now proceed using the chain of identities in \eqref{eq:commuting-omegakj-W} to commute $w$ with $\Omega_2(\underline{n})$ with area $\lesap n^{p-1}$ in $G^\lrcorner_{p,3}$. \hfill \qedsymbol{} \medskip

This completes the proof of the First and Second commuting $(k,j)$-Lemmas for $G^\lrcorner_{p,3}.$ To close this section we record a consequence of Lemma \ref{lem:commutatortinyletter} and Lemma \ref{lem:omegacommuteswithproducts} that will only be used in the proof of the Cancelling $k$-Lemma \ref{lem:cancelling} and in Section \S\ref{sec:fromp3topqandbeyond}.

\begin{corollary}[cf. {\cite[Corollary 6.24]{lipt}}]\label{cor:omegapminus1}
    For all $\underline{n} \in \R^{p-1}$, there exists $m \in \R$ with $|m| \lesssim_p |\underline{n}|$ such that the identity 
    \begin{equation*}
        \Omega_{p-1}(\underline{n})^{\pm 1} =_\calp (\Omega_{p-2}^{3}(|\underline{n}|, \ldots, |\underline{n}|))^m
    \end{equation*}
    holds in $G^\lrcorner_{p,3}$ for all $p\geqslant 5$ (respectively $G_{p,3}$ for all $p \geqslant 3$) with area $\lesssim_p |\underline{n}|^{p-1}$.
\end{corollary}
\begin{proof}
    The proof is analogous to the one for \cite[Corollary 6.24]{lipt} so we omit it here (it uses Lemma \ref{lem:commutatortinyletter}, Lemma \ref{lem:omegacommuteswithproducts}, and \cite[Lemma 6.14]{lipt} which is stated for $G_{p-1,p-1}$ in \cite{lipt} but can be proved in precisely the same way for $G_{p-1,3}$).
\end{proof}

\subsubsection{Proof of the Main commuting Lemma for $G^\lrcorner_{p,3}$ and $G_{p,3}$}

To deduce the Main commuting Lemmas from the First and Second commuting $(k,j)$-Lemmas for $G^\lrcorner_{p,3}$ (respectively for $G_{p,3}$) we require the Reduction Lemma \ref{lem:reductionlemma}. Its proof follows the one of \cite{lipt}, but requires some modifications; we thus include it here. We start with two auxiliary results.

\begin{lemma}\label{lem:trick}
    Let $\underline{n} :=(n_1, n_2, n_3) \in \R^3$. The identity $$\Omega_3 ^{p-2} (n_1, n_2, n_3) =_\calp \Omega_{p-1} (n_1, n_2, 1, \ldots, 1, n_3)$$ holds in $G^{\lrcorner}_{p,3}$ for $p \geqslant 5$ (respectively $G_{p,3}$ for $p \geqslant 3$) with area $\lesssim_p |\underline{n}|^2$.
\end{lemma}
\begin{proof}
    In $G_{p,3}^{\lrcorner}$ this is a straight-forward consequence of iteratively applying the identities $x_{j}^{n_3} =_\calp [x_1, x_{j-1}^{n_3}]$ for $3 \leqslant j \leqslant p-2$ and the identity $\Omega_1^3(n_3)=_\calp [y_1^{\tilde n_3},y_3^{\tilde n_3}]^{\pm 1}\Omega_2(1,n_3)$, which hold with area $\lesssim_p |\underline n|^2$. The proof for $G_{p,3}$ is analogous, but simpler.
\end{proof}

Using Lemma \ref{lem:trick} we can rewrite the identity in Lemma \ref{lem:omega-u} for $(2,2)$ as follows.

\begin{corollary}[cf. {\cite[Lemma 6.26]{lipt}}]\label{cor:movingomega2}
    Let $p \geqslant 5$, $n\geqslant 1$, $\alpha \geqslant 0$, and $\underline{n} \in \R^2$. If $u:= u(x_1, x_2) \in \mathcal{F}[\alpha]$ with $|u| \leqslant n$, then there exists $M= O_{\alpha,p}(1)$ such that the identity $$ \Omega_2 (\underline{n})^{\pm1} \cdot u =_\calp u \cdot \prod_{i=1}^{M} \Omega_{\ell_i}(\underline{\tilde{n}}_i)^{\pm1}$$ holds in $G^\lrcorner _{p,3}$ with area $\lesap n^{p-1}$ for suitable $2 \leqslant \ell_i \leqslant p-1$, and $\underline{\tilde{n}}_i \in \R^{\ell_i}$ with $|\underline{\tilde{n}}_i| \lesap n$.
\end{corollary}
\begin{proof}
    Apply the identity from Lemma \ref{lem:trick} at most $O_{\alpha,p}(1)$-times to the identity in Lemma \ref{lem:omega-u} for $(2,2)$ to obtain the desired identity and the area estimate. Finally, for all $l_i \geqslant p$ we remove the $\Omega_{l_i}$-words as explained in Addendum \ref{add:nullhomotopicomegawords}(3) with area $\lesap n^{p-1}$.
\end{proof}

Note that the above corollary is not needed for $G_{p,3}$ since Addendum \ref{add:omega-u-gp3} to Lemma \ref{lem:omega-u} already gives us the required identity.

\begin{proof}[Proof of the Reduction Lemma \ref{lem:reductionlemma}]
    The proof is done by induction on $2\alpha$. The base of induction, $\alpha=1$, is trivial with $L=0$ and the product being empty, since if $w=x_1 ^{\beta_1} x_2 ^{\gamma_1}$ represents an element in the derived subgroup of $G^\lrcorner _{p,3}$, then $\beta_1 = \gamma_1 = 0$. 
    
    Suppose that the statement is true for some $2\alpha \geqslant 1$ and let $w := \prod_{i=1}^{\mu} x_1 ^{\beta_i} x_2 ^{\gamma_i} \in \mathcal{F}[2(\alpha +1)]$. Observe that we can assume $\mu = 2(\alpha +1)$. The following identities hold in $G^{\lrcorner}_{p,3}$, the area estimates of which are explained below.
    \begingroup
    \allowdisplaybreaks
    {\small  
    \begin{align*}
        \prod_{i=1}^{\mu} x_1 ^{\beta_i} x_2 ^{\gamma_i} &\stackrel{\ref{eq:1-(2,2)omega-u}}{=}_F x_1 ^{\beta_1 + \beta_2} x_2 ^{\gamma_1} \cdot [x_2 ^{\gamma_1}, x_1 ^{\beta_2}] \cdot x_2 ^{\gamma_2}\cdot \prod_{i=3}^{\mu} x_1 ^{\beta_i} x_2 ^{\gamma_i} \\
        &\stackrel{\ref{eq:2-(2,2)omega-u}}{=}_F x_1 ^{\beta_1 + \beta_2} x_2 ^{\gamma_1 + \gamma_2} \cdot [x_2 ^{\gamma_1}, x_1 ^{\beta_2}] \cdot \left[ [x_2 ^{\gamma_1}, x_1 ^{\beta_2}], x_2 ^{\gamma_2} \right] \cdot \prod_{i=3}^{\mu} x_1 ^{\beta_i} x_2 ^{\gamma_i}\\
        &\stackrel{\ref{eq:3-(2,2)omega-u}}{=}_\calp x_1 ^{\beta_1 + \beta_2} x_2 ^{\gamma_1 + \gamma_2} \cdot \Omega_2 (\beta_2, \gamma_1)^{-1} \cdot \Omega_3 ^{p-2}(\underline{\hat{\gamma}}_1) \cdot \prod_{i=3}^{\mu} x_1 ^{\beta_i} x_2 ^{\gamma_i} \\
        &\stackrel{\ref{eq:4-(2,2)omega-u}}{=}_\calp x_1 ^{\beta_1 + \beta_2} x_2 ^{\gamma_1 + \gamma_2} \cdot \Omega_2 (\beta_2, \gamma_1)^{-1} \cdot \Omega_{p-1}(\underline{\tilde{\gamma}}_1) \cdot \prod_{i=3}^{\mu} x_1 ^{\beta_i} x_2 ^{\gamma_i} \\
        &\stackrel{\ref{eq:5-(2,2)omega-u}}{=}_\calp x_1 ^{\beta_1 + \beta_2} x_2 ^{\gamma_1 + \gamma_2} \cdot \Omega_2 (\beta_2, \gamma_1)^{-1} \cdot \left( \prod_{i=3}^{\mu} x_1 ^{\beta_i} x_2 ^{\gamma_i} \right) \cdot \prod_{i=1}^{\nu} \Omega_{s_i}(\underline{\eta}_{i})^{\pm 1} \\
        &\stackrel{\ref{eq:6-(2,2)omega-u}}{=}_\calp x_1 ^{\beta_1 + \beta_2} x_2 ^{\gamma_1 + \gamma_2} \cdot \left( \prod_{i=3}^{\mu} x_1 ^{\beta_i} x_2 ^{\gamma_i} \right) \cdot \left( \prod_{i=1}^{N} \Omega_{l_i}(\underline{\tilde{n}}_i)^{\pm 1} \right) \cdot \left( \prod_{i=1}^{\nu} \Omega_{s_i}(\underline{\eta}_{i})^{\pm 1} \right) \\
        &\stackrel{\ref{eq:7-(2,2)omega-u}}{=}_\calp x_1 ^{\beta_1 + \beta_2} x_2 ^{\gamma_1 + \gamma_2} \cdot \left( \prod_{i=3}^{\mu} x_1 ^{\beta_i} x_2 ^{\gamma_i} \right) \cdot \prod_{i=1}^L \Omega_{l_i}(\underline{\hat{\eta}}_i)^{\pm 1} 
    \end{align*}
     }%
    \endgroup
    \begin{enumerate}[label={(\arabic*)}]
        \item is a free identity so it has no area;\label{eq:1-(2,2)omega-u}
        \item is also a free identity;\label{eq:2-(2,2)omega-u}
        \item is a consequences of the First commuting $(2,2)$-Lemma \ref{lem:almostkjcommuting} for $G^\lrcorner_{p,3}$ so it has area $\lesssim_p n^{p-1}$ where $|\underline{\hat{\gamma}}_1| \lesssim_p n$;\label{eq:3-(2,2)omega-u}
        \item follows from applying Lemma \ref{lem:trick}, thus it has area $\lesssim_p n^2$ where $|\underline{\tilde{\gamma}}_1| \lesssim_p n$;\label{eq:4-(2,2)omega-u}
        \item is a consequence of Lemma \ref{lem:omega-u} so it has area $\lesssim_p n^2$, where $|\underline{\eta}_i| \lesssim_{\alpha, p} n$, $\nu = O_{\alpha,p}(1)$, and $s_i \geqslant p$;\label{eq:5-(2,2)omega-u}
        \item follows from Corollary \ref{cor:movingomega2}, so it has area $\lesssim_p n^{p-1}$, where $l_i \geqslant 2$, $|\underline{\tilde{n}}_i| \lesssim_{\alpha, p} n$, and  $N = O_{\alpha,p}(1)$;\label{eq:6-(2,2)omega-u}
        \item follows by first removing all the $\Omega_{l_i}$- and $\Omega_{s_i}$-words with $l_i, s_i \geqslant p$ using Addendum \ref{add:nullhomotopicomegawords}(3). Since there are at most $O_{\alpha,p}(1)$ of them, this has total area $\lesssim_{\alpha,p} n^{p-1}$. We then merge the remaining $\Omega$-words into a single product without any area.\label{eq:7-(2,2)omega-u}
    \end{enumerate}
    The word $\upsilon(x_1, x_2) := x_1 ^{\beta_1 + \beta_2} x_2 ^{\gamma_1 + \gamma_2} \cdot \left( \prod_{i=3}^{\mu} x_1 ^{\beta_i} x_2 ^{\gamma_i} \right)$ has the same exponent sum for $x_1$ and $x_2$ as that of the word $w$, therefore it represents an elements in the derived subgroup of $G^{\lrcorner}_{p,3}$ of length at most $n$ and $\upsilon (x_1, x_2) \in \mathcal{F}[2\alpha]$. Hence applying the induction hypothesis for $2\alpha$ to the word $\upsilon (x_1, x_2)$ we obtain the desired result.
\end{proof}

\subsubsection*{Proof of the Main commuting Lemma \ref{lem:maincommuting}}
If both $w_1$ and $w_2$ are powers of $x_2$ the statement is clear. Otherwise, we apply the Reduction Lemma \ref{lem:reductionlemma} to $w_1$ to obtain the identity $w_1 =_\calp \prod_{i=1}^{L} \Omega_{l_i}(\underline{\eta}_i)^{\pm 1}$ in $G^\lrcorner _{p,3}$ with area $\lesssim_p n^{p-1}$ where $2 \leqslant l_i \leqslant p-1$. Thus, if $w_2$ represents an element of the derived subgroup of $G^\lrcorner _{p,3}$, then by applying the Second commuting $(l_i,2)$-Lemma \ref{lem:kjcommuting} for $G^\lrcorner_{p,3}$ we get the desired identity; if $w_2 := x_2 ^k$, then by applying the First commuting $(l_i,2)$-Lemma \ref{lem:almostkjcommuting} for $G^\lrcorner_{p,3}$ we get that the identity
{\small
$$
[w_1, w_2] =_\calp \prod\limits_{i=1}^{D} \Omega_3 ^{p-2}(\underline{\eta}_i)^{\pm1}
$$
}%
holds in $G^\lrcorner _{p,3}$ with area $\lesssim_p n^{p-1}$. Moreover, $D$ equals the number of $l_i$ that are equal to 2, in particular $D \leqslant L = O_{\alpha,p}(1)$. \hfill \qedsymbol{} \medskip

\subsubsection{Proof of the cancelling $k$-Lemma for $G^\lrcorner_{p,3}$}

As in \cite[Section 6.7]{lipt} the proof of the Cancelling $k$-Lemma \ref{lem:cancelling} is done by descending induction on $k$. For the base case $k=p-1$ suppose that for some positive integer $M_{p-1}$ and $\underline{n}_{p-1,i} \in \R^{p-1}$ with $|\underline{n}_{p-1,i}| \leqslant n$, the word 
{\small$$
w(x_1,x_2) := \prod_{i=1}^{M_{p-1}} \Omega_{p-1}(\underline{n}_{p-1,i})^{\pm1}
$$}%
is null-homotopic in  $G^\lrcorner_{p,3}$ for all $p \geqslant 5$ (respectively $G_{p,3}$ for all $p \geqslant 3$). It follows from Corollary \ref{cor:omegapminus1} that for all $1 \leqslant i \leqslant p-1$ there exist some $m_i \in \R$ with $|m_i| \lesssim_p |\underline{n}_{p-1,i}|$ such that the identity 
{\small
$$
\Omega_{p-1}(\underline{n}_{p-1,i})^{\pm1} =_\calp \left( \Omega_{p-2}^3(|\underline{n}_{p-1,i}|, \ldots, |\underline{n}_{p-1,i}|)\right)^{m_i}
$$
}%
holds in $G^\lrcorner_{p,3}$ for all $p \geqslant 5$ (respectively $G_{p,3}$ for all $p \geqslant 3$ ) with area $\lesssim_p n^{p-1}$. 

For all $1\leqslant i \leqslant p-1$, the word $\Omega_{p-2}^3(|\underline{n}_{p-1,i}|, \ldots, |\underline{n}_{p-1,i}|)$ represents an element in  $G_{p-1,3} \hookrightarrow G_{p,3}$ (respectively $G^\lrcorner_{p,3}$). In particular, the word
{\small$$
\prod_{i=1}^{M_{p-1}} \left( \Omega_{p-2}^3(|\underline{n}_{p-1,i}|, \ldots, |\underline{n}_{p-1,i}|)\right)^{m_i}
$$}%
is null-homotopic in $G_{p-1,3}$. Thus, by merging suitable pairs of $\Omega_{p-2}^3$-words into new shorter $\Omega_{p-2}^3$-words, in $\lesssim M_{p-1}n$ steps we conclude from the induction hypothesis \ref{item:IH-fromp-1top} for $G_{p-1,3}$ that it has area $\lesssim_{M_{p-1},p} n^{p-1}$. It now follows that the null-homotopic word $w(x_1,x_2)$ has area $\lesssim_{M_{p-1},p} n^{p-1}$ in $G^\lrcorner_{p,3}$ for all $p \geqslant 5$ (respectively $G_{p,3}$ for all $p \geqslant 3$). This proves the base case. 

We now assume that the statement of the Cancelling $(k+1)$-Lemma \ref{lem:cancelling} holds. The proof of the induction step is analogous to the one in \cite[Section 6.7]{lipt}. The only changes required are some small modifications in the proofs of the auxiliary results for $G_{p,3}^{\lrcorner}$ that we explain below. Once these have been proved the remainder of the induction step proceeds as in \cite[p.754]{lipt}. One may worry that the applications of the Main commuting Lemma \ref{lem:maincommuting} there could produce error terms. However, this is \emph{not} the case, since we only apply it for pairs of words in the derived subgroup. 

\begin{lemma}[{\cite[Lemma 6.27]{lipt}}]\label{lem:extracting}
    Let $n \geqslant 1$. For $2 \leqslant k \leqslant p-2$ and $\underline{n} \in \R^k$ with $|\underline{n}| \leqslant n$ an identity of the form
    \[
    \Omega_k (\underline{n})^{\pm 1} =_\calp x_{k+1}^{\ell} \cdot E_{p,k}(\underline{n})
    \]
    holds in $G^{\lrcorner}_{p,3}$ for $p \geqslant 5$ (respectively $G_{p,3}$ for all $p \geqslant 3$) with area $\lesssim_p n^{p-1}$, where $E_{p,k}(\underline{n})$ is equal to $\prod_{i= k+1}^{p-1} \Omega_{i}(\underline{m}_i)^{\pm 1}$ with $|\underline{m}_i| \lesssim_p n$. Moreover, $|\ell| \lesssim_p n^k$.
\end{lemma}
\begin{proof}
    The proof is analogous to the one of \cite[Lemma 6.27]{lipt}. It is done by induction in $m:= \lceil \log_2(\underline{n}) \rceil$. The induction step uses the Cutting in half $k$-Lemma \ref{lem:cuttinginhalf} stated below, which corresponds to \cite[Lemma 6.10]{lipt}; its proof requires a modification for $k=2$ with respect to the proof in \cite{lipt}. Aside from this, the same arguments as in \cite{lipt} can be applied without modification.\footnote{Note that $k \geqslant 2$, so whenever in the proof of \cite[Lemma 6.27]{lipt} the Main commuting Lemma is applied, we do not produce any error terms, since all the terms involved represent elements in the derived subgroup of $G^\lrcorner _{p,3}$.}
\end{proof}

\begin{lemma}[{\cite[Lemma 6.10]{lipt}}, Cutting in half $k$-Lemma]\label{lem:cuttinginhalf}
    Let $2 \leqslant k \leqslant p-2$ and $\underline{n} \in \R^k$. The following identities hold in $G^{\lrcorner}_{p,3}$ for $p \geqslant 5$ (respectively $G_{p,3}$ for $p \geqslant 3$):
    $$\Omega_k (2\underline{n}) =_\calp \Omega_k (\underline{n})^{2^k} \cdot w_k (\underline{n}) \quad \text{and} \quad \Omega_k (2\underline{n}) =_\calp w_k (\underline{n}) \cdot \Omega_k (\underline{n})^{2^k},$$
    where $w_k (\underline{n}) =_\calp \prod_{i=1}^L \Omega_{l_i}(\underline{\eta}_i)^{\pm 1}$ with $L = O_p (1)$, $l_i \geqslant k+1$, and $|\underline{\eta}_{i}| \lesssim_p |\underline{n}|$ for $1 \leqslant i \leqslant L$. Moreover, each of the identities has area $\lesssim_p |\underline{n}|^{p-1}$.
\end{lemma}
\begin{proof}
    Like for \cite[Lemma 6.10]{lipt} the proof is by ascending induction in $k$. The case $G^\lrcorner_{p,3}$ requires a minor adjustment for  $k=2$, which we explain below. The remainder of the induction step is completely analogous to the one in \cite{lipt} and we omit it.  So assume $k=2$. Then the following identities hold in $G^\lrcorner _{p,3}$
    \begingroup
    \allowdisplaybreaks
    {\small  
    \begin{align*}
        \Omega_2 (2\underline{n}) := [x_1 ^{2 n_1}, x_2 ^{2 n_2}] &\stackrel{\ref{eq:1-cutting}}{=}_F [x_1 ^{2n_1}, x_2 ^{n_2}] \cdot \left[ x_1 ^{2n_1}, x_2 ^{n_2} \right]^{x_2 ^{n_2}} \\  
        &\stackrel{\ref{eq:2-cutting}}{=}_\calp \left[ x_1 ^{2n_1}, x_2 ^{n_2} \right]^2 \cdot \prod\limits_{i=1}^{D} \Omega_3 ^{p-2}(\underline{\eta}_i)^{\pm1} \\
        &\stackrel{\ref{eq:3-cutting}}{=}_F \left( \left[ x_1 ^{n_1}, x_2 ^{n_2} \right]^{x_1 ^{n_1}} \cdot [x_1 ^{n_1}, x_2 ^{n_2}] \right)^2 \cdot \prod\limits_{i=1}^{D} \Omega_3 ^{p-2}(\underline{\eta}_i)^{\pm1} \\
        &\stackrel{\ref{eq:4-cutting}}{=}_F \Big( \Omega_2 (\underline{n}) \cdot [\Omega_2 (\underline{n}), x_1 ^{n_1}] \cdot [x_1 ^{n_1}, x_2 ^{n_2}] \Big)^2 \cdot \prod\limits_{i=1}^{D} \Omega_3 ^{p-2}(\underline{\eta}_i)^{\pm1} \\
        &\stackrel{\ref{eq:5-cutting}}{=}_\calp \Omega_2 (\underline{n})^4 \cdot \Omega_3 (n_1, n_1, n_2)^{-2} \cdot \prod\limits_{i=1}^{D} \Omega_3 ^{p-2}(\underline{\eta}_i)^{\pm1} \\
        &\stackrel{\ref{eq:6-cutting}}{=}_\calp \Omega_2 (\underline{n})^4 \cdot \Omega_3 (n_1, n_1, n_2)^{-2} \cdot \prod\limits_{i=1}^{D} \Omega_{p-1} (\underline{\tilde{n}})^{\pm1}.
    \end{align*}
     }%
    \endgroup
    \begin{enumerate}[label={(\arabic*)}]
        \item is just the free identity $[u, v \cdot w] =_F [u,w] \cdot [u,v]^w$, so it has no area;\label{eq:1-cutting}
        \item is a consequence of the Main commuting Lemma \ref{lem:maincommuting} for $G^\lrcorner_{p,3}$, where $|\underline{\eta}_i| \lesssim_p |\underline{n}|$ and $D=O_{\alpha,p}(1)$, so it has area $\lesssim_p n ^{p-1}$;\label{eq:2-cutting}
        \item is an application of the free identity $[u \cdot v, w] =_F [u,w]^v \cdot [v,w]$, so it has no area;\label{eq:3-cutting}
        \item is a free identity so it has no area;\label{eq:4-cutting}
        \item is a consequence of the Main commuting Lemma \ref{lem:maincommuting}, hence it has area $\lesssim_p n^{p-1}$;\label{eq:5-cutting}
        \item follows from Lemma \ref{lem:trick} applied $O_{\alpha,p}(1)$ times, thus it has area $\lesssim_p n^2$.\label{eq:6-cutting}
    \end{enumerate}
    Setting $w_2 := \Omega_3 (n_1, n_1, n_2)^{-2} \cdot \prod_{i=1}^{D} \Omega_{p-1} (\underline{\tilde{n}})^{\pm1}$ completes the proof for $k=2$. 
\end{proof}

\subsubsection{Proof of the upper bounds for $\delta_{G_{p,3}}$ and $\delta_{G^\lrcorner_{p,3}}$}\label{remark-general-proof}

The proof is analogous to \cite[Theorem 6.1, p.756]{lipt}. We include it here for the readers convenience, since it is short and explains how the various technical results proved above fit together to yield the upper bound on the Dehn function.

One has to show that for all $\alpha \geqslant 1$ every null-homotopic word in $\mathcal{G}[\alpha]$ of length at most $n$ has area $\lesssim_p n^{p-1}$. Consider a null-homotopic word $w:= w(x_1, x_2, y_1, y_{p-1}) \in \mathcal{G}[\alpha]$ of length $|w| \leqslant n$. There are $u:= u(x_1, x_2)$ and $v:=v(y_1, y_{p-1})$ with $|u|,|v|\leqslant n$ such that the equality $w =_\calp u \cdot v$ holds in $G_{p,3}^\lrcorner$ with area $\lesssim_p n^2$. Since $\langle x_1, x_{p-1}, y_1, y_{p-1} \rangle$ generates a 5-Heisenberg subgroup of $G^\lrcorner_{p,3}$, the identity $v(y_1, y_{p-1}) =_\calp v(x_1, x_{p-1})$ holds in $G^\lrcorner_{p,3}$ with area $\lesssim_p n^2$.  Moreover, since the centre is quadratically distorted in the 5-Heisenberg subgroup \cite[Theorem 2.2]{Osindistort} and the word $v(x_1, x_{p-1})$ represents a central element, there exists $b \in \R$ with $|b| \lesssim_p n^2$ such that the identity $v(x_1, x_{p-1}) =_\calp z^b$ holds in $G^\lrcorner_{p,3}$. Similarly, there exists $\underline{\hat{n}} \in \R^{p-2}$ with $|\underline{\hat{n}}| \lesssim_p n$ such that $z^b =_\calp \Omega^3_{p-2}(\underline{\hat{n}})$ holds in $G^\lrcorner_{p,3}$. Therefore, the identity $v(x_1, x_{p-1}) =_\calp \Omega_{p-1}^3(\underline{\hat{n}})$ holds in  the subgroup $L_{p-1}\times_Z L_3$ of $G^\lrcorner_{p,3}$ and thus with area $\lesssim_p n^{p-2}$ by the induction hypothesis for $p-1$. Using that the relation $[x_1,x_2^a]=x_3^a$ has area $\lesssim_p a^2$ by \Cref{cor:commutingxkxl}, we can then replace $\Omega_{p-2}^3(\underline{\hat{n}})$ by $\Omega_{p-1}(\underline{\tilde{n}})$ with $|\underline{\tilde{n}}|\lesssim_p n$ at cost $\lesssim_p n^2$. Furthermore, for $\alpha$ sufficiently large $\Omega_{p-1}(\underline{\tilde{n}}) \in \mathcal{F}[\alpha]$. Therefore, after possibly increasing $\alpha$, the word $u\cdot \Omega_{p-1}(\tilde{\underline{n}})$ is null-homotopic in $\mathcal{F}[\alpha]$ of length at most $n$, which allows us to apply the results of the previous sections: first, the Reduction Lemma \ref{lem:reductionlemma} implies that

{\small
\begin{equation}\label{eq:reductionw}
    w =_\calp u\cdot \Omega_{p-1}(\tilde{\underline{n}})=\calp \prod_{i=1}^L \Omega_{l_i}(\underline{\eta}_i)^{\pm1}
\end{equation} 
}%
holds in $G^\lrcorner_{p,3}$ with area $\lesssim_p n^{p-1}$, where $|\underline{\eta}_i| \lesssim_p n$ and $2 \leqslant l_i \leqslant p-1$. We can then apply the Main commuting Lemma \ref{lem:maincommuting} to rearrange the $\Omega$-words on the left hand side of \eqref{eq:reductionw}; note that all the $\Omega$ words represent elements in the derived subgroup of $G^\lrcorner_{p,3}$ so they commute with area $\lesssim_{\alpha,p} n^{p-1}$. We thus get
{\small
\begin{equation}\label{eq:cancellingw}
    w =_\calp \left(\prod_{i=1}^{M_k} \Omega_k (\underline{n}_{k,i})^{\pm 1} \right) \left(\prod_{i=1}^{M_{k+1}} \Omega_{k+1} (\underline{n}_{k+1,i})^{\pm 1}\right) \ldots \left(\prod_{i=1}^{M_{p-1}} \Omega_{p-1} (\underline{n}_{p-1,i})^{\pm 1}\right)
\end{equation}
}%
Since $w$ is a null-homotopic word, the word on the right hand side of \eqref{eq:cancellingw} is also null-homotopic, therefore by the Cancelling 2-Lemma \ref{lem:cancelling} we get that $w$ has area $\lesssim_p n^{p-1}$ in $G^\lrcorner_{p,3}$. This finishes the proof of the induction step \ref{item:IH-fromp-1top} for $p$. \hfill \qedsymbol{} \medskip

\subsection{From $G_{p,3}$ and $G_{p,3}^\lrcorner$ to the general case}\label{sec:fromp3topqandbeyond}
We can now obtain upper bounds on Dehn functions for a larger class of central products using the results obtained for $G_{p,3}$ and $G^\lrcorner_{p,3}$. In particular, we treat the upper bound corresponding to Theorem \ref{th:general-factor} and Theorem \ref{th:model-filiform}. 

\begin{proposition}\label{prop:upper-LpcentralH}
    Let $q$ be an integer such that $q \geqslant 3$. Let $L$ be a simply connected nilpotent group of class $\leqslant q-1$ with one-dimensional cyclic centre. Let $K$ be either $L_p$ for $p>q$ or $L^\lrcorner_p$ for $p > \max\{q, 4\}$. The Dehn function of $G := K \times_Z L$ satisfies $\delta_G (n) \preccurlyeq n^{p-1}$.
\end{proposition}
\begin{proof}
    We prove the statement for $K=L_p$ the argument for $K=L^\lrcorner_p$ is exactly the same.
    Let $\calp$ be the presentation for $G$ given by Proposition \ref{prop:compactpresentations} and Lemma \ref{lem:central-prod-is-cglc}. Let $w$ be a null-homotopic word of length at most $n$ in $G$. Then, the identity $w =_\calp u(x_1, x_2)v_L$ holds in $G$ with area $\lesssim_p n^2$ where $v_L$ is a word in the generating set of $L$ and $|u|,|v_L| \leqslant n$. 
    
    Since the centre $Z$ is distorted in $L$ with polynomial distortion of degree $q-1$ \cite[Theorem 2.1]{Osindistort} and $v_L$ represents a central element of $G$, there exists $b$ with $|b| \lesssim_p n^{q-1}$ such that $v_L =_\calp z^b$. Since $q-1 \leqslant p-2$ there thus exists $\underline{\tilde{n}}$ with $|\underline{\tilde{n}}| \lesssim_p n$ such that we can rewrite $v_L$ as $\Omega_{p-2}^3(\underline{\tilde{n}})$ in $L_{p-1} \times_Z L \hookrightarrow L_p \times_Z L$. Moreover, since $L_{p-1} \times_Z L$ has nilpotency class $p-2$, it follows from \cite[Theorem A]{GerstenRileyHolt} (see \Cref{sec:GHR} for a generalisation to the case when $L$ has no lattice), that the identity $v_L =_\calp \Omega_{p-2}^3(\underline{\tilde{n}})$ has area $\lesssim_p n^{p-1}$. It follows from \cref{lem:first-lemma-simple-k-fold} that there exist $y_1, \ldots, y_{q-1}$ in $L$ such that $y_1$ and $[y_2, \ldots, y_{q-1}]$, generate a subgroup of $L$ isomorphic to the 3-Heisenberg group, which together with $L_p$ generates a subgroup of $L_p\times_Z L$ isomorphic to $L_p\times_Z L_3$. Therefore, $w=_\calp u(x_1, x_2) \cdot  \Omega_{p-2}^3 (\tilde{\underline{n}})$ is a null-homotopic word in $L_p \times_Z L_3$ of length $\lesssim_p n$ and  Proposition \ref{prop:upperbound-dehn-Gp3} (respectively \cref{prop:upperbound-dehn-Gcornerp3}) implies that it has area $\lesssim_p n^{p-1}$.
\end{proof}

An immediate consequence of Proposition \ref{prop:upper-LpcentralH} is the following result that establishes the upper bound of Theorem \ref{th:model-filiform} for $k > \ell$.

\begin{corollary}\label{cor:upper-model-filiform}
    Let $L$ be either $L_q$  with $q \geqslant 3$ or $L^\lrcorner_q$ with $q \geqslant 5$, and let $K$ be either $L_p$ with $p>q$ or $L^\lrcorner_p$ with $p > \max\{q, 4\}$. Then the group $G :=K \times_Z L$ has Dehn function $\delta_G (n) \preccurlyeq n^{p-1}$.
\end{corollary}

We close this section with the upper-bound of Theorem \ref{th:model-filiform} for $k=\ell$.

\begin{proposition}\label{prop:upper-LpcornerLpcorner}
    Let each of $K$ and $L$ be either $L_p$ with $p \geqslant 3$ or $L^\lrcorner_p$ with $p \geqslant 5$. Then the Dehn function of $G:= K \times_Z L$ satisfies satisfies $\delta_G (n) \preccurlyeq n^{p-1}$.
\end{proposition}
\begin{proof}
    Let $\calp$ be the presentation for $G$ given by Lemma \ref{lem:central-prod-is-cglc} and Proposition \ref{prop:compactpresentations}.
    The proof is done by ascending induction on $p$. \footnote{Note that the proofs for $L_p^{\lrcorner}\times_Z L_p^{\lrcorner}$ and $L^\lrcorner_p \times_Z L_p$ only use the induction hypothesis for $G_{q,q}$ with $q<p$.} For the base case we have that by \cite{Allcock} and \cite{OlsSapCombDehn} the Dehn function of $L_3 \times_Z L_3$ is quadratic and  by \cite{lipt} the Dehn function of $L_p \times_Z L_p$ is cubic if $p=4$ and quartic if $p=5$. Assume now that the statements hold for $p-1 \geqslant 4$.  
    
    Let $n \geqslant 1$ be an integer. Let $w := w(x_1, x_2, y_1, y_2)$ be a null-homotopic word in $G$ of length at most $n$. Using the fact that the $y_i$'s commute with the $x_i$'s we can rewrite $w$ as $w_1(x_1, x_2)w_2(y_1, y_2)$ in $G$ with area $\lesssim_p n^2$. Since $w_1 w_2$ is also a null-homotopic word in $G$ and $\langle x_1, x_2\rangle \cap \langle y_1, y_2 \rangle = \langle z \rangle$, where $x_p = z = y_p$ is the generator of the centre of $G$, we get that $w_1$ and $w_2$ represent elements in the centre. Thus, there exists $d \in \R$ such that the identities $w_1 (x_1,x_2)=_\calp z^d$ and $w_2(y_1, y_2) =_\calp z^{-d}$ hold in $G$. 
    
    Since the distortion of $\langle z \rangle$ in $G$ is $\simeq n^{p-1}$, we get $|d| \lesssim_p n^{p-1}$. It follows from Proposition \ref{prop:compactpresentations} that the identities 
    {\small
    $$
    z^d =_\calp \Omega_{p-1}(\hat{n}) \quad \text{and} \quad z^{d} =_\calp \widetilde{\Omega}_{p-1}(\hat{n})
    $$
    }
    hold in $G$ for some $\hat{n}\in \R^{p-1}$ with $|\hat{n}| \lesssim_p n$. Therefore, the words $w_1(x_1, x_2) \cdot \left(\Omega_{p-1}(\hat{n})\right)^{-1}$ and $ \widetilde{\Omega}_{p-1}(\hat{n})\cdot w_2(y_1, y_2)$ are null-homotopic words in $G^\lrcorner_{p,3} \hookrightarrow G$ with area $\lesssim_p n^{p-1}$. 
    
    By Corollary \ref{cor:omegapminus1} we can rewrite the $\Omega_{p-1}$- and $\widetilde{\Omega}_{p-1}$-words as products of  $\Omega_{p-2}^3$- and $\widetilde{\Omega}_{p-2}^3$-words in $G_{p-1,p-1} \hookrightarrow G$. It then follows by applying the induction hypothesis on $p$ $\lesssim_p n$ times to pairs of $\Omega_{p-2}^3$- and $\widetilde{\Omega}_{p-2}^3$-words, that the identity $\Omega_{p-1}(\hat{n}) =_\calp \widetilde{\Omega}_{p-1}(\hat{n})$ holds in $G$ with area $\lesssim_p n^{p-1}$. Overall, the identities 
    {\small
    \begin{align*}
        w &=_\calp w_1(x_1, x_2)w_2(y_1, z_2) \\
        &=_F w_1(x_1, x_2) \cdot \Omega_{p-1}(\hat{n})^{-1} \widetilde{\Omega}_{p-1}(\hat{n})\cdot w_2(y_1, y_2) \\
        &=_\calp 1
    \end{align*}
    }%
    hold in $G$ with area $\lesssim_p n^{p-1}$.
\end{proof}

\subsection{Proof of Theorem~\ref{th:sbe}}

\begin{remark}
    The following proof assumes knowledge of \cite[Section 9]{lipt}.
    The reader unfamiliar with Carnot groups arising as asymptotic cones of nilpotent groups may first consult Appendix~\ref{sec:QI} and especially Proposition~\ref{prop:preparing-cor-QI} before reading \cite[Section 9]{lipt} and then come back here.
\end{remark}

Let $p \geqslant 3$.
As in \cite{lipt}, it readily follows from our proof that $L_p \times_Z L_3$ admits $(n^{p-1}, n)$ as a filling pair\footnote{See \cite[\S3.3]{lipt} for the definition of a filling pair we use here.}. For this one checks that all prefix words of the transformations we apply represent group elements of distance $\lesssim n$ to the identity element in the Cayley graph for our chosen generating set. Thus, one argues as in \cite[Section 9]{lipt} that $L_p \times_Z L_ 3$ and its associated Carnot graded group are not $O(r^e)$-equivalent for $e \in [0, 1/(2p))$. On the other hand Cornulier shows in \cite[Proposition 6.13]{cornulier2017sublinear} that the two groups are $O(r^{2/(p-1)})$-bilipschitz equivalent. This shows that Cornulier's bound is optimal in the limit as the nilpotency class $p-1$ tends to $+\infty$.

\section{Proofs of Theorems~\ref{th:general-factor}, \ref{th:model-filiform}, and \ref{th:lowdim}}\label{sec:proofsAB}

In the first part of this section we recall the statements of Theorem \ref{th:general-factor} and Theorem \ref{th:model-filiform} whose proofs consist of putting together the upper and lower bounds obtained in the previous sections. The second part treats the proof of Theorem \ref{th:lowdim} whose proof is done by treating each of the groups individually. The treatment of these cases can be divided into two categories: the \emph{elementary ones} and the \emph{non-elementary ones}. The latter are proved using Theorem \ref{th:general-factor} and Theorem \ref{th:model-filiform}, while the former are done with a hands-on treatment. 

\begingroup
\renewcommand{\thetheorem}{\ref{th:general-factor}}

\begin{theorem}
    Let $k> \ell\geqslant 2$ be integers.
    Let $K$ be either the group $L_{k+1}$ or $L_{k+1}^\lrcorner$. Let $L$ be a simply connected nilpotent Lie group with one-dimensional centre of nilpotency class $\ell$.
    Let $G = K\times_ZL$.
    Then $\delta_G(n) \asymp n^k$.
\end{theorem}
\begin{proof}
    The upper bound follows from Proposition \ref{prop:upper-LpcentralH}. The lower bound follows from Proposition \ref{prop:lower-bound-general}.
\end{proof}

\renewcommand{\thetheorem}{\ref{th:model-filiform}}
\begin{theorem}
    Let $k\geqslant \ell \geqslant 2$ be integers.
    Let $K$ be either the group $L_{k+1}$ or $L_{k+1}^\lrcorner$ . Let $L$ be either the group $L_{\ell+1}$ or $L_{\ell+1}^\lrcorner$. Let $G = K\times_ZL$.
    Then $\delta_G(n) \asymp n^k$.
\end{theorem}
\begin{proof}
    The upper bound follows from Corollary \ref{cor:upper-model-filiform} and Proposition \ref{prop:upper-LpcornerLpcorner}. The lower bound follows from Proposition \ref{prop:lower-bound-general}.
\end{proof}

\renewcommand{\thetheorem}{\ref{th:lowdim}}

\begin{theorem}
    Let $k\geqslant \ell\geqslant 2$ be integers and let $K$ and $L$ be simply connected nilpotent Lie groups with one-dimensional centres, and class $k$ and $\ell$ respectively. Assume that
    $
    \max\{\dim K, \dim L\} \leqslant 5 
    $.
    Let $G = K \times_ZL$.
    Then $\delta_G(n) \asymp n^k$.
\end{theorem}
\endgroup
\setcounter{theorem}{0}

\subsection{Proof of Theorem \ref{th:lowdim}} 

There are twenty-one central products $K \times_Z L$ as in the assumption of Theorem~\ref{th:lowdim}, namely where $K$ and $L$ have one-dimensional centre and dimension at most $5$.
 In Table~\ref{tab:groupsTheoremlowdim} we list them, using de Graaf's notation.
 We treat them separately in \S\ref{subsubsec:high-heisenberg-factor} the group with a factor isomorphic to the 5-Heisenberg group. For the readers convenience we mention the group corresponding to the given central product on Magnin's list \cite{Magnin} when it has dimension $6$ or $7$. We order the groups in the lexicographic order of $(k,\ell,d)$ where $k$ and $\ell$ are the nilpotency classes of $K$ and $L$, and $d = \dim G$.
 For readability, let us recall from \S\ref{subsubsec:grp-low-dim-prelim} that
 $$
 L_{3,2}= L_3, \quad L_{4,3} = L_4, \quad L_{5,4} = H_5, \quad L_{5,6}=L_5^\lrcorner, \quad \text{and} \quad L_{5,7} = L_5.
 $$
 
\begin{table}[t]
    \centering
    \begin{tabular}{c|c|c|c|c|c}
    Group & $(k,\ell, d)$ & $\delta(n)$ & reference & Our name & Other names \\
    \hline 
        $L_{3,2} \times_Z L_{3,2}$ & $(2,2,5)$ & $n^2$ & \cite{Allcock}       & $L_3 \times_Z L_3 = H_5$ & $L_{5,4}$ \cite{deGraafclass} \\
        $L_{5,4} \times_Z L_{3,2}$ & $(2,2,7)$ & $n^2$ & \cite{Allcock}       &   $H_5 \times_Z L_3 = H_7$        &  \\
        $L_{5,4} \times_Z L_{5,4}$ & $(2,2,9)$ & $n^2$ & \cite{Allcock}       &    $H_5 \times_Z H_5 = H_9$       &  \\
        $L_{4,3} \times_Z L_{3,2}$ & $(3,2,6)$ & $n^3$ & \cite{lipt}          & $L_4 \times_Z L_3$ & $\mathcal G_{6,2}$ \cite{Magnin} \\
        $L_{5,5} \times_Z L_{3,2}$ & $(3,2,7)$ & $n^3$ & \S\ref{subsec:G7318} & & $\mathcal{G}_{7,3.18}$ \cite{Magnin} \\
        $L_{4,3} \times_Z L_{5,4}$ & $(3,2,8)$ & $n^3$ & Theorem~\ref{th:general-factor} & $L_4 \times_Z H_5$ & \\
        $L_{5,5} \times_Z L_{5,4}$ & $(3,2,9)$ & $n^3$ & \S\ref{subsubsec:high-heisenberg-factor} & \\
        $L_{4,3} \times_Z L_{4,3}$ & $(3,3,7)$ & $n^3$ & \cite{lipt} & $L_4 \times_Z L_4$ & $\mathcal G_{7,3.16}$ \cite{Magnin} \\
        $L_{5,5} \times_Z L_{4,3}$ & $(3,3,8)$ & $n^3$ & \S\ref{subsec:L55L43} & \\
        $L_{5,5} \times_Z L_{5,5}$ & $(3,3,9)$ & $n^3$ & \S\ref{L55L55} & \\
        $L_{5,6} \times_Z L_{5,4}$ & $(4,2,9)$ & $n^4$ & Theorem~\ref{th:general-factor} & $L^\lrcorner_5 \times_Z H_5$ \\
        $L_{5,7} \times_Z L_{5,4}$ & $(4,2,9)$ & $n^4$ & Theorem~\ref{th:general-factor} & $L_5 \times_Z H_5$\\
        $L_{5,7} \times_Z L_{3,2}$ & $(4,3,7)$ & $n^4$ & Theorem~\ref{th:model-filiform} & $L_5 \times_Z L_3$ & $\mathcal G_{7,3.17}$ \cite{Magnin} \\
        $L_{5,6} \times_Z L_{3,2}$ & $(4,3,7)$ & $n^4$ & Theorem~\ref{th:model-filiform} & $L_5^\lrcorner \times_Z L_3$ & $\mathcal G_{7,2.30}$ \cite{Magnin} \\
        $L_{5,7} \times_Z L_{4,3}$ & $(4,3,8)$ & $n^4$ & \cite{lipt} & $L_5 \times_Z L_4$ & \\
        $L_{5,6} \times_Z L_{4,3}$ & $(4,3,8)$ & $n^4$ & Theorem~\ref{th:model-filiform} & $L_5^\lrcorner \times_Z L_4$ & \\
        $L_{5,7} \times_Z L_{5,5}$ & $(4,3,9)$ & $n^4$ & Theorem \ref{th:general-factor} & \\
        $L_{5,6} \times_Z L_{5,5}$ & $(4,3,9)$ & $n^4$ & Theorem \ref{th:general-factor} & \\
        $L_{5,7} \times_Z L_{5,7}$ & $(4,4,9)$ & $n^4$ & \cite{lipt} & $L_5 \times_Z L_5$ & \\
        $L_{5,7} \times_Z L_{5,6}$ & $(4,4,9)$ & $n^4$  & Theorem~\ref{th:model-filiform} & $L_5^\lrcorner \times_Z L_5$ \\
        $L_{5,6} \times_Z L_{5,6}$ & $(4,4,9)$ & $n^4$  & Theorem~\ref{th:model-filiform} & $L_5^\lrcorner \times_Z L_5^\lrcorner$ \\
    \end{tabular}\vspace{.3cm}
    
    \caption{Central products of groups with one-dimensional centre and dimension at most $5$. In the fifth column we provide the name according to the notation in the present paper, when applicable. The sixth column gives names from de Graaf's and Magnin's classification (note that Magnin considers the Lie algebras over the complex numbers).}
    \label{tab:groupsTheoremlowdim}
\end{table}

\subsubsection{Non-elementary cases of Theorem~\ref{th:lowdim}} 

\label{subsubsec:non-elem}
We start with the non-elementary cases, which are treated using Theorems \ref{th:general-factor} and Theorem \ref{th:model-filiform}.

\begin{proposition}
    The following groups all have quartic Dehn functions:
    \begin{enumerate}
        \item $L_{5,6} \times_Z L_{3,2}$
        \item $L_{5,6} \times_Z L_{4,3}$
        \item $L_{5,6} \times_Z L_{5,6}$
        \item $L_{5,6} \times_Z L_{5,7}$
        \item $L_{5,i} \times_Z L_{5,4}$ for $i =6,7$.
        \item $L_{5,i} \times_Z L_{5,5}$ for $i =6,7$.
    \end{enumerate}
\end{proposition}
\begin{proof}
\begin{enumerate}
    \item Since $L_{5,6}$ is the group $L_5^{\lrcorner}$, this is a consequence of Theorem~\ref{th:model-filiform} with $k=4$ and $\ell = 2$.
    \item This is a consequence of Theorem~\ref{th:model-filiform} with $k=4$ and $\ell = 3$.
    \item  This is a consequence of Theorem~\ref{th:model-filiform} with $k=\ell=4$.
    \item Since $L_{5,6} \times_Z L_{5,7}$ is the group $G_{5,5}^{\lrcorner}$, this is a consequence of Theorem~\ref{th:model-filiform} with $k=\ell = 4$.
    \item 
    This is a special case of Theorem~\ref{th:general-factor} with $k=4$ and $\ell = 2$. 
    \item 
    This is a special case of Theorem~\ref{th:general-factor} with $k=4$ and $\ell = 3$. 
    \qedhere
\end{enumerate}
\end{proof}

\begin{proposition}
\label{prop:dehn-L567L54}
   The group $G = L_{4,3} \times_Z L_{5,4}$ has a cubic Dehn function.
\end{proposition}
\begin{proof}
    This is a special case of Theorem~\ref{th:general-factor} with $k=3$ and $\ell = 2$.
\end{proof}

This finishes the treatment of the non-elementary cases. 

\subsubsection{Elementary cases of Theorem~\ref{th:lowdim}} \label{subsec:lowdim-case-by-case}
In this paragraph, we take an opposite approach to the rest of the paper in that we compute the Dehn function of finitely presented groups $\Gamma_{d,i}$ arising as cocompact lattices in $L_{d,i}$. No knowledge of Lie groups and Lie algebras is needed here for the computations of the upper bound. For the lower bound we make use of some computations of cocycles in Lie algebra cohomology, which could be avoided.

\subsubsection*{The group $\Gamma_{5,5} \times_Z \Gamma_{3,2}$.}\label{subsec:G7318}

Let us recall that by Proposition~\ref{prop:pres-lattice-L55} the group $L_{5,5}$ contains a lattice $\Gamma_{5,5}$ that is generated by $x_1,\ldots, x_5$, with the nontrivial commutators
\begin{equation}
\label{eq:pres-L55L32}
    [x_1,x_2] = x_3, \ [x_1,x_3] = x_4, \  [x_2,x_5] = x_4.
\end{equation}
We adapt the argument in \S4 of \cite{lipt}, to prove the following: 

\begin{proposition}
\label{prop:dehn-L55H3}
The Dehn function of $\Gamma = \Gamma_{5,5} \times_Z \Gamma_{3,2}$ is cubic.
\end{proposition}

\begin{proof}
    Let us start with the proof that the Dehn function of $\Gamma$ is at most cubic.
The presentation of $\Gamma$ that we shall work with is the following:
the generating set is 
\[ \{x_1,x_2,x_3,x_4,x_5,y_1,y_2\} \] 
and the (non-trivial) relations are
\[ [x_1,x_2] = x_3, \ [x_1,x_3] = x_4, \ [x_2,x_5] = x_4, \ [y_1,y_3] = x_4; \]
it is also convenient to add two letters corresponding to artificial generators $y_2$ and $y_5$ and the relations $y_2=y_1$ and $y_5=y_3$, because we use the fact that the set $\{x_2,x_5, y_2, y_5\}$ generates an integral 5-Heisenberg group which is the central product of $\left\langle x_2,x_5 \right\rangle$ and of $\left\langle y_2,y_5 \right\rangle$. This can be summarized by the following diagram.

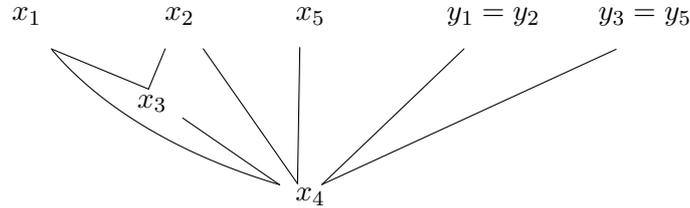
\begin{center}
    \begin{tikzpicture}[line cap=round,line join=round,>=triangle 45,x=1cm,y=0.6cm]
\clip(-5,-2.8) rectangle (7,3.2);
\draw (-3.65,2.14) node[anchor=north west] {$x_1 $};
\draw (-1.64,2.14) node[anchor=north west] {$ x_2 $};
\draw (0.08,2.14) node[anchor=north west] {$ x_5 $};
\draw (2.07,2.14) node[anchor=north west] {$ y_1=y_2 $};
\draw (4.06,2.14) node[anchor=north west] {$y_3=y_5$};
\draw (-2,0.21) node[anchor=north west] {$ x_3 $};
\draw (0.08,-1.85) node[anchor=north west] {$ x_4 $};
\draw (-3,1)-- (-1.72,0.11);
\draw (-1.5,1)-- (-1.72,0.11);
\draw (-1,1)-- (0.24,-1.98);
\draw (0.27,1.03)-- (0.24,-1.98);
\draw (-1.27,-0.53)-- (0,-2);
\draw (2.43,1)-- (0.56,-2);
\draw (4.43,0.99)-- (0.56,-2);
\draw [shift={(3,4)}] plot[domain=3.61:4.25,variable=\t]({1*6.71*cos(\t r)+0*6.71*sin(\t r)},{0*6.71*cos(\t r)+1*6.71*sin(\t r)});
\end{tikzpicture}
\end{center}

By Lemma~\ref{lem:back-first-factor}, to prove Proposition \ref{prop:dehn-L55H3} it is sufficient to show that null-homotopic words $w(x_1,x_2,x_5)$ of length at most $n$ have area at most $Cn^3$ in the given presentation.

We define $\Omega(m,n):=[x_1^{m},x_3^n]$ with $m,n \in \mathbf Z$. We describe an algorithm to reduce $w$ to a trivial word, which is a variant of the one given in \cite[\S4]{lipt}. It consists of first rewriting $w(x_1,x_2,x_5)$ as a product of words $\Omega(m_i,n_i)$ at a cubic cost. Then the product of the $\Omega(m_i,n_i)$ is reduced to a single word of the form $\Omega(m,n)$, which is trivial, since $w$ is null-homotopic. 

We give names to the steps of our algorithm and compute the total area of our transformations at the end. First note that if there is no occurrence of $x_1^{\pm 1}$ in $w$ we are done, because $\{x_2,x_5\}$ generates a Heisenberg subgroup, which has cubic Dehn function. We thus assume that there is at least one occurrence of $x_1^{\pm 1}$ in $w$.

\begin{description}
\item[(M\textsubscript{1})]
Start from $w(x_1,x_2,x_5)$, look for the leftmost occurence of $x_1^{\pm 1}$ that has other preceding letters and decompose $w$ as $w=x_1^{k_{1,0}}u_1(x_2,x_5)x_1^{\pm 1}v_1(x_1,x_2,x_5)$ for some $k_{1,0} \in \mathbf Z$. Move the $x_1^{\pm 1}$ to the left using the identities
\[ \begin{array}{ccc}
      x_2x_1=x_1x_3^{-1}x_2,   & \vspace{.2cm} & x_2^{-1}x_1=x_1x_3x_2^{-1},  \\
      x_2x_1^{-1}=x_1^{-1}x_3x_4x_2,   & \vspace{.2cm} & x_2^{-1}x_1^{-1}=x_1^{-1}x_3^{-1}x_4^{-1}x_2^{-1}.
    \end{array} \]
    (Note also that $x_1$ commutes with $x_5$.)
This produces occurrences of $x_3^{\pm 1}$ and $x_4^{\pm 1}$ directly right of $x_1^{\pm 1}$ every time it crosses an occurrence of $x_2^{\pm 1}$. We carry these occurrences to the left with us. At the end of this step we are left with a word of the form
\[ 
x_1^{k_{1,1}}x_3^{k_{3,1}}x_4^{m_{1,2}} u_1(x_2,x_5)v_1(x_1,x_2,x_5), \]
for some $k_{3,1}, m_{1,2} \in \mathbf Z$, where $|k_{3,1}|$ is bounded above by the total number of occurrences of $x_2^{\pm 1}$ in $w$ and $m_{1,2}$ is either zero or equal to $k_{3,1}$, depending on the sign of the exponent of the $x_1^{\pm 1}$ we moved. The total cost of this step is $\lesssim n^2$.
\item[(M\textsubscript{3})]
Move all the $x_3^{\pm 1}$'s produced in Step {\bf (M\textsubscript{1})} to the very left. At the end of this step, we are left with a word of the form
\[ x_3^{k_{3,1}} x_1^{k_{1,1}} \Omega(\widetilde{m}_1,\widetilde{n}_1) x_4^{m_{1,2}}u_1(x_2,x_5) v_1(x_1,x_2,x_5). \]
for some $\widetilde m_1, \widetilde n_1 \in \mathbf Z$. Note that $|\widetilde{m}_1|+|\widetilde{n}_1|$ is bounded by $n$, as $\widetilde{m}_1=k_{1,1}$ and $\widetilde{n}_1=k_{3,1}$ and that this transformation happens in the free group on the generating set, meaning that it does not have any cost.
\item[(M\textsubscript{$\Omega$})] Move the $x_4^{m_{1,2}}$ and the $\Omega(\widetilde{m}_1,\widetilde{n}_1)$ created in Step {\bf (M\textsubscript{3})} to the very right using the 5-Heisenberg group generated by the set $\left\{x_1,x_3,y_1,y_3\right\}$. After that the word has the form 
\[ x_3^{k_{3,1}} x_1^{k_{1,1}}  u_1(x_2,x_5) v_1(x_1,x_2,x_5) \Omega(\widetilde{m}_1,\widetilde{n}_1) x_4^{m_{1,2}}. \]
Then use the 5-Heisenberg group once more to obtain an identity of the form $\Omega(\widetilde{m}_1,\widetilde{n}_1)x_4^{m_{1,2}}=\Omega(m_1,n_1)$ with $|m_1| + |n_1|\lesssim n$.
The cost of this step is quadratic in $|m_1|+|n_1|\leqslant n$.
\end{description}

We inductively repeat the moves {\bf (M\textsubscript{1})}, {\bf (M\textsubscript{3})} and {\bf (M\textsubscript{$\Omega$})}, $i$ times\footnote{Strictly speaking we apply the inductive step to the word $u_i (x_2, x_5) v_i (x_1, x_2, x_5)$ in \eqref{eq:L55-H3-step-i}.} to obtain a word of the form
\begin{equation}\label{eq:L55-H3-step-i}
 x_3^{k_{3,i}} x_1^{k_{1,i}} u_i (x_2, x_5) v_i (x_1, x_2, x_5) \Omega(m_i, n_i) \cdots \Omega(m_1, n_1),
\end{equation}
and it is always true that $\vert u_i \vert + 1 + \vert v_i \vert \leqslant n$, since we did not create new occurrences of $x_1^{\pm 1}$, $x_2^{\pm 1}$ and $x_5^{\pm 1}$ in the process. Moreover, we reduce the total number of $x_1^{\pm 1}$'s by at least one in each step.

The algorithm ends when there is not a single occurrence of $x_1^{\pm 1}$ left in $v_i$. At this point, say $i = I$, we are left with a word of the form 
\[x_3^{k_1,I}x_1^{k_{1,I}} U(x_2,x_5) \Omega(m_I, n_I) \cdots \Omega(m_1, n_1).\]
Since $\Omega(m_I, n_I) \cdots \Omega(m_1, n_1)$ represents a central element in $G$, the same must be true for $x_3^{k_1,I} x_1^{k_2,I} U(x_2,x_5)$.
In particular, its $x_1$-exponent sum $k_{1,I}$ must be zero, as its image under the abelianization map is trivial.

The remaining word is then 
\begin{equation}
\label{eq:intermezzo-L55H3-1}
    x_3^{k_{3,I}} U(x_2,x_5) \Omega(m_I, n_I) \cdots \Omega(m_1, n_1).
\end{equation}

Using that the 5-Heisenberg group has quadratic Dehn function \cite{OlsSapCombDehn}, we can first rewrite \eqref{eq:intermezzo-L55H3-1} as 
\[x_3^{k_{3,I}} U(y_2,y_5) \Omega(m_I, n_I) \cdots \Omega(m_1, n_1),\] 
and then as
\begin{equation}
\label{eq:intermezzo-L55H3-2}
    x_3^{k_{3,I}} U(x_1,x_3) \Omega(m_I, n_I) \cdots \Omega(m_1, n_1)
\end{equation}
at quadratic cost.

We now observe that there is some $x_3^{k_{3,I+1}}$ with $|k_{3,I+1}|\leqslant n$ such that $x_3^{-k_{3,I+1}} U(x_1,x_3)$ is central. Thus, an identity of the form $x_3^{-k_{3,I+1}} U(x_1,x_3)=\Omega(m_{I+1},n_{I+1})$ holds in the corresponding 5-Heisenberg subgroup at quadratic cost. This yields a word of the form
\begin{equation*}
    x_3^{k_{3,I+1}} \Omega(m_{I+1},n_{I+1}) \cdots \Omega(m_1,n_1).
\end{equation*}
Since the product of the $\Omega(m_i,n_i)$ is central, $k_{3,I+1} = 0$.

Let us now bound the costs.
Apart from the steps between \eqref{eq:intermezzo-L55H3-1} and \eqref{eq:intermezzo-L55H3-2}, which have quadratic cost, the procedure described above has $I$ steps, each of which has quadratic cost (since this is true for each of steps {\bf (M\textsubscript{1})}, {\bf (M\textsubscript{3})} and {\bf (M\textsubscript{$\Omega$})} individually). Since $I$ is the number of occurrences of $x_1^{\pm 1}$ in the word $w$ at the beginning 
it follows that $I \leqslant \vert w \vert \leqslant n$. In particular, at total cost $\lesssim I\cdot n^2\lesssim n^3$ we transformed the initial word $w$ into a null-homotopic word of the form
\[\Omega(m_{I+1},n_{I+1}) \cdots \Omega(m_1,n_1) \]
with $I+1\leqslant {\rm const}\cdot n$ and $|m_i|+|n_j|\leqslant {\rm const}\cdot n$. By \cite[Lemma 4.5]{lipt} the area of such a word is cubic in $n$. This completes the proof of Proposition \ref{prop:dehn-L55H3}.

We now turn to the lower bound.
Letting $\xi_1, \xi_2, \xi_3, \xi_4, \eta_1, \eta_3$ be the basis dual to $X_1, X_2, X_3, X_4, Y_1, Y_3$ in the Lie algebra of $L_{5,5} \times_Z L_{4,3}$ we find that $\xi_2 \wedge \xi_3$ is a cocycle and not a coboundary, and so defines a central extension of the Lie algebra of $L_{5,5} \times_Z L_{3,2}$, with generators $\widetilde X_1, \ldots, \widetilde X_4, \widetilde Y_1, \widetilde Y_3, \widetilde{Z}$ and Lie brackets
\[ [\widetilde X_1,\widetilde X_2] = \widetilde X_3, \  [\widetilde X_1, \widetilde X_3] = \widetilde X_4, \ [\widetilde X_2, \widetilde X_5] = \widetilde X_4, \ [\widetilde Y_1, \widetilde Y_3] = \widetilde X_4, \ [\widetilde X_2, \widetilde X_3] = \widetilde{Z}. \]
Using the Baker--Campbell--Hausdorff formula exactly in the same way as in Proposition~\ref{prop:pres-lattice-L55}, it then follows that the group $\Gamma_{5,5} \times_Z \Gamma_{3,2}$ has a central extension $\widetilde \Gamma_1$ whose presentation is given by the set of generators $\widetilde x_1, \widetilde x_2, \widetilde x_3, \widetilde x_4, \widetilde y_1, \widetilde y_3, \widetilde z$ and the relations
    \[ [\widetilde x_1,\widetilde x_2] = \widetilde x_3, \ [\widetilde x_1, \widetilde x_3] = \widetilde x_4, \ [\widetilde x_2, \widetilde x_5] = \widetilde x_4, \  [\widetilde y_1, \widetilde y_3] = \widetilde x_4, \ [\widetilde x_2, \widetilde x_3] = \widetilde{z}. \]
    More precisely, $\widetilde \Gamma_1$ fits in the short exact sequence 
    \[ 1 \to \langle \widetilde z \rangle \to \widetilde \Gamma_1 \to \Gamma \to 1 \]
    where $\widetilde x_i \mapsto x_i$ and $\widetilde y_j \mapsto y_j$ for all $i,j$.
    The subgroup generated by $\widetilde{z}$ lies in the third term of the central series of $\widetilde \Gamma_1$, and thus, it is cubically distorted. By \cite[Proposition 7.2]{lipt} (using criterion (i) there), the Dehn function of $\Gamma$ is at least cubic.
\end{proof}

\subsubsection*{The group $\Gamma = \Gamma_{5,5} \times_Z \Gamma_{4,3}$.}\label{subsec:L55L43}
The generating set for the group $\Gamma_{5,5} \times_Z \Gamma_{4,3}$ is \[ \{x_1,x_2,x_3,x_4,x_5,y_1,y_2,y_3\} \] 
and the relations that we use are
\[ [x_1,x_2] = x_3, \ [x_1,x_3] = x_4, \ [x_2,x_5] = x_4, \  [y_1,y_2] = y_3, \ [y_1,y_3] = x_4. \]
These are summarized in the diagram below.
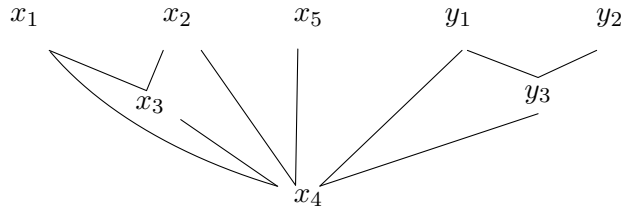
\begin{center}
    \begin{tikzpicture}[line cap=round,line join=round,>=triangle 45,x=1cm,y=0.6cm]
\clip(-5,-3.4) rectangle (7,2.2);
\draw (-3.65,2.14) node[anchor=north west] {$x_1 $};
\draw (-1.64,2.14) node[anchor=north west] {$ x_2 $};
\draw (0.08,2.14) node[anchor=north west] {$ x_5 $};
\draw (2.07,2.14) node[anchor=north west] {$ y_1 $};
\draw (4.06,2.14) node[anchor=north west] {$y_2$};
\draw (-2,0.21) node[anchor=north west] {$ x_3 $};
\draw (0.08,-1.85) node[anchor=north west] {$ x_4 $};
\draw (-3,1)-- (-1.72,0.11);
\draw (-1.5,1)-- (-1.72,0.11);
\draw (-1,1)-- (0.24,-1.98);
\draw (0.27,1.03)-- (0.24,-1.98);
\draw (-1.27,-0.53)-- (0,-2);
\draw (2.43,1)-- (0.56,-2);
\draw (3.43,-0.4)-- (0.56,-2);
\draw (3.43,-0.4) node[anchor=south] {$ y_3 $};
\draw (2.5,1) -- (3.43,0.4) -- (4.2,1);
\draw [shift={(3,4)}] plot[domain=3.61:4.25,variable=\t]({1*6.71*cos(\t r)+0*6.71*sin(\t r)},{0*6.71*cos(\t r)+1*6.71*sin(\t r)});
\end{tikzpicture}
\end{center}

\begin{proposition}
\label{prop:dehn-L55L43}
The Dehn function of $\Gamma = \Gamma_{5,5} \times_Z \Gamma_{4,3}$ is cubic.
\end{proposition}

\begin{proof}
Let us start with the upper bound.
The subgroup $K_0$ of $\Gamma_{5,5}$ generated by $x_1$ and $x_2$ is isomorphic to $L_{4,3}$ and contains the centre of $L_{5,5}$. Since $L_{4,3} \times_Z L_{4,3}$ has cubic Dehn function \cite{lipt}, Lemma~\ref{lem:back-first-factor} implies that it suffices to prove that any null-homotopic word $w(x_1,x_2,x_5)$ in $\Gamma$ of length at most $n$ has area at most $Cn^3$ in the given presentation.
This is an immediate consequence of Proposition \ref{prop:dehn-L55H3}, since the subgroup of $G$ generated by $\left\{x_1,x_2,x_5,y_1,y_3\right\}$ is isomorphic to $L_{5,5} \times_Z L_{3,2}$ and its presentation, given in \eqref{eq:pres-L55L32}, is a subpresentation of our presentation for $G$.

Now, let us prove that the Dehn function of $\Gamma$ is at least cubic.
The central extension $\widetilde \Gamma_1$ of the group $\Gamma_{5,5} \times_Z \Gamma_{3,2}$ considered in the proof of Proposition~\ref{prop:dehn-L55H3} has itself an extension given by adding the generator $\widetilde y_2$ and setting the additional relation $[\widetilde y_1, \widetilde y_2] = \widetilde y_3$. This is a central extension of $\Gamma_{5,5} \times_Z \Gamma_{4,3}$. The distortion of the subgroup generated by $\widetilde{z}$ in this central extension is again cubic. We deduce in the same way using \cite[Proposition 7.2]{lipt} that the Dehn function of $\Gamma_{5,5} \times_Z \Gamma_{4,3}$ is at least cubic.
\end{proof}

\subsubsection*{The group $\Gamma = \Gamma_{5,5} \times_Z \Gamma_{5,5}$}\label{L55L55}
\begin{proposition}
\label{prop:dehn-L55L55}
The Dehn function of $\Gamma = \Gamma_{5,5} \times_Z \Gamma_{5,5}$ is cubic.
\end{proposition}

\begin{remark}
A weaker estimate can be derived more directly. One can check that the real Malcev completions of $\Gamma_{5,5} \times_Z \Gamma_{5,5}$ and $\Gamma_{4,3} \times_Z \Gamma_{4,3} \times \mathbf Z^2$ have the same associated Carnot graded group.
Also, the Dehn function of $\Gamma_{4,3} \times_Z \Gamma_{4,3} \times \mathbf Z^2$ is $n^3$ by \cite{lipt}.
It then follows from \cite{DrutuRemplissage} that $\delta_\Gamma(n) \preccurlyeq n^{3+\epsilon}$ for every $\epsilon > 0$. 
\end{remark}

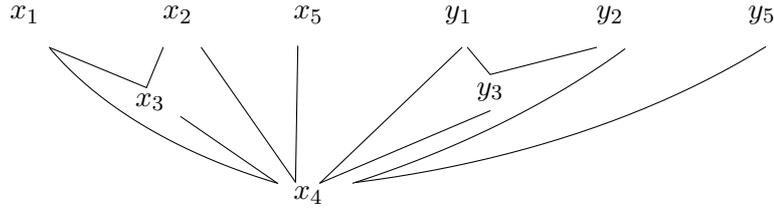
\begin{center}
    \begin{tikzpicture}[line cap=round,line join=round,>=triangle 45,x=1cm,y=0.6cm]
\clip(-5,-3.4) rectangle (7,2.2);
\draw (-3.65,2.14) node[anchor=north west] {$x_1 $};
\draw (-1.64,2.14) node[anchor=north west] {$ x_2 $};
\draw (0.08,2.14) node[anchor=north west] {$ x_5 $};
\draw (2.07,2.14) node[anchor=north west] {$ y_1 $};
\draw (4.06,2.14) node[anchor=north west] {$y_2$};
\draw (-2,0.21) node[anchor=north west] {$ x_3 $};
\draw (0.08,-1.85) node[anchor=north west] {$ x_4 $};
\draw (-3,1)-- (-1.72,0.11);
\draw (-1.5,1)-- (-1.72,0.11);
\draw (-1,1)-- (0.24,-1.98);
\draw (0.27,1.03)-- (0.24,-1.98);
\draw (-1.27,-0.53)-- (0,-2);
\draw (2.43,1)-- (0.56,-2);
\draw (2.8,-0.4)-- (0.56,-2);
\draw (2.8,-0.4) node[anchor=south] {$ y_3 $};
\draw (2.5,1) -- (2.8,0.4) -- (4.2,1);
\draw [shift={(3,4)}] plot[domain=3.61:4.25,variable=\t]({1*6.71*cos(\t r)+0*6.71*sin(\t r)},{0*6.71*cos(\t r)+1*6.71*sin(\t r)});
\draw [shift={(-4.74,8.55)}] plot[domain=5.21:5.6,variable=\t]({1*12.01*cos(\t r)},{1*12.01*sin(\t r)});
\draw [shift={(-1.12,8.23)}] plot[domain=4.92:5.52,variable=\t]({1*10.44*cos(\t r)},{1*10.44*sin(\t r)});
\draw (6.06,2.15) node[anchor=north west] {$y_5 $};
\end{tikzpicture}
\end{center}

\begin{proof}
Again, the diagram above is meant to recall the relations in $\Gamma$ and ease the reading of the proof.
    Let $w$ be a null-homotopic word of length at most $n$ over the generating set $\{ x_1, x_2,x_5,y_1,y_2,y_5 \}$.
    We can rewrite $w$ as $w_1(x_1,x_2,x_5)w_2(y_1,y_2,y_5)$ at a quadratic cost.
    Let's now consider the word $w_2$, and move all instances of $y_{5}^{\pm 1}$ to the right, starting with the leftmost.
    Each time a $y_{5}^{\pm 1}$ passes through a $y_2^{\pm 1}$ from left to right, a $x_4^{\mp 1}$ is produced.
    Since there are, in total, at most $n$ occurrences of $y_5^{\pm 1}$ and of $y_2^{\pm 1}$, at most $n^2$ occurrences of $x_4^{\pm 1}$ are produced in this process, and the cost of all $y_5^{\pm 1}$ crossing them from left to right is at most $n^3$.
    At the end of this process the word has form
    \begin{equation*}
        w_1(x_1,x_2,x_5)u_2(y_1,y_2,x_4)y_5^k,
    \end{equation*}
    for some $k \in \mathbf Z$, where $\vert k \vert \leqslant n$.
    Now, since $x_4$ is central in $\Gamma$, moving all instances of $x_4^{\pm 1}$ that are present in the subword $u_2$ to the right of $u_2$ costs, again, at most $n^3$. The resulting word is
    \begin{equation*}
        w_1(x_1,x_2,x_5)v_2(y_1,y_2)x_4^my_5^k,
    \end{equation*}
    for some $m \in \mathbf Z$, where $\vert m \vert \leqslant n^2$ and, again, $\vert k \vert \leqslant n$.
    Now consider the subgroup $H$ of $\Gamma$ generated by $x_1$ and $x_3$. It is isomorphic to a 3-Heisenberg group, therefore we can rewrite $x_4^m$ as $[x_1^{\sqrt{\pm m}}, x_3^{\pm \sqrt{\pm m}}]$ (where the $\pm$ should be set according to the sign of $m$) at cost $\leqslant C n^3$.   
    The resulting word 
    \begin{equation*}
        w' = w_1(x_1,x_2,x_5)v_2(y_1,y_2)[x_1^{\sqrt{\pm m}}, x_3^{\pm \sqrt{\pm m}}]y_5^k
    \end{equation*}
    now has length $\leqslant c \cdot n$. Since it is null-homotopic, we have that $k=0$.
    Therefore, the word
     \begin{equation*}
        w' = w_1(x_1,x_2,x_5)v_2(y_1,y_2)[x_1^{\sqrt{\pm m}}, x_3^{\pm \sqrt{\pm m}}]
    \end{equation*}
    is a null-homotopic word in $\Gamma_{5,5} \times_Z \Gamma_{4,3}$, and by Proposition~\ref{prop:dehn-L55L43} its area is at most $Cn^3$. This finishes the proof.
    
Let us now turn to the lower bound.
    Arguing as in the proof of Proposition~\ref{prop:dehn-L55H3}, we can construct a  central extension $\widetilde \Gamma_3$ of $\Gamma_{5,5} \times_Z \Gamma_{5,5}$ whose presentation is as follows: the generators are 
    \[ \widetilde x_1, \ldots, \widetilde x_5, \, \widetilde y_1, \,  \widetilde y_2, \, \widetilde y_3, \, \widetilde y_5, \, \widetilde{z} \]
    and the relations are
    \[ [\widetilde x_1,\widetilde x_2] = \widetilde x_3, \ [\widetilde x_1, \widetilde x_3] = \widetilde x_4, \ [\widetilde x_2, \widetilde x_5] = \widetilde x_4, \ [\widetilde y_1, \widetilde y_3] = \widetilde x_4, \ [\widetilde x_2, \widetilde x_3] = \widetilde{z}, \]
    all other commutators between generators being trivial.
    The subgroup generated by $\widetilde{z}$ in this central extension is cubically distorted, so we can conclude again using criterion (i) in \cite[Proposition 7.2]{lipt}.
\end{proof}

\subsubsection{Remaining case.}

\label{subsubsec:high-heisenberg-factor}

We exclude from the enumeration the higher Heisenberg groups, whose Dehn function was determined in \cite{Allcock}. 

 \begin{lemma}
 \label{lem:wlog-avoid-high-heisenberg}
     Let $K$ be a non-abelian simply connected nilpotent Lie group with a one-dimensional centre and let $H$ be a Heisenberg group of any dimension. Then
     \begin{equation*}
         \delta_{K \times_Z H}(n) \preccurlyeq \delta_{K \times_Z L_{3}}(n).
     \end{equation*}
 \end{lemma}

\begin{proof}
Any Heisenberg group is an iterated central product of copies of the group $L_3$. 
It is then sufficient to prove that, given a Heisenberg group $H$ and a group $K$ as in the assumptions, $\delta_{K \times_Z H \times_Z L_3} \preccurlyeq \delta_{K \times_Z H}$.
Starting from a given presentation of the factors, equip both groups with adapted presentations $\mathcal P$ and $\mathcal Q$ respectively using Lemma~\ref{lem:central-prod-is-cglc}.
Let $w$ be a null-homotopic word of length $n$ in the generating set of $G = K\times_Z H \times_Z L_3$. 
Then with quadratic  cost, one can rewrite 
\[ w =_{\mathcal P} u_K u_H u_L \]
where $u_K$, $u_H$ and $u_L$ are words in the generating sets of $K$, $H$ and the copy $L$ of $L_3$ in $G$ respectively.
The words $u_K$, $u_H$ and $u_L$ represent central elements in $K \cap H \cap L$.
Since $H \times_Z L$ is a higher Heisenberg group we have that $\delta_{H \times_Z L}(n) \preccurlyeq n^2$.
With quadratic cost, one can rewrite $u_Hu_L$ as a single commutator word $v$ of length bounded by a constant times $n$, in the generating set of $H$.
This proves that 
\begin{equation*}
    \operatorname{Area}_{\mathcal P}(w) \lesssim_{\mathcal Q} n^2 + \operatorname{Area}_{\mathcal Q}(u_K v) \lesssim_{\mathcal P, \mathcal Q} \delta_{K \times_Z H}(n).
\end{equation*}
where we used in the last inequality that the Dehn function of $K \times_Z H$ is at least quadratic.
\end{proof}

\begin{proposition}
\label{prop:dehn-L55L54}
    The group $\Gamma = \Gamma_{5,5} \times_Z \Gamma_{5,4}$ has cubic Dehn function.
\end{proposition}

\begin{proof}
    The upper bound follows from the combination of Lemma~\ref{lem:wlog-avoid-high-heisenberg} and Proposition~\ref{prop:dehn-L55H3}.
    For the lower bound, consider the central extension $\widetilde \Gamma_1$ defined in the proof of Proposition~\ref{prop:dehn-L55H3}. The group $\widetilde \Gamma_1 \times_Z L_3$ naturally is a central extension of $\Gamma$ with cubically distorted centre.
\end{proof}

\section{Uncountably many disjoint pairs of groups with the same associated Carnot graded group and different Dehn functions}
\label{sec:uncountable}

In this section we first explain the proof of Theorem \ref{thm:main-uncountable} and then explain the construction of the lowest-dimensional examples that satisfy the conclusion of Theorem \ref{th:general-factor} and have no lattices.

\subsection{An uncountable family of 4-nilpotent groups with same associated Carnot graded group and different Dehn functions}

In his textbook, Raghunathan shows that there are uncountably many isomorphism types of nilpotent Lie algebras of class $2$ without rational forms \cite[Remark 2.14]{RagDS}.
We develop Raghunathan's argument in larger generality, allowing more parameters, but still constructing nilpotent Lie algebras of class $2$. Then we show that for appropriate values of the parameters, the Lie algebras that are built in this way all have $3$-distorted central extensions. 

Let $(m,k)$ be a pair of positive integers. We will assume that 
{\small
\begin{equation}
    \label{eq:standing-assumption-pair-n-k}
    m^2 + k^2 < \binom{m}{2}k.
\end{equation}
}%
Note that this condition implies
{\small
\begin{equation}
\label{eq:k-not-too-large}
    k = \frac{k^2}{k} < \frac{1}{k} \left( k\binom{m}{2} -m^2 \right) < \binom{m}{2}.
\end{equation}
}%
Let $E = \mathbf R^m$ and $V = \mathbf R^k$.
Consider
\[ \mathcal L(m,k) = \{ \mu \in \bigwedge\nolimits^2 E^\ast \otimes V \colon \mu \text{ is surjective} \}. \]
Condition \eqref{eq:k-not-too-large} ensures that $\mathcal L (m,k)$ is open\footnote{We mean in the usual topology; it is also a Zariski open set, but we will not need this.} and non-empty in $\bigwedge^2 E^\ast \otimes V$. 
Let $G = \operatorname{GL}_m(\mathbf R) \times \operatorname{GL}_k(\mathbf R)$ act on $\bigwedge^2 E^\ast \oplus V$ in the tautological way.

Every $\mu \in \mathcal L(m,k)$ represents a $2$-nilpotent Lie algebra law on $E \oplus V$ for which $V$ is exactly the centre, the bracket being given by 
$[e+v,e'+v'] = \mu(e,e')$. This Lie algebra has a rational form if and only if $g.\mu(\bigwedge^2 (\mathbf Q^m)^\ast) = {\mathbf Q}^k$ for some $g \in G$. We denote by $\mathcal L_{\mathbf Q}(m,k)$ the subspace of $\mathcal L(m,k)$ formed by those $\mu$ having at least one rational form. 
In this way $\mathcal L(m,k)/G$ corresponds to the set of isomorphism types of nilpotent Lie algebras $\mathfrak n$ of class $2$ with $\dim \mathfrak n=m+k$ and $\dim C^2 \mathfrak n =k$, while $\mathcal L_{\mathbf Q}(m,k)/G$ corresponds to the subset of those ones having rational forms.

Note that the action of $G$ on $\mathcal L(m,k)$ is smooth; since $\dim G = m^2+k^2 < \dim \mathcal L(m,k)$, every $G$-orbit has empty interior.
In particular, $\mathcal L_{\mathbf Q}(m,k)$ is a countable union of subsets with empty interior.
By Baire's theorem it follows that $\mathcal L(m,k) \setminus \mathcal L_{\mathbf Q}(m,k)$ is non-empty; even better, it is partitioned into uncountably many $G$-orbits.

We now view $E^\ast \otimes V^\ast$ as a subspace of $\bigwedge^2(E\oplus V)^\ast$ by mapping the simple tensor $\alpha \otimes \beta$ to $\iota_{E^\ast} \alpha \wedge \iota_{V^\ast} \beta$ and extending linearly.
For every $\mu \in \mathcal L(m,k)$, denote by $d_\mu$ the exterior derivative of the Lie algebra cohomology associated to $\mu$ (see e.g. \cite{ChevalleyEilenberg} for the definition). In view of the previous identification, we see $d_\mu$ as a map
\[ d_\mu \colon E^\ast \otimes V^\ast \to \bigwedge\nolimits^3 E^\ast \]
and let $Z^{2,3}_\mu = \{ \omega \in E^\ast \otimes V^\ast \colon d_\mu \omega = 0 \}$.
The notation can be explained as follows: there is a grading on the space of cocycles on the Lie algebra represented by $\mu$. Under this grading (that we will take in the positive integers for convenience) $E^\ast$ has weight $1$, $V^\ast$ has weight $2$, and $Z_\mu^{2,3}$ is the space of $2$-cocycles of weight $3$. Since  $d_{\mu}: E^{\ast}\oplus V^{\ast}\to \bigwedge ^2 (E^{\ast}\oplus V^{\ast})$ has image in $\bigwedge^2 E^{\ast}$ there are no 2-cochains of weight 3. Thus, the non-zero elements of $Z_{\mu}^{2,3}$ are in one-to-one correspondence with the space of cubically distorted central extensions of the Lie group $\widetilde{L}_{\mu}$ associated with the Lie algebra defined by $\mu$.

\begin{lemma}
\label{lemma:the-new-raghunathan}
    Let $m \geqslant 6$, $k \geqslant 4$. If \eqref{eq:standing-assumption-pair-n-k} holds and
    {\small
    \begin{equation}
    \label{eq:bounds-on-k}
        \frac{1}{m}\binom{m}{3} < k, 
    \end{equation}
    }%
    then every $\mu \in \mathcal L(m,k)$ has $Z_\mu^{2,3} \neq 0$. In particular, there is a cubically distorted central extension of $\widetilde{L}_{\mu}$.
\end{lemma}

\begin{proof}
    A dimension count gives that $\dim E^\ast \otimes V^\ast = mk$, while $\dim \bigwedge^3 E^\ast = \binom{m}{3}$.
    Thus if $m$ and $k$ are as in the assumptions , $d_\mu$ has nonzero kernel for all $\mu \in \mathcal L$.
\end{proof}

Note that the assumption of Lemma~\ref{lemma:the-new-raghunathan} is not void: let $p \geqslant 2$ be a parameter, set $m=3p$ and $k=3p^2$. 
Then as $p$ goes to infinity,
$m^2+k^2  \sim 9p^4$
is eventually smaller than $\binom{m}{2}k \sim \frac{27}{2}p^4$,
while
{\small
\[ \frac{1}{m}\binom{m}{3} \sim \frac{3p^2}{2} \; \text{ and }  k \sim 3p^2, \]}%
so that \eqref{eq:standing-assumption-pair-n-k} and  \eqref{eq:bounds-on-k} are satisfied by infinitely many pairs $(m,k)$; see also Remark \ref{rmk:minimal-dimension} below).

\begingroup
\renewcommand{\thetheorem}{\ref{thm:main-uncountable}}
\begin{theorem}
    There exists an uncountable family $\left\{(G_i,H_i)\right\}_{i\in I}$ of pairs of 4-nilpotent groups such that the following hold:
    \begin{enumerate}
        \item $H_i$ and $G_i$ have isomorphic associated Carnot graded group for every $i\in I$;
        \item $\delta_{G_i}(n)\asymp n^4 \prec n^5 \asymp \delta_{H_i}(n)$,  for every $i\in I$; and
        \item the associated Carnot graded groups of $H_i$ and $H_j$ for all $i,j\in I$ with $i\neq j$ are not isomorphic. In particular, $H_i$ and $H_j$ are not isomorphic.
    \end{enumerate}
\end{theorem}
\endgroup
\setcounter{theorem}{0}

\begin{proof}
    Let $(m,k)$ be as in the assumption of Lemma~\ref{lemma:the-new-raghunathan}.
    For every $\mu \in \mathcal L(m,k) \setminus \mathcal L_{\mathbf Q}(m,k)$ and $\omega \in Z^{2,3}_\mu \setminus \{0 \}$, denote $\widetilde L_{\mu,\omega}$ the simply connected nilpotent Lie group whose Lie algebra is the central extension of the nilpotent Lie algebra with law $\mu$ corresponding to the two-cocycle $\omega$.     
    Let 
    \[G_{\mu, \omega}  = L_5 \times_Z \widetilde L_{\mu, \omega}. \]
    Then $G_{\mu, \omega}$ has Dehn function $n^4$ by Theorem~{\ref{th:general-factor}}. The associated Carnot graded group of $G_{\mu, \omega}$, on the other hand, is isomorphic to the group 
    \[ L_5 \times \widetilde L_{\mu, \omega} / Z(\widetilde L_{\mu, \omega})\cong L_5\times \widetilde{L}_{\mu} \]
    which has Dehn function $n^5$. Moreover, $L_5\times \widetilde{L}_\mu \cong L_5\times \widetilde{L}_{\mu'}$ if and only if $\mu$ and $\mu'$ are in the same $G$-orbit in $\mathcal{L}(m,k)$.
    As $\mathcal L(m,k) \setminus \mathcal L_{\mathbf Q}(m,k)$ is the union of uncountably many $G$-orbits, this concludes the proof.
\end{proof}

\begin{remark}
    With the notation as in the proof, note that the Dehn function of $\widetilde{L}_{\mu}$ is cubic by Corollary \ref{cor:dehn-lower-lie}. The use of the latter is necessary when $\mu \notin \mathcal L_{\mathbf Q}(m,k)$, since the group $\widetilde{L}_{\mu}$ does not have lattices in this case by \cite{MalcevNilvarietes}.
\end{remark}
\begin{remark}\label{rmk:minimal-dimension}
    For a fixed choice of $(m,k)$ the groups $G_i$ and $H_i$ constructed in Theorem \ref{thm:main-uncountable} are all of dimension $5+m+k$. The smallest dimensional pair to which Theorem \ref{thm:main-uncountable} applies is $(m,k)=(6,4)$, meaning that we obtain a $15$-dimensional family satisfying its conclusion. 
\end{remark}

\subsection{An explicit 14-dimensional central product without lattices}
\label{subsec:central-product-wo-lattice}

As explained in Remark \ref{rmk:minimal-dimension}, the lowest-dimensional family satisfying Theorem \ref{thm:main-uncountable} that we know is 15-dimensional. Here we show that we can find an uncountable family of examples without lattices to which Theorem \ref{th:general-factor} applies already in dimension 14. However, these examples will all have the same Carnot graded group and thus not satisfy the conclusion of Theorem \ref{thm:main-uncountable}.

Here we let $\lambda \in \mathbf R \setminus \{0 \}$ and we consider the Lie algebra 
\[ \mathfrak g_\lambda  = \mathfrak l_8 \times_Z \mathfrak k_{7,\lambda} \]
where $\mathfrak k_{7,\lambda}$ is the filiform Lie algebra of dimension $7$ with brackets
{\small
\begin{align*}
    [Y_1,Y_i] & = Y_{i+1}, 1 \leqslant i \leqslant 6; \\
    [Y_2,Y_3] & = Y_5;  \\
    [Y_2,Y_4] & = Y_6; \\
    [Y_2,Y_5] & = \lambda Y_7; \\
    [Y_3,Y_4] & = (1-\lambda) Y_7.
\end{align*}
}%
Let $G_\lambda$ be the simply connected Lie group with Lie algebra $\mathfrak g_\lambda$. 

\begin{proposition}
\label{prop:G-lambda}
    There are uncountably many values of $\lambda$ for which $G_\lambda$ has no lattices.
\end{proposition}

\begin{proof}
In view of the Malcev correspondence, $G_\lambda$ has a lattice if and only if there is a basis of $\mathfrak g_{\lambda}$ in which the Lie brackets have rational coefficients \cite{MalcevNilvarietes}. 
We start by observing that $\mathfrak k_{7,\lambda}$ is a $7$-dimensional filiform ideal in $\mathfrak g_\lambda$ which is maximal among filiform ideals with respect to inclusion, and that it is unique with these properties.
    In addition, the Lie algebras $\mathfrak k_{7,\lambda}$ are pairwise non-isomorphic (see \cite[page 191]{Magnin}; there the complexification of $\mathfrak k_{7,\lambda}$ is called $\mathcal G_{7,1.1(i_\lambda)}$), and thus there are uncountably many. 
    It follows that the $\mathfrak g_\lambda$ are pairwise non-isomorphic, and there are uncountably many. Since there are only countably many isomorphism types of rational Lie algebras, uncountably many of the $\mathfrak{g}_{\lambda}$ have no rational form and thus no lattices.
    \end{proof}

    \begin{remark}
        A slight variant of the concluding argument can be obtained by invoking that torsion-free finitely generated nilpotent  groups are finitely presented, so that there are countably many isomorphism types of those, and each of them sits in a unique simply connected Lie group, by uniqueness of the real Malcev completion.
    \end{remark}

\begin{remark}
    In fact one can prove that $G_\lambda$ has lattices if and only if $\lambda$ is rational. The proof can be done using the same method as was used in \cite{HamrouniSouissi} to prove that $K_{7,\lambda}$ has lattices if and only if $\lambda$ is rational. We leave the details to the reader.
\end{remark}

\begin{appendix}
\renewcommand{\thesubsection}{\Alph{section}.\arabic{subsection}}

\section{Quasiisometry invariants of nilpotent groups}
\label{sec:QI}

Below we provide some context that motivates Corollary~\ref{cor:QI} and explain its proof. We then proceed to explain how in special cases it can also be deduced from existing results in the literature. Note that we currently don't know a proof based on results from the literature that covers all cases and does not use our main results; it would be interesting to know if there is such a proof.

\subsection{Context and motivation}\label{sec:Carnot-graded-and-Pansus-Theorem}
An outstanding problem in Geometric Group Theory is the classification of nilpotent  groups up to quasiisometry. In retrospect, this was raised by Gromov's Polynomial Growth Theorem \cite{GromovPolyGrowth}.
The most general version of the problem asks to classify simply connected nilpotent Lie groups up to quasiisometry, and it is conjectured that two such groups should be isomorphic if and only if they are quasiisometric \cite{CornulierQIHLC}.

Pansu's original contribution to this problem is the following:

\begin{theorem}[Pansu, {\cite{PansuCCqi}}]
    Let $G$ and $G'$ be simply connected nilpotent Lie groups. If $G$ and $G'$ are quasiisometric, then $\operatorname{gr}(G)$ and $\operatorname{gr}(G')$ are isomorphic.
\end{theorem}

Here, given a nilpotent Lie group $G$ with Lie algebra $\mathfrak{g}=Lie(G)$, $\operatorname{gr}(G)$ designates the nilpotent Lie group with Lie algebra 
\begin{equation}
\label{eq:carnotify}
    \operatorname{gr}(\operatorname{Lie}(G)):= \bigoplus_{i\geqslant 1} C^i 
    {\mathfrak g}/ C^{i+1}(\mathfrak g),
\end{equation} 
with brackets defined by the (well-defined) linear maps $$C^i \mathfrak g / C^{i+1} \mathfrak g \otimes C^j \mathfrak g / C^{j+1} \mathfrak g \to C^{i+j} \mathfrak g / C^{i+j+1} \mathfrak g$$ induced by the brackets on $\mathfrak{g}$ for all $i,j >0$; we call $\operatorname{gr}(G)$ the {\em Carnot graded group associated to} $G$.
One of the steps in the proof of Pansu's theorem consists of identifying the {\em asymptotic cone} of $G$ as $\operatorname{gr}(G)$ with a particular metric, and thus Pansu's theorem reduces the classification problem to that of discerning groups with isomorphic associated Carnot graded groups (or equivalently, bilipschitz homeomorphic asymptotic cones) up to quasiisometry.
Under this viewpoint, a special role is played by the {\em Carnot groups}, that is, those $G$ such that $G \cong \operatorname{gr}(G)$; the latter possess a special positive grading $(\mathfrak g_i)$ on their Lie algebra given by the direct sum decomposition \eqref{eq:carnotify}, which translates into a one-parameter group of automorphisms $(\sigma_t)_{t \in \mathbf R}$, $\sigma_t$ rescaling $\mathfrak g_i$ by $e^{ti}$ for all $i$.

Another approach to the quasiisometry classification was initiated by Shalom in \cite{ShalomHarmonic} using Gromov's dynamical characterization of quasiisometry.
Let us state below the current best result given by Shalom's approach, which was pushed by Sauer \cite{SauerHom} and more recently by Gotfredsen and Kyed \cite{GotfredsenKyed}. 
The key to the extension of Sauer's theorem to locally compact groups was provided by \cite{BaderRosendal}; see the third footnote in \cite[\S1F]{cornulier2017sublinear}.

\begin{theorem}[\cite{GotfredsenKyed}]
\label{th:shalom}
    Let $G$ and $G'$ be simply connected nilpotent Lie groups. If $G$ and $G'$ are quasiisometric, then the cohomology rings $H^\ast(G,\mathbf R)$ and $H^\ast(G',\mathbf R)$ are isomorphic.
\end{theorem}

It follows from Theorem~\ref{th:shalom} that, under the same assumptions on $G$ and $G'$,
    \begin{enumerate}
        \item The Betti numbers $b_i(G)$ and $b_i(G')$ are equal for all $i \geqslant 0$.
        \item \label{item:shalom} If $G$ and $G'$ have lattices $\Gamma$ and $\Gamma'$ respectively, then $b_i(\Gamma) = b_i(\Gamma')$ for all $i \geqslant 0$, where we use the fact, due to Nomizu, that $b_i(\Gamma) = b_i(G)$ for all $i \geqslant 0$ (and similarly for $\Gamma'< G'$) \cite{Nomizu1954Cohomology}. 
    \end{enumerate}

The conclusion of Shalom's original theorem is equivalent to the statement \eqref{item:shalom}. Sauer's result is Theorem~\ref{th:shalom} under the additional assumption that $G$ and $H$ have lattices, or equivalently, that they are real points of unipotent algebraic groups defined over $\mathbf Q$.

\begin{remark}
    By semicontinuity one always has $b_i(\operatorname{gr}(G)) \geqslant b_i(G)$ for all $i \geqslant 0$ (see e.g. \cite[\S 5]{Mumfordabelian}).
\end{remark}

\begin{remark}
    Shalom and Sauer's theorems also yield the quasiisometry invariance of certain cohomological dimensions under more general assumptions. These ideas were pursued by Li, who obtained the quasiisometry invariance of various cohomologies, however, for infinite-dimensional representations \cite{LiCohomology}. 
    One is faced with the following dilemma: the cohomology of nilpotent Lie groups with coefficients in finite-dimensional modules is computable and rich, but it is hard to prove it is a
    quasiisometry invariant (for instance, as remarked in \cite{CornulierQIHLC}, knowing it is invariant when coefficients are taken in the adjoint module would allow significant progress); whereas if one takes coefficients in certain infinite-dimensional modules, the invariance is somewhat more natural, but the cohomology groups or reduced cohomology groups turn out to vanish or be much harder to compute. For instance $\overline H^\ast (G, L^p(G,\mathbf R))$ vanishes for $G$ simply connected nilpotent and $p \notin \{ 1,\infty \}$, the key point being that powers of left translations by non-trivial elements in the centre are parallel to the identity, which contradicts the uniqueness of a cochain of minimising $L^p$ norm in any nontrivial cohomology class \cite[page 146]{AsInv}.
\end{remark}

\begin{remark}
Since the first version of this paper appeared, Antonelli and Young found and published a new proof of Theorem~\ref{th:shalom}. Their statement is more functorial, see~\cite[Theorem 1.1 and 1.3]{antonelli2024ergodicmapscohomologynilpotent}: starting from any quasiisometry, they build a so-called ergodic quasiisometry whose pullback induces the desired isomorphism of cohomology rings. Although they state the quasiisometry invariance of the complex cohomology rings, their proof in \cite{antonelli2024ergodicmapscohomologynilpotent} also provides it for the real cohomology rings.
\end{remark}

\subsection{Proof of Corollary~\ref{cor:QI} using Theorems \ref{th:general-factor}, \ref{th:model-filiform} and \ref{th:lowdim}}

Let us recall that Corollary~\ref{cor:QI} applies to groups $G$ obtained as central products of $K$ and $L$, where the pair $(K,L)$ should satisfy the assumption in one of the Theorems \ref{th:general-factor}, \ref{th:model-filiform} or \ref{th:lowdim}, and have different nilpotency classes $k$ and $\ell$.

\begin{proposition}\label{prop:preparing-cor-QI}
    Let $G$, $K$ and $L$ be as above. Then 
    \[ \delta_{G}(n) \asymp n^k \quad \text{whereas} \quad  \delta_{\operatorname{gr}(G)}(n) \asymp n^{k+1}.  \]
\end{proposition}
\begin{proof}
    The statement concerning $G$ is precisely the conclusion of Theorem \ref{th:general-factor}, \ref{th:model-filiform} or \ref{th:lowdim}. As for the statement concerning $\operatorname{gr}(G)$, we start by observing that since $K$ and $L$ have different nilpotency classes $k$ and $\ell$, and $k > \ell$,
    \begin{equation*}
        \operatorname{gr}(G) \cong \operatorname{gr}(K) \times \operatorname{gr}(L/Z(L)).
    \end{equation*}
    Now $K$ is always either $L_{k+1}$, $L_{k+1}^{\lrcorner}$, in which case $\operatorname{gr}(K)$ is always $L_{k+1}$, or $L_{5,5}$ in which case $\operatorname{gr}(K)$ is $L_4\times \R$. In every case, $\operatorname{gr}(K)$ has Dehn function $n^{k+1}$ (for instance, by \cite[Example 7.8]{lipt} and the Gersten--Holt--Riley upper bound). 
\end{proof}

Corollary~\ref{cor:QI} then directly follows from Proposition~\ref{prop:preparing-cor-QI} and the quasiisometry invariance of the Dehn function.

\subsection{Some particular cases of Corollary~\ref{cor:QI}}

We check that Corollary~\ref{cor:QI} follows from \cite{ShalomHarmonic} in the following cases:

\begin{enumerate}
    \item \label{item:model-fils-repeated}
    $G = G_{p,q} = L_{p} \times_Z L_q$ where $p>q \geqslant 3$.
    \item \label{item:reminders}
    \begin{enumerate}
        \item  \label{item:reminders-lowdim}
        $G = K \times_Z L$ where $k > \ell$ and $\max(\dim K, \dim L) \leqslant 5$.
        \item \label{item:reminders-heisenberg factors}
        $G$ is one of the groups in the two infinite families
    \[ J_{k,m} = L_{k+1} \times_Z H_{2m+1} \quad  k\geqslant 3, m \geqslant 2 \]
    and
    \[ J^{\lrcorner}_{k,m} = L^{\lrcorner}_{k+1} \times_Z H_{2m+1}, \quad k\geqslant 3, m \geqslant 2,\]
    where $H_{2m+1}$ denotes the real Heisenberg group of dimension $2m+1$.
    \end{enumerate}    
\end{enumerate}

These are not mutually exclusive: there is some overlap between \eqref{item:model-fils-repeated} and \eqref{item:reminders-lowdim} occuring when $p \leqslant 5$ in \eqref{item:model-fils-repeated} and between \eqref{item:reminders-lowdim} and \eqref{item:reminders-heisenberg factors} occuring when $m=2$ in \eqref{item:reminders-heisenberg factors}.

\subsubsection{Case when both factors are model filiform}
Here we prove Corollary \ref{cor:QI} in Case \eqref{item:model-fils-repeated} using Theorem \ref{th:shalom}.    It was computed in \cite[Lemma 7.12]{lipt} that 
    for $p > q \geqslant 3$,
    {\small \[ b_2(G_{p,q}) - b_2(L_p \times L_{q-1}) = \frac{3+(-1)^{p+1}}{2} > 0. \]}%
    It now follows from Theorem~\ref{th:shalom} 
    that $G_{p,q}$ and its associated Carnot graded group $L_p \times L_{q-1}$ cannot be quasiisometric.

\renewcommand{\arraystretch}{1.2}

\begin{table}[t]
    \centering
    \begin{tabular}{c|c|c|l|l}
         $G$ (non-Carnot) & $\delta_G$ & $\delta_{\operatorname{gr}(G)}$ & Betti numbers of $G$ & Betti numbers of $\operatorname{gr}(G)$  \\
         \hline
         $L_{5,5} \times_Z L_{3,2}$ & $n^3$ & $n^3$ & 1,5,10,11,11,10,5,1 &  1,5,{\bf 11},15,15,11,5,1 \\
         $L_{5,5} \times_Z L_{4,3}$ & $n^3$ & $n^3$ & 1,5,11,14,14,14,11,5,1 &  1,5,11,{\bf 15},16,15,11,5,1 \\
         $L_{5,5} \times_Z L_{5,5} $ & $n^3$ & $n^3$ & 1,6,16,25,26,26,25,16,6,1 & same as for $G$ \\
         $L_{5,5} \times_Z L_{5,4} $ & $n^3$ & $n^4$ & 
         1,7,21,34,33,33,34,21,7,1 & 1,7,21,34,{\bf 37},{37},34,21,7,1 \\
         $L_{5,7} \times_Z L_{3,2}$ & $n^4$ & $n^5$ & 1,4,6,9,9,6,4,1 & 1,4,{\bf 8},11,11,8,4,1 \\
         $L_{5,6} \times_Z L_{3,2}$ & $n^4$ & $n^5$ & 1,4,6,8,8,6,4,1 & 1,4,{\bf 8},11,11,8,4,1 \\
         $L_{5,7} \times_Z L_{4,3}$ & $n^4$ & $n^5$ & 1,4,7,9,10,9,7,4,1 & 1,4,{\bf 9},14,16,14,9,4,1 \\
         $L_{5,6} \times_Z L_{4,3}$ & $n^4$ & $n^5$ & 1,4,7,10,12,10,7,4,1 & 1,4,{\bf 9},14,16,14,9,4,1 \\
         $L_{5,6} \times_Z L_{5,5}$ & $n^4$ & $n^5$ & 1,5,11,16,21,21,16,11,5,1 &  1,5,11,{\bf 17},22,22,17,11,5,1 \\
         $L_{5,7} \times_Z L_{5,5}$  & $n^4$ & $n^5$ & 1,5,11,16,19,19,16,11,5,1 & 1,5,11,{\bf 17},22,22,17,11,5,1\\
\end{tabular}
~\vspace{.2cm}
    
    \caption{Betti numbers for the central products of low dimension. We emphasize in bold the first difference with the Betti numbers of the associated Carnot group. Note that $b_i(\operatorname{gr}(G)) \geqslant b_i(G)$ for all $i$.}
    \label{tab:betti}
\end{table}

\subsubsection{Low dimensions}
With the help of {\sc Sagemath}, we computed the Betti numbers for the Lie algebra cohomology with trivial coefficients of the non-Carnot graded central products of low-dimension.
The results are in Table~\ref{tab:betti}. We also indicate on the right the Betti numbers of the associated Carnot graded group.
In particular, this proves Corollary~\ref{cor:QI} in Case \eqref{item:reminders-lowdim} due to Theorem \ref{th:shalom}.

\subsubsection{The groups $J_{k,m}$}
Here we let $k \geqslant 3$ and $m \geqslant 2$.
Let $\mathfrak j_{k,m}$, resp.\ $\mathfrak j_{k,m}^{\lrcorner}$ be the Lie algebra of the group $J_{k,m}$, resp.\ $J_{k,m}^{\lrcorner}$ defined in Case \eqref{item:reminders-heisenberg factors}.
The Carnot graded group associated to $J_{k,m}$ and $J_{k,m}^{\lrcorner}$ is $L_{k+1} \times \mathbf R^{2m}$ in both cases.
We now compute the second Betti numbers of $J_{k,m}$ and $J_{k,m}^{\lrcorner}$ using Lie algebra cohomology and check that they differ from those of the associated Carnot graded group. We start with $\mathfrak j_{k,m}$ and describe the mild changes needed to treat $\mathfrak j_{k,m}^{\lrcorner}$.

A basis of $\bigwedge^1 \mathfrak j_{k,m}^\ast$ is spanned by
$\xi_1, \ldots, \xi_k, \theta_1, \ldots, \theta_{2m}, \zeta$
where
\begin{equation*}
\begin{cases}
    d \xi_i  = - \xi_1 \wedge \xi_{i-1} & 3 \leqslant i \leqslant k \\
    d \zeta  = - \xi_1 \wedge \xi_k - \omega
\end{cases}
\end{equation*}
where $\omega$ is the symplectic form 
\[ \omega = \sum_{j=1}^m \theta_j \wedge \theta_{j+m} \]
and all the exterior derivatives of the remaining one-forms from the basis vanish.
So 
\begin{align}
\label{eq:2-boundaries}
    B^2 (\mathfrak j_{k,m}) & = \operatorname{span}\left(\xi_1 \wedge \xi_2, \ldots, \xi_1\wedge \xi_{k-1}, \xi_1 \wedge \xi_k + \omega \right)
\end{align}
has dimension $k-1$.
On the other hand, computing the exterior derivatives of $2$-forms
we find that
{\small
\begin{align*}
    d(\xi_1 \wedge \xi_i) & = 0  & 2 \leqslant i \leqslant k \\
    d(\xi_2 \wedge \xi_i) & = - \xi_1 \wedge \xi_2 \wedge \xi_{i-1} & 3 \leqslant i \leqslant k \\
    d(\xi_i \wedge \xi_j) & = \xi_1 \wedge \xi_{i-1} \wedge \xi_j - \xi_1 \wedge \xi_i \wedge \xi_{j-1} & 3 \leqslant i \leqslant k-2, i+2\leqslant j \leqslant k \\
    d(\xi_i \wedge \xi_{i+1}) & = \xi_1 \wedge \xi_{i-1} \wedge \xi_{i+1} &  3 \leqslant i \leqslant k-1 \\
    d(\xi_i \wedge \theta_j) & = 0 & 1\leqslant i \leqslant 2, 1 \leqslant j \leqslant 2m \\
    d(\xi_i \wedge \theta_j) & = \theta_j \wedge \xi_1 \wedge \xi_{i-1}  & 3\leqslant i \leqslant k, 1 \leqslant j \leqslant 2m \\
    d(\xi_1 \wedge \zeta) & = \xi_1 \wedge \omega &  \\
    d(\xi_2 \wedge \zeta) & = \xi_2 \wedge \omega - \xi_1 \wedge \xi_2 \wedge \xi_k &  \\
    d(\xi_i \wedge \zeta) & = \xi_i \wedge \omega - \xi_1 \wedge \xi_{i-1} \wedge \zeta - \xi_1 \wedge \xi_i \wedge \xi_k & 3 \leqslant i \leqslant k-1 \\
    d(\xi_{k} \wedge \zeta) & = \xi_{k} \wedge \omega - \xi_1 \wedge \xi_{k-1} \wedge \zeta & \\
    d(\theta_j \wedge \zeta) & = \theta_j \wedge \xi_1 \wedge \xi_k + \sum_{1 \leqslant k \leqslant m, k \neq j, k+m \neq j} \theta_j \wedge \theta_k \wedge \theta_{k+m} & 1 \leqslant j \leqslant 2m \\
    d(\theta_i \wedge \theta_j ) & = 0 & 1 \leqslant i < j \leqslant 2m.
\end{align*}}%
For $l \geqslant 2$, define 
\begin{equation*}
    \nu_{2l} = \xi_{2,2l-1} - \xi_{3,2l-2} + \cdots +(-1)^{l+1}\xi_{l,l+1}.
\end{equation*}
Then, proceeding as in the lines dedicated to the obtention of \cite[Eq.(7.21)]{lipt},
\begin{equation*}
    Z^2 (\mathfrak j_{k,m}, \mathbf R) = B^2 (\mathfrak j_{k,m}, \mathbf R) \oplus \operatorname{span} 
    \begin{cases} 
        \nu_{2l} & 2 \leqslant l \leqslant \lfloor (k+1)/2 \rfloor \\ 
        \xi_i \wedge \theta_j & 1 \leqslant i \leqslant 2, 1 \leqslant j \leqslant 2m \\
        \theta_i \wedge \theta_j & 1 \leqslant i < j \leqslant 2m.
    \end{cases} 
\end{equation*}
Note that we did not include the cocycle $\xi_1 \wedge \xi_k$, since it is equivalent to $- \omega$ in cohomology, and as such, to a linear combination of the $\theta_i \wedge \theta_j$. 
Thus, 
\begin{equation}\label{eq:second-betti-jkm}
    b_2(\mathfrak j_{k,m}) = \left\lfloor \frac{k+1}{2} \right\rfloor -2 +1 + 4m+m(2m-1)  =  \left\lfloor \frac{k+1}{2} \right\rfloor - 1 + m(2m+3). 
\end{equation}
On the other hand, by the Künneth formula
{\small
\begin{align*}
    b_2(\mathfrak l_{2k+1} \times \mathbf R^{2m})
    & = b_2(\mathfrak l_{k+1}) + b_2(\mathbf R^{2m}) + b_1(\mathfrak l_{k+1}) \cdot b_1 (\mathbf R^{2m}) \\ 
    & = \left\lceil \frac{k+1}{2} \right\rceil  + \frac{2m(2m-1)}{2} + 2 \cdot 2m \\
    & = \left\lceil \frac{k+1}{2} \right\rceil  + m(2m+3).
\end{align*}}%
This finishes the proof that $b_2(\mathfrak j_{k,m}) \neq b_2(\mathfrak l_{k+1} \times \mathbf R^{2m})$.

We now describe the adaptation needed for $\mathfrak j_{k,m}^{\lrcorner}$. It is convenient to assimilate the linear spaces underlying $\mathfrak j_{k,m}$ and $\mathfrak j_{k,m}^\lrcorner$ simultaneously as a single vector space $V$ whose dual is spanned by $\xi_1, \ldots, \xi_k, \theta_1, \ldots, \theta_{2m}, \zeta$ while still denoting $(\bigwedge^\ast, d)$ and $(\bigwedge^\ast, d^{\lrcorner})$ the differential complexes computing the Lie algebra cohomologies in order to distinguish them from one another. The $d^\lrcorner$-differential of one-forms is now
\[\begin{cases}
    d^{\lrcorner} \xi_i  = - \xi_1 \wedge \xi_{i-1} & 3 \leqslant i \leqslant k \\
    d^{\lrcorner} \zeta  = - \xi_1 \wedge \xi_k - \xi_2 \wedge \xi_3 -\omega.
\end{cases}
\]
This change from $d\zeta$ to $d^\lrcorner \zeta$ creates some changes in some of the exterior derivatives of the $2$-forms in the basis involving $\zeta$. We mention all such derivatives below, even if some of them turn out to be the same as in the case of $\mathfrak j_{k,m}$.
{\small
\begin{align*}
    d^{\lrcorner} (\xi_1 \wedge \zeta) & = \xi_1 \wedge \omega + \xi_1 \wedge \xi_2 \wedge \xi_3 \\
    d^{\lrcorner} (\xi_i \wedge \zeta) & = \xi_i \wedge \omega - \xi_1 \wedge \xi_i \wedge \xi_k & 2 \leqslant i \leqslant 3\\
    d^{\lrcorner} (\xi_i \wedge \zeta) & = \xi_i \wedge \omega - \xi_1 \wedge \xi_{i-1} \wedge \zeta -  \xi_1 \wedge \xi_i \wedge \xi_k + \xi_1 \wedge \xi_2 \wedge \xi_3  & 4 \leqslant i \leqslant k-1 \\
    d^{\lrcorner} (\theta_j \wedge \zeta) & = \theta_j \wedge \xi_1 \wedge \xi_k + \theta_j \wedge \xi_2 \wedge \xi_3 + \sum_{1 \leqslant k \leqslant m, k \neq j, k+m \neq j} \theta_j \wedge \theta_k \wedge \theta_{k+m} & 1 \leqslant j \leqslant 2m.
\end{align*}
}%
Let $W$ be the subspace of $\bigwedge^2 V^\ast$ of $2$-forms having $\zeta$ as a factor, and let $R$ be the complementary subspace to $W$ spanned by two-forms in the basis not having $\zeta$ as a factor.
We have seen (by \eqref{eq:2-boundaries} and \eqref{eq:second-betti-jkm}) that $d\colon \bigwedge^2 V^\ast \to  \bigwedge^3 V^\ast $ has 
{\small
\begin{equation*}
    \operatorname{rank} d = \binom{k+1+2m}{2} - m(2m+3) - \left\lfloor \frac{k+1}{2} \right\rfloor +1 - k + 1 > k+1 +2m = \dim W.
\end{equation*}
}%
Now, $\operatorname{rank} d^{\lrcorner}$ is the largest size of a nonsingular minor of $d^\lrcorner$; since $d_{\mid R} = d^{\lrcorner}_{\mid R}$ and $d$ and $d^\lrcorner$ are injective on $W$ (by the above computations of derivatives of $2$-forms), we have that $\operatorname{rank} d^\lrcorner = \operatorname{rank} d$. Finally 
\begin{equation*}
    B^2 (\mathfrak j_{k,m}^\lrcorner) = \operatorname{span}\left(\xi_1 \wedge \xi_2, \ldots, \xi_1\wedge \xi_{k-1}, \xi_1 \wedge \xi_k + \omega + \xi_2 \wedge \xi_3 \right)
\end{equation*}
has the same dimension as $B^2 (\mathfrak j_{k,m})$. This completes the proof that 
\[ b_2(\mathfrak j^\lrcorner_{k,m}) = b_2(\mathfrak j_{k,m}). \]
The two corresponding Lie groups have the same associated Carnot group $L_{k+1} \times \mathbf R^{2m}$, so this finishes the proof.

\section{Gersten--Holt--Riley's Theorem for simply connected nilpotent Lie groups}

\label{sec:groups-without-lattices}\label{sec:GHR}
In this section we explain how Gersten--Holt--Riley's proof that a finitely generated nilpotent group of class $c$ has Dehn function bounded above by $n^{c+1}$ adapts to simply connected nilpotent Lie groups by using compact presentations. 
In \cite[{$5.A'_5$}]{AsInv} Gromov asserts that the Dehn function of a simply connected nilpotent group of nilpotency class $c$ is bounded above by $n^{c+1}$. This was proved by Gersten, Holt and Riley for finitely generated nilpotent groups using combinatorial arguments in \cite{GerstenRileyHolt} (see also \cite[\S 5.1]{Tessera-LSG-of-LCS-groups}). Their proof adapts to the situation of simply connected nilpotent groups $G$. When $G$ has a lattice, this is an immediate consequence of \cite[Theorem B]{GerstenRileyHolt}, as lattices in nilpotent Lie groups are uniform. In general, one can adapt Gersten--Holt--Riley's proof to compact presentations. Here we explain this by sketching their proof and explaining the main adaptations required. For details we refer to \cite{GerstenRileyHolt}.
\begin{theorem}\label{th:GHR-lie}
    Let $G$ be a simply connected nilpotent Lie group of nilpotency class $c$. Then $G$ admits an area filling length pair of the form $(n^{c+1},n)$.\footnote{We refer to \cite{GerstenRileyHolt} for the definition of an area filling length pair.}
\end{theorem}
\begin{proof}
    The proof is by induction on the nilpotency class $c$ of $G$. If $c=1$, then $G$ is finite-dimensional abelian and thus $(n^2,n)$ defines an area filling length pair for it.
    
    It suffices to give an upper bound on the Dehn function for a suitably chosen compact presentation of $G$. We start by choosing this presentation.
    
    The lower central series quotients $C^i(G)/C^{i+1}(G)$ can naturally be viewed as finite dimensional real vector spaces. In particular, we can choose finite subsets $\mathcal{A}_i\subset C^i(G)\setminus C^{i+1}(G)$ for $1\leqslant i \leqslant c$ such that the projection of $\mathcal{A}_i$ to $C^{i}(G)/C^{i+1}(G)$ identifies it with a basis. Let $\mathcal{A}:=\cup_{1\leqslant i\leqslant c} \mathcal{A}_i$. Then the set 
    $$\mathcal{X}:=\left\{a^r\mid a\in \mathcal{A},~ r\in\left[0,1\right]\right\}$$ 
    defines a compact generating set for $G$. We can then choose a compact presentation $\left\langle \mathcal{X}\mid \mathcal{R}\right\rangle$ such that $\mathcal{R}$ contains all $(c+1)$-fold commutators of elements of $\mathcal{X}^{\pm 1}$, all relations of the form $x^r\cdot x^s=x^{r+s}$ for $r,~s,~r+s\in\left[0,1\right]$, and all commutators $\left[x,z^r\right]$ with $x\in \mathcal{X}^{\pm 1}$, $z\in \mathcal{A}_c^{\pm 1}$ and $a\in\left[0,1\right]$. In particular, for any subset $\mathcal{Y}\subseteq \mathcal{A}$ the canonical presentation for the free nilpotent group of class $c$ generated by $\mathcal{Y}$ is a subpresentation of $\left\langle \mathcal{X}\mid \mathcal{R}\right\rangle$, allowing us to define finitely many families of compression words $u_{z_i}(q)$, ${q}\in\mathbb{N}$ for the elements of the finite subset $\mathcal{A}_c=\left\{z_1,\dots,z_k\right\}\subset \mathcal{A}$, which satisfy the properties of \cite[Corollary 3.3]{GerstenRileyHolt}. This enables us to prove the induction step using essentially the same arguments as in \cite{GerstenRileyHolt}, with only minor modifications, that we shall now explain.

    Consider the quotient $\overline{G}:=G/C^c(G)$ of $G$ by its centre and equip it with the presentation $\left\langle \overline{\mathcal{X}}\mid\overline{\mathcal{R}}\right\rangle$, where $\overline{\mathcal{X}}:=\mathcal{X}\setminus \mathcal{X}_c$ and $\overline{\mathcal{R}}$ is obtained from $\mathcal{R}$ by deleting all occurences of elements from $\mathcal{X}_c:=\left\{a^r\mid a\in \mathcal{A}_c,~ r\in[0,1]\right\}$ from the words in $\mathcal{R}$. 
    
    Let $w=w(\mathcal{X})$ be a null-homotopic word of length $n\in \mathbb{N}$ in $G$ and let $\overline{w}=\overline{w}(\overline{\mathcal{X}})$ be the word in $\overline{G}$, which is obtained by removing letters from $\mathcal{X}_c^{\pm 1}$ from $w$. By induction $\overline{G}$ has an area filling length pair of the form $(\overline{\lambda} n^c,\overline{\lambda} n)$ for some $\overline{\lambda}>0$. We can thus find a sequence of transformations $\overline{w}=\overline{w}_0,~\overline{w}_1,~ \dots,~ \overline{w}_{\overline{m}}=_{F_{\overline{\mathcal{X}}}} 1$ reducing $\overline{w}$ to the trivial word with the following properties:
    \begin{itemize}
        \item $\overline{w}_i$ is obtained from $\overline{w}_{i-1}$ by either introducing or removing a pair $\overline{x}\cdot \overline{x}^{-1}$, with $\overline{x}\in\overline{X}^{\pm 1}$, or a relation $r\in \overline{\mathcal{R}}^{\pm 1}$ somewhere in $\overline{w}_{i-1}$,
        \item $\overline{m}\leqslant \overline{\lambda} n^c$ and the words $\overline{w}_i$ have length $\leqslant \overline{\lambda} n$.
    \end{itemize}

    The induction step consists of using this sequence to produce a new sequence $w=v_0,~v_1,~ \dots,~v_m=_{F_{\mathcal{X}}} 1$ reducing $w$ to the trivial word by $m\leqslant \lambda n^{c+1}$ applications of elementary expansions, elementary reductions and relations such that the $v_i$ have length bounded by $\lambda n$, where $\lambda>0$ is a constant that only depends on our chosen presentation for $G$. It closely follows the induction step in \cite{GerstenRileyHolt} with some minor modifications. In the case of finitely generated nilpotent groups, Gersten--Holt--Riley account for the fact that the centre of $G$ may contain torsion elements (see discussion on p.808 of \cite{GerstenRileyHolt}, as well as Lemmas 3.4 and 3.5 of their work), this issue can not arise for simply connected nilpotent Lie groups. This simplifies the proof in our setting. However, we do instead have to account for the fact that central elements are represented by a product of real powers of element of $\mathcal{A}_c$, while the compression words $u_{z_i}(q)$ defined in \cite{GerstenRileyHolt} can only absorb integer powers of such elements.

    As in \cite{GerstenRileyHolt} the filling sequence $\left\{v_i\right\}$ for $w$ is constructed from the filling sequence for $\overline{w}$ in $\overline{m}$ steps, where at  step $i$ we use the transformation from $\overline{w}_{i-1}$ to $\overline{w}_i$ to build a sequence of transformations $w_{i-1}= v_{j_{i-1}},~ \dots, v_{j_i}=w_i$ between words $w_i$ of the form
    \[
        w_i= \left(z_k^{s_{ki}^-}u_{z_k}(q_{ki}^-)\right)^{-1} \cdots \left(z_1 ^{s_{1i}^-}u_{z_1}(q_{1i}^-)\right)^{-1} \overline{w}_i \left( z_1 ^{s_{1i}^+}u_{z_1}(q_{1i}^+)\right) \cdots \left( z_k^{s_{ki}^+}u_{k_1}(q_{ki}^+)\right),
    \]
    with $q_{ji}^{\pm}\in \mathbb{N}$ and $s_{ji}^{\pm}\in \left[0,1\right]$. 

    If $\overline{w}_{i}$ is obtained from $\overline{w}_{i-1}$ by an elementary expansion or reduction of a letter from $\overline{X}^{\pm 1}$, then $j_{i}=j_{i-1}+1$ and we obtain $w_i$ from $w_{i-1}$ by performing the same elementary expansion or reduction to the subword $\overline{w}_{i-1}$. 

    If $\overline{w}_{i}$ is obtained from $\overline{w}_{i-1}$ by inserting a relation $\overline{r}\in \overline{R}^{\pm 1}$ somewhere in $\overline{w}_{i-1}$, then we pass from $v_{j_{i-1}}$ to $v_{j_{i-1}+1}$ by inserting the corresponding relation $r\in R^{\pm 1}$. We then move all positive powers of elements of $\mathcal{A}_c$ appearing in $r$ to the right until they are adjacent to their corresponding compression word $u_{z_{j}}(q_{j(i-1)})$. Similarly we move all negative powers of elements of $\mathcal{A}_c$ appearing in $r$ to the left until they are adjacent to their corresponding compression word. The properties of the compression words ensure that the cost of this is $\lesssim n$, using the same arguments as in \cite{GerstenRileyHolt}. At the end of this process we obtain a word of the form
    {\small
    \[
        u_i= \left(z_k^{t_{ki}^-}u_{z_k}(q_{k(i-1)}^-)\right)^{-1} \cdots \left(z_1 ^{t_{1i}^-}u_{z_1}(q_{1(i-1)}^-)\right)^{-1} \overline{w}_i \left( z_1 ^{t_{1i}^+}u_{z_1}(q_{1(i-1)}^+)\right) \cdots \left( z_k^{t_{ki}^+}u_{k_1}(q_{k(i-1)}^+)\right),
    \]}%
    for some real numbers $t_{ji}^{\pm}$ which are bounded by a uniform constant that only depends on the maximal length of a word in $\mathcal{R}$. We now ``compress'' the integer parts of the $t_{ji}^{\pm}$ into the $u_{z_j}(q_{j(i-1)})$, creating new compression words $u_{z_j}(q_{ji}^{\pm})$ as in \cite{GerstenRileyHolt}, and define $s_{ji}^{\pm}:= t_{ji}^{\pm}-\lfloor t_{ji}^{\pm}\rfloor$. 

    The total cost of the transformations required for passing from $w_0$ to $w_{\overline{m}}$ can be estimated using the same arguments as in the proof of Theorem B in \cite{GerstenRileyHolt}; they rely on the fact that the total compression process has cost $\lesssim n^{c+1}$ and that the length of the compression words is $\lesssim n$ in all steps by \cite[Corollary 3.3]{GerstenRileyHolt}. In particular, there is a constant $\lambda>0$ that only depends on our chosen presentation for $G$ such that $m\leqslant \lambda n^{c+1}$ and the length of all words $v_i$ with $i\leqslant j_{\overline{m}}$ is $\leqslant \lambda n$.

    After completing the $\overline{m}$ steps of this transformation, we obtain the word
    \[
    w_{\overline{m}}=\left(z_k^{s_{k\overline{m}}^-}u_{z_k}(q_{k\overline{m}}^-)\right)^{-1} \cdots \left(z_1 ^{s_{1\overline{m}}^-}u_{z_1}(q_{1\overline{m}}^-)\right)^{-1}\left( z_1 ^{s_{1\overline{m}}^+}u_{z_1}(q_{1\overline{m}}^+)\right) \cdots \left( z_k^{s_{k\overline{m}}^+}u_{k_1}(q_{k\overline{m}}^+)\right).
    \]
    Since $w_{\overline{m}}$ is null-homotopic and every element $z_j^t$, with $t\in\left[0,\infty\right)$ is uniquely represented by an expression of the form  $z_j^su_{z_j}(q)$ with $s\in\left[0,1\right)$ and $q\in \mathbb{N}$, the word $w_{\overline{m}}$ freely reduces to the trivial word.

    This completes the proof of the induction step. We refer to \cite{GerstenRileyHolt} for further details.
\end{proof}

\end{appendix}

\bibliographystyle{alpha}
\bibliography{biblioNilpDehn}

\end{document}